\documentclass[11pt,a4paper]{article}

\usepackage{amsmath}     
\usepackage{dsfont}		
\usepackage{amsfonts}
\usepackage{amsthm}
\usepackage{graphicx}
\usepackage{amssymb}
\usepackage{framed}

\usepackage{latexsym}
\usepackage{euscript,makeidx,color,mathrsfs}
\usepackage{enumerate}

\usepackage[frak]{paper_diening}

\usepackage[colorlinks,linkcolor=blue,anchorcolor=green,citecolor=red]{hyperref}

\def\sqr#1#2{{\vcenter{\vbox{\hrule height.#2pt
              \hbox{\vrule width.#2pt height#1pt \kern#1pt \vrule width.#2pt}
              \hrule height.#2pt}}}}
\newcommand{\Div}{\divergence}
\newcommand{\ep}{\bfvarepsilon}
\newcommand{\R}{\mathbb R}

\newcommand{\dd}{\mathrm d}

\newcommand{\dt}{\, \mathrm{d}t}

\def\5n{\negthinspace \negthinspace \negthinspace \negthinspace \negthinspace }
\def\4n{\negthinspace \negthinspace \negthinspace \negthinspace }
\def\3n{\negthinspace \negthinspace \negthinspace }
\def\2n{\negthinspace \negthinspace }
\def\1n{\negthinspace }

\def\={\buildrel \triangle \over =}

\def\i{\infty}
\def\bal{\begin{aligned}}
\def\eal{\end{aligned}}

\def\({\Big (}
\def\){\Big )}
\def\[{\Big[}
\def\]{\Big]}

\def\bde{\begin{definition}\label}
\def\ede{\end{definition}}
\def\be{\begin{equation}}
\def\bel{\begin{equation}\label}
\def\ee{\end{equation}}
\def\beq{\begin{equation*}\begin{aligned}}
\def\eeq{\end{aligned}\end{equation*}}
\def\bt{\begin{theorem}\label}
\def\et{\end{theorem}}
\def\bc{\begin{corollary}\label}
\def\ec{\end{corollary}}
\def\bl{\begin{lemma}\label}
\def\el{\end{lemma}}
\def\bp{\begin{proposition}\label}
\def\ep{\end{proposition}}
\def\bas{\begin{assumption}\label}
\def\eas{\end{assumption}}
\def\br{\begin{remark}\label}
\def\er{\end{remark}}
\def\bex{\begin{example}\label}
\def\ex{\end{example}}
\def\ba{\begin{array}}
\def\ea{\end{array}}
\def\ed{\end{document}}

\def\square#1{\vbox{\hrule\hbox{\vrule height#1%
     \kern#1\vrule}\hrule}}
\def\rectangle#1#2{\vbox{\hrule\hbox{\vrule height#1%
     \kern#2\vrule}\hrule}}

\font\tenbb=msbm10 \font\sevenbb=msbm7 \font\fivebb=msbm5

\newfam\bbfam
\scriptscriptfont\bbfam=\fivebb \textfont\bbfam=\tenbb
\scriptfont\bbfam=\sevenbb

\newtheorem{theorem}{\hskip 1.3em Theorem}[section]
\newtheorem{definition}[theorem]{\hskip 1.3em Definition}
\newtheorem{proposition}[theorem]{\hskip 1.3em Proposition}
\newtheorem{corollary}[theorem]{\hskip 1.3em Corollary}
\newtheorem{lemma}[theorem]{\hskip 1.3em Lemma}
\newtheorem{remark}[theorem]{\hskip 1.3em Remark}
\newtheorem{example}[theorem]{\hskip 1.3em Example}
\newtheorem{algorithm}[theorem]{\hskip 1.3em Algorithm}

\newtheorem{assumption}[theorem]{\hskip 1.3em Assumption}

\usepackage[backend=biber,maxbibnames=5,sorting=nyt,maxalphanames=5,style=alphabetic,bibencoding=utf8,giveninits,url=false,isbn=false]{biblatex}
\AtBeginBibliography{\scriptsize}

\usepackage{tikz}
\usepackage{pgfplots}
\pgfplotsset{
  width=.65\linewidth,
  axis background/.style={fill=black!5!white},
  grid style={densely dotted,semithick},
  compat=newest 
}
 
\definecolor{plot_color1}{HTML}{d7191c}
\definecolor{plot_color2}{HTML}{fdae61}
\definecolor{plot_color3}{HTML}{c2a5cf}
\definecolor{plot_color4}{HTML}{abd9e9}
\definecolor{plot_color5}{HTML}{2c7bb6}

\pgfplotscreateplotcyclelist{MyColors}{%
    {plot_color1,line width=1.2},
    {dotted,plot_color2,line width=1.2},
    {dotted,plot_color3,line width=1.2},
	{dash dot,plot_color4,line width=1.2},
    {dash dot,plot_color5,line width=1.2},}
    
\usepackage{color}
\usepackage[normalem]{ulem}
\usepackage{tabularray}
\usepackage{lscape}
\definecolor{Silver}{rgb}{0.752,0.752,0.752}

\usepackage{booktabs}

\newcommand{\jw}[1]{\color{blue}Joern: #1\color{black}}

\makeatletter
   
   \@addtoreset{equation}{section}
\makeatother

\usepackage{subcaption}

\begin{document}

\title{Numerical analysis of the stochastic Navier-Stokes equations}

\author{Dominic Breit\thanks{Mathematisches Institut,
TU Clausthal, Erzstrasse 1,
D-38678 Clausthal-Zellerfeld, Germany.
 {\small\it
e-mail:} {\small\tt dominic.breit@tu-clausthal.de}} \qquad
Andreas Prohl\thanks{
Mathematisches Institut, Universit\"at T\"ubingen, Auf der Morgenstelle 10,
D-72076 T\"ubingen, Germany.
{\small\it
e-mail:} {\small\tt prohl@na.uni-tuebingen.de}}\qquad
J\"orn Wichmann\thanks{School of Mathematics, Monash University, 9 Rainforest Walk, Victoria 3800, Australia.
 {\small\it
e-mail:} {\small\tt joern.wichmann@math.uni-bielefeld.de}}
}

\date{\today}
\maketitle

\begin{abstract}
The developments over the last five decades concerning numerical discretisations
of the
incompressible Navier--Stokes equations have lead to reliable tools for
their approximation: those include stable methods to properly address the
incompressibility constraint, stable discretisations to account for 
convection dominated problems, efficient time (splitting) methods, and
methods to tackle their nonlinear character. While these tools may successfully be applied to reliably simulate even more
complex fluid flow PDE models, their understanding requires a fundamental revision in the case
of stochastic fluid models, which are gaining increased  importance nowadays. 

This work motivates and surveys optimally convergent numerical methods for the stochastic Stokes and Navier--Stokes equations that were obtained in the last decades. Furtheremore, we computationally illustrate the failure of some of
those methods from the deterministic setting, if they are straight-forwardly applied to the stochastic case. In fact, we explain why some of these deterministic methods perform sub-optimally by highlighting crucial analytical differences between the deterministic and stochastic equations --- and how modifications of the deterministic methods restore their optimal performance if they properly address the probabilistic nature of the stochastic problem. 

Next to the numerical analysis of schemes, we propose a general benchmark of prototypic
fluid flow problems driven by different types of noise to also compare new algorithms by simulations in terms of complexities, efficiencies, and possible limitations. The driving motivation is to reach a better comparison of simulations for new schemes in terms of accuracy and complexities, and to also complement theoretical performance studies for restricted settings of data by more realistic ones.
\end{abstract}

\tableofcontents

\section{Introduction} \label{sec:intro}

\subsection{Why noise and what type of noise is used?}
Stochastic components are commonly added to the balance laws from fluid mechanics to account for external uncertainties and
thermodynamical fluctuations. 
Moreover, they are used to model turbulent effects in fluid flows.
In turbulence theory one often splits the velocity field in two components: the resolved large-scale and slow-varying component and the unresolved small-scale and fast-varying component. For example, in atmospheric flows these small-scale perturbations are caused by the irregularities of hills and mountain profiles; see \cite[Section 2]{FP1}. Hereby, small scales describe the wind fluctuations at space distances of one to ten meters and large scales at space distances of ten to thousand  kilo-meters. As far as the time-scales are concerned, to change them (\emph{e.g.}, from seconds, over minutes to hours) corresponds to increasing the number of observations, respectively dividing the units on the time axis by the number of observations (zooming out). For example, small-scale fluctuations of the wind pattern are observed at a time scale of one second. They seem random without any pattern. At an intermediate scale of one minute they start to possess additional structure; rather than behaving arbitrarily they look like a random walk. Finally, a white noise pattern (\emph{e.g.}, the action of a Wiener process) can be seen on the large scale of one hour.

The mathematical description of turbulence is by far non-trivial. Starting from the seminal works of Kolmogorov~\cite{Ko1}, turbulence has been studied from a statistical perspective: while one snapshot (a particular time instance) of a fluid might look very chaotic and substantially different from another snapshot, an ensemble of snapshots can be used to derive stable characteristics of the turbulent flow.

It is motivated by the statistical approach that \textit{noise} (a random variable that depends on intrinsic fluid properties and/or external, environmental stimuli) has been introduced in the Navier--Stokes equations to capture the statistics of turbulent flows mathematically. But a universal noise that accurately represents all scales simultaneously has not been found yet; instead, depending on the scale different noises have been proposed, each encapsulating different aspects.

\begin{figure*}[t!]
    \centering
    \begin{subfigure}[t]{0.25\textwidth}
        \centering
        \includegraphics[width=1.0\textwidth]{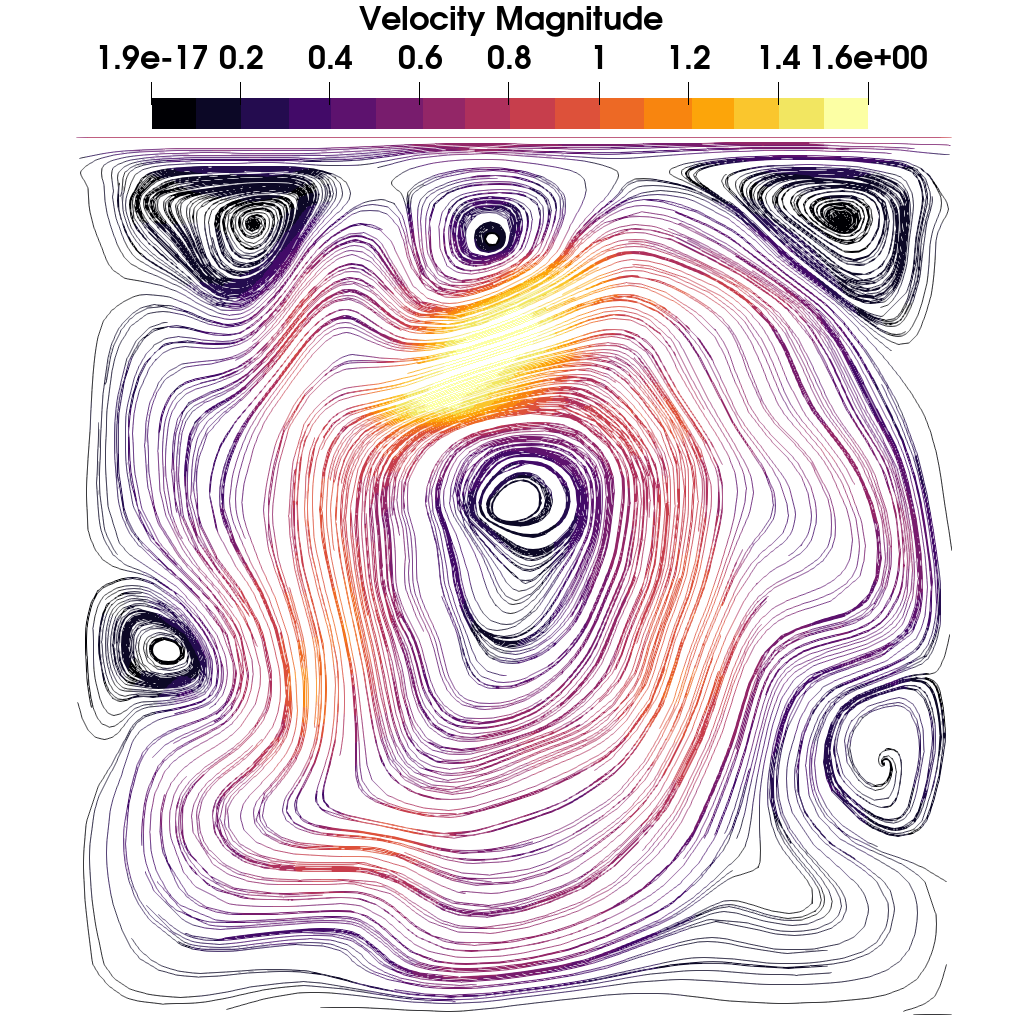}
    \end{subfigure}%
    ~ 
    \begin{subfigure}[t]{0.25\textwidth}
        \centering
        \includegraphics[width=1.0\textwidth]{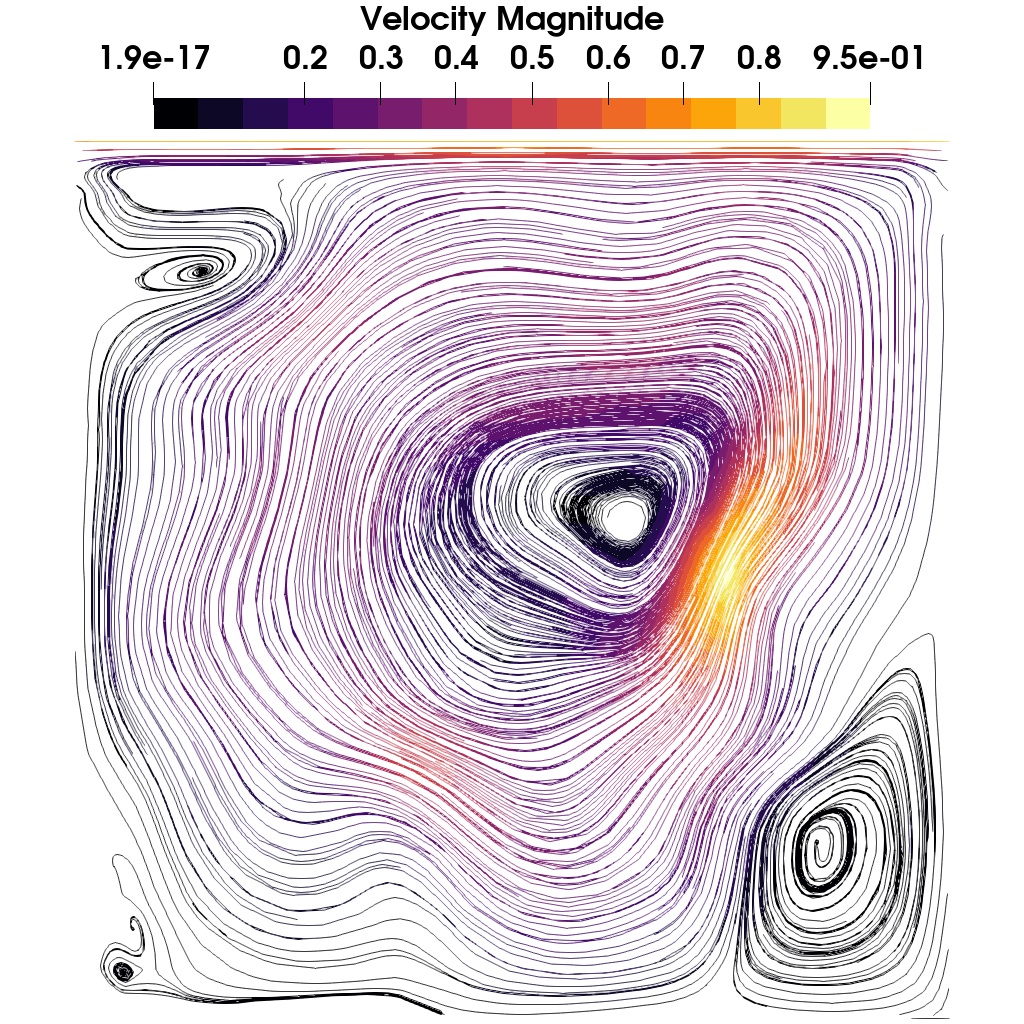}
    \end{subfigure}%
    \begin{subfigure}[t]{0.25\textwidth}
        \centering
        \includegraphics[width=1.0\textwidth]{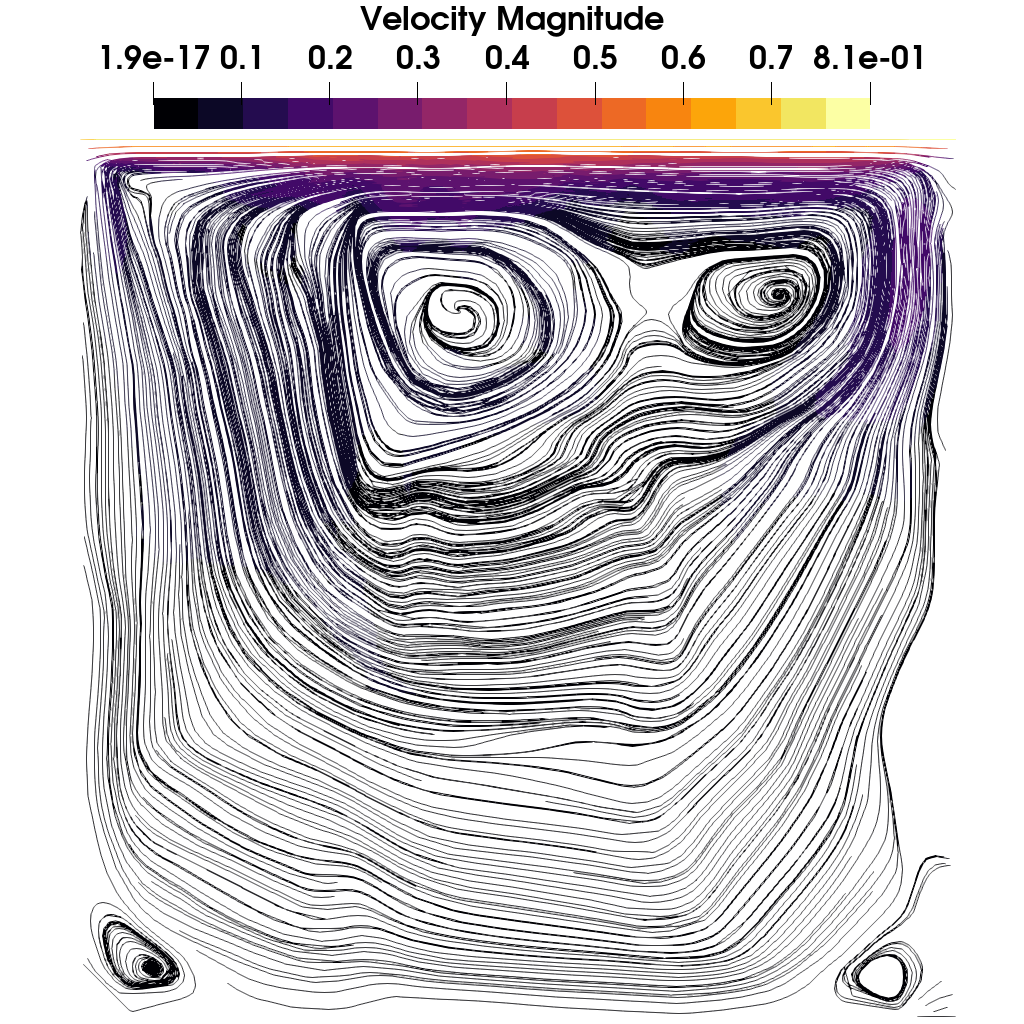}
        \end{subfigure}%
        \begin{subfigure}[t]{0.25\textwidth}
        \centering
        \includegraphics[width=1.0\textwidth]{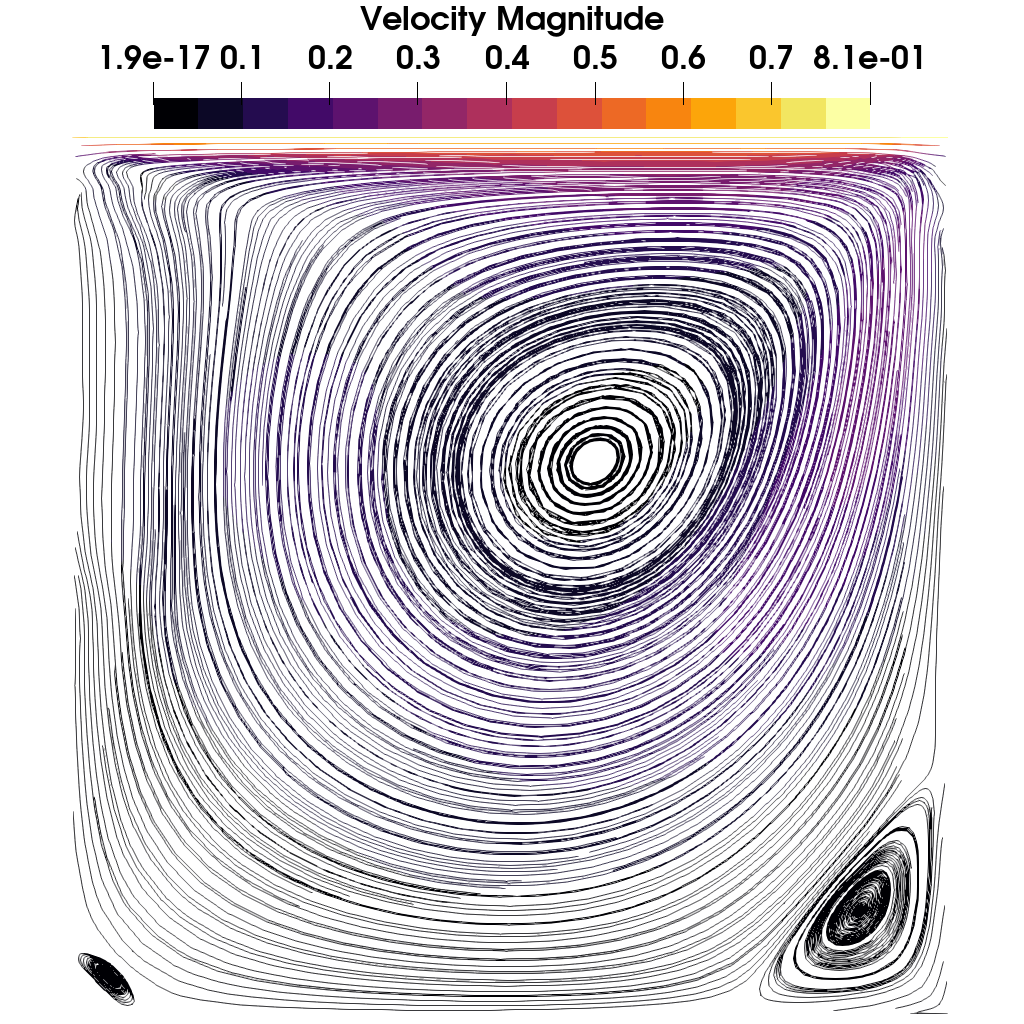}
    \end{subfigure}
    \caption{Snapshots at a common time of realizations of the velocity field streamlines for the lid-driven cavity experiment (Example~\ref{ex:lid-driven}) forced by (from left to right) additive noise, multiplicative noise, transport noise, and no noise.} 
    \label{fig:compareVelocity}
\end{figure*}

Next, we will motivate three commonly-used noises: \textit{additive noise}, \textit{transport noise}, and \textit{multiplicative noise}, which initiate different dynamics as indicated in the snapshots of realizations in Figure~\ref{fig:compareVelocity}. For this, we denote the large-scale and small-scale velocity fields by~${\bf u}$ and~${\bfeta}$, respectively. Relating their dynamics to each other is now a modelling choice.

\subsubsection{Additive noise}
\begin{framed}
We call any noise that is independent of intrinsic properties of the fluid, such as its velocity field and pressure,  \textit{additive noise}.
\begin{example} \label{ex:AdditiveNoise-intro}
A common additive noise is given by:
\begin{equation} \tag{additive noise}
\sum_k \int \bfsigma_k \,\dd W_k,
\end{equation}
where $\{W_k\}_k$ are independent real-valued Wiener processes, and $\{\bfsigma_k\}_k$ are space-time dependent vector fields. We do not specify the index set for $k$ to shorten the notation; but typically $k \in \mathbb{N}$.
\end{example}
\end{framed}
Additive noise is typically used to model external stimuli that influence the dynamics of fluids randomly; \emph{e.g.}, these stimuli can be induced by hills and mountains for atmospheric flows, as described earlier, or by thermal fluctuations at substantially smaller length scales.

 In \cite{BoEc} a model for 2D turbulence is suggested which fits well with statistically observed data: they describe $\bfeta$ (the small-scale velocity field) and $\vartheta$ (the corresponding small-scale pressure) by a Navier--Stokes system with friction and additive stochastic forcing with Gaussian distribution, as given in the example above. When appropriately scaled with an intensity parameter\footnote{For atmospheric flows as described at the beginning of this section, a particular value for the intensity parameter was given in~\cite[Section 2]{FP1}, namely~$\varepsilon=\frac{1}{60}$ hertz.}~$\varepsilon$, they propose that the small-scale velocity evolves according to:
\begin{subequations}\label{eq:NSsmall}
\begin{eqnarray}
\partial_t \bfeta+\big((\bfeta+\bfu)\cdot\nabla \big)\bfeta &=& \mu\Delta\bfeta-\nabla \vartheta-\varepsilon^{-2}\bfeta+\varepsilon^{-2}\sum_k\bfsigma_k\dot{W}_k ,\\
\Div \bfeta &=& 0, 
\end{eqnarray}
\end{subequations}
where $\dot{W}_k$ denotes the distributional time derivative of $W_k$ and the vector fields $\{\bfsigma_k\}_k$ are additionally assumed to be solenoidal. 

Additive noise also appears in~\cite{Bi}, where the Kolmogorov--Obukhov theory of statistical turbulence, proposed in~\cite{Ko1,Ko2} and~\cite{Ob}, is derived from the stochastic Navier--Stokes equations. Note that the suggested noise is in general not solenoidal.

\subsubsection{Transport noise}\label{sec:1.1.2}
\begin{framed}
For a given a velocity field~$\bfv$, we refer to the following noise as \textit{transport noise}: 
\begin{equation} \tag{transport noise}
\sum_k \int (\bfsigma_k\cdot\nabla)\bfv \circ \dd W_k,
\end{equation}
where $\{W_k\}_k$ and $\{\bfsigma_k\}_k$ are as in Example~\ref{ex:AdditiveNoise-intro}. The symbol~`$\circ$' indicates Stratonovich integration. 
\end{framed}
Especially amongst physicists, Stratonovich integration is favoured over It\^{o} integration since it satisfies the classical chain rule. As we will see later, this will have remarkable consequences on important quantities of interest, including pathwise conservation of the \textit{kinetic energy} of the system.    

In what follows, we present two reasons for considering transport noise in the context of fluids: firstly, it arises naturally as the scaling limit of appropriately coupled small-scale and large-scale dynamics; and secondly, it emerges from geometric mechanics, when unresolved dynamics are incorporated rather than neglected. 

\textbf{Scaling limit.}
The following scaling limit was rigorously studied in~\cite{FP1,FP2}. There, they authors assumed that the large-scale velocity and the corresponding pressure~$\pi$ are described by the Navier--Stokes equations, where the small-scale velocity from \eqref{eq:NSsmall} affects the advection of the large-scale velocity, \emph{i.e.},
\begin{subequations} \label{eq:NSlarge}
\begin{eqnarray}
\partial_t \bfu +\big( (\bfu + \bfeta)\cdot\nabla\big) \bfu &=& \mu\Delta \bfu-\nabla \pi  , \\ 
{\rm div} \, {\bf u} &=& 0. 
\end{eqnarray}
\end{subequations}
The small-scale velocity from \eqref{eq:NSsmall}, which is stimulated by additive noise, is the only source of randomness for the large-scale dynamics. Notice that, due to the dependence of the small-scale velocity on the intensity parameter~$\varepsilon$, the large-scale velocity and pressure depend on it, too. 

By arguing that friction and noise become increasingly important for the small-scale dynamics from \eqref{eq:NSsmall}, the limit $\varepsilon \to 0$ 
of the coupled system is an important matter. In~\cite{FP1,FP2}, this limit was analysed and they showed that the asymptotic large-scale evolution is given by the Navier--Stokes equations forced by transport noise:
\begin{subequations} \label{eq:SNS}
\begin{eqnarray} \label{eq:SNS-a}
\partial_t\bfu+(\bfu\cdot\nabla)\bfu &=& \mu\Delta\bfu-\nabla \pi+\sum_k(\bfsigma_k\cdot\nabla)\bfu\circ\dot{W}_k,\\
\Div \bfu &=& 0.
\end{eqnarray}
\end{subequations}
Consequently, to consider transport noise in the Navier--Stokes equations accounts for the impact of infinitely strong friction and additive noise on the small-scale dynamics. A similar result for 3D flows was proved in~\cite{DP}. In this setting, however, the validity of the model for the small-scale dynamics is less clear.
We also remark that transport noise already appeared in the Kraichnan model of turbulence~\cite{Kr1,Kr2}.

Notice that, even in the asymptotic case, the small-scale dynamics does not affect the balance of the large-scale energy. Therefore, the same energy law, as known for the deterministic Navier--Stokes equations holds in these cases too: 
\begin{equation} \label{eq:energy-preservation}
\partial_t {\mathscr E}(\bfu) = -\mu \int \abs{\nabla \bfu}^2 \dd x,
\end{equation}
where ${\mathscr E}(\bfu) := \tfrac{1}{2} \int \abs{\bfu}^2 \dd x$ denotes the \textit{kinetic energy}. This energy law is a consequence of the Stratonovich integration used for defining the transport noise.

\textbf{Geometric mechanics.}
A geometric mechanics perspective was taken in~\cite{HOLM2,HOLM,HOLM1}. In order to include stochastic parametrizations of unresolved
dynamics, the evolution of the flow map (\emph{i.e.}, the map that traces the position of the fluid particles) is given by\footnote{In \cite{HOLM} the model is slightly different and $\bfu(\bfq,t)$ on the right-hand side is replaced by $\bfu(x,t)$, such that $\bfq$ is not the stochastic flow.}
\begin{equation} \label{eq:2706}
\partial_t\bfq=\bfu(\bfq,t)+\sum_{k}\bfzeta_k(\bfq)\circ \dot{W}_k,
\end{equation}
where~$\bfq$ denotes the position, $\bfu$ is the (time-regular) velocity and $\{\bfzeta_k\}_k$ are time-dependent vector fields, encoding the time-irregular unresolved dynamics; see also~\cite{MiRo} and~\cite{ACC}. While the mathematical discussion of this approach does not require a particular choice of these vector fields, in applications they must be specified; \emph{e.g.}, in~\cite{Co1,Co2} they are computed numerically.

 Now the velocity field is advected by the right-hand side of \eqref{eq:2706}, and two additional terms:
\begin{equation} \nonumber
\sum_k(\bfzeta_k\cdot\nabla)\bfu\circ\dot{W}_k \qquad \text{ and } \qquad \bigg(\nabla\sum_k\bfzeta_k\circ\dot{W}_k\bigg)^\top \bfu,
\end{equation}
complement the material derivative in the equations of motion. This leads to the equations
\begin{subequations} \label{eq:SNS2}
\begin{eqnarray}  \label{eq:SNS2-a}
\partial_t\bfu+(\bfu\cdot\nabla)\bfu &=& \mu\Delta\bfu-\nabla \pi-\sum_k(\bfzeta_k\cdot\nabla)\bfu\circ\dot{W}_k-\big(\nabla\sum_k\bfzeta_k\circ\dot{W}_k\big)^\top \bfu,\\
\Div \bfu &=& 0.
\end{eqnarray}
\end{subequations}
The first additional term is transport noise and corresponds to the transportation of the fluid in random directions. The second term is linear multiplicative noise of Stratonovich-type and encodes the stretching of the fluid. 

\subsubsection{Multiplicative noise}

\begin{framed}
Given a velocity field~$\bfv$, we refer to the following noise as \textit{multiplicative noise}: 
\begin{equation} \tag{multiplicative noise} \label{def:multi-Noise}
\sum_k \int \bfsigma_k(\bfv) \dd W_k,
\end{equation}
where $\{W_k\}_k$ are as in Example~\ref{ex:AdditiveNoise-intro}, and $\{\bfsigma_k\}_k$ are vector fields which depend on temporal and spatial variables, as well as the velocity field.
\end{framed}
Sometimes operator-valued (instead of vector field-valued) coefficients~$\{\bfsigma_k\}_k$ and Hilbert space-valued (instead of real-valued) random variables~$\{W_k\}_k$ are considered; more details can be found, \emph{e.g.}, in~\cite[Chapter~2]{Liu2015}. But we neglect these general assumptions to simplify the notation. 

Multiplicative noise generalizes additive noise by allowing the coefficients to additionally depend on the velocity field. As we have seen in~\eqref{eq:SNS2-a} (even though Stratonovich integration was used), this additional dependence is needed, \emph{e.g.}, to account for random stretching effects. In this context, assuming that transport effects are negligible (\emph{i.e.}, $\abs{\bfzeta_k} \approx 0$ but $\abs{\nabla \bfzeta_k} \approx 1$), the evolution in its It\^o formulation is governed by:  
\begin{subequations} \label{eq:SNS3}
\begin{eqnarray}  \label{eq:SNS3-a}
\partial_t\bfu+(\bfu\cdot\nabla)\bfu &=& \mu\Delta\bfu-\nabla \pi + \bfLambda \bfu -\sum_k \bfsigma_k(\bfu)  \dot{W}_k,\\
\Div \bfu &=& 0,
\end{eqnarray}
\end{subequations}
with $\bfsigma_k(\bfu) = (\nabla \bfzeta_k)^\top \bfu$ and $\bfLambda = \frac{1}{2} \sum_k \big(\nabla \bfzeta_k \big)^\top \big(\nabla\bfzeta_k\big)^\top$. As a consequence, the Navier--Stokes equations driven by multiplicative noise is the right model.  
 
\subsection{Phenomenological differences of these noises} 
So far, we have seen that all three noises -- ranging from externally induced additive noise to intrinsically motivated transport and multiplicative noises -- are used to encapsulate different aspects of unresolved and typically irregular scales of turbulent flow. But how do different noises trigger different dynamics, and how do they affect hydrodynamic stabilities which {\em e.g.} manifest by the formation and interaction of vortexes? Of course, it is not only the type of the noises here which determines different dynamics; but also the spatial structure of the noise -- being {\em e.g.} rough or smooth -- exerts a crucial impact in this regard. 

To motivate the different quantitative impact of each noise, we here present various comparisons of the dynamics for the \textit{lid-driven cavity} experiment ({\em i.e.}, Example~\ref{ex:lid-driven} but for the Navier--Stokes equations). For this, we use the noises that are given in Example~\ref{ex:vortexes}; these are generated by a common spatial shape function and they only differ in how the randomness interacts with it. 

In the following two subsections, we report on different dynamics for single realizations as well as on statistical differences which we observe in simulations conducted within the same numerical setting (\emph{i.e.}, parameters, meshes in space and time, etc.). 

\subsubsection{Pathwise differences}
Even though one particular realization of the dynamics is usually not representing the full spectrum of all possible evolution dynamics, a comparison of realisations between different noises provides a first indication for their quantitative differences; \emph{e.g.}, in the lid-driven cavity experiment (\emph{cf.} Figure~\ref{fig:compareVelocity}), additive noise supports the formation of multiple small vortices, while multiplicative noise accelerates the large central vortex; and transport noise reduces the action of the boundary conditions on the interior velocity, leading to a slower velocity in comparison with the deterministic case.

\subsubsection{Statistical differences} 
Comparing the statistics of the velocity fields (\emph{i.e.}, their distribution) instead of a single realization eliminates the chance of selecting a \textit{rare event} -- an outlier that differs substantially form other realisations and that does not show the typical behaviour. However, a visual representation of these statistics is infeasible, since the distribution is a probability measure on an infinite-dimensional space. Thus, only finite dimensional sub-statistics can be illustrated, such as \textit{statistical quantities}, \textit{marginal distributions} or the action of the distribution against an \textit{observable}. Each sub-statistics focuses on different aspects of the dynamics, which we will discuss next.

\textbf{Statistical quantities} (\emph{e.g.}, the mean-value, median, 
quantiles or standard deviation) aggregate the data to characterise the shape of the distribution. For example, for distributions whose standard deviation (\emph{cf.} Figure~\ref{fig:compareVelocity_SD}) is small, the mean-value (\emph{cf.} Figure~\ref{fig:compareVelocity_mean}) is a good representation for the behaviour of all possible trajectories at the same time. This aggregation eliminates information about the individual events, but preserves spatial information (\emph{i.e.}, it is still a function in space).

In the lid-driven cavity experiment, we observe in Figure~\ref{fig:compareVelocity_mean} that the expected velocity fields display different numbers, locations, and geometric details of arising vortices if compared to the pathwise simulations in Figure~\ref{fig:compareVelocity}.

\begin{figure*}[t!]
    \centering
    \begin{subfigure}[t]{0.33\textwidth}
        \centering
        \includegraphics[width=1.0\textwidth]{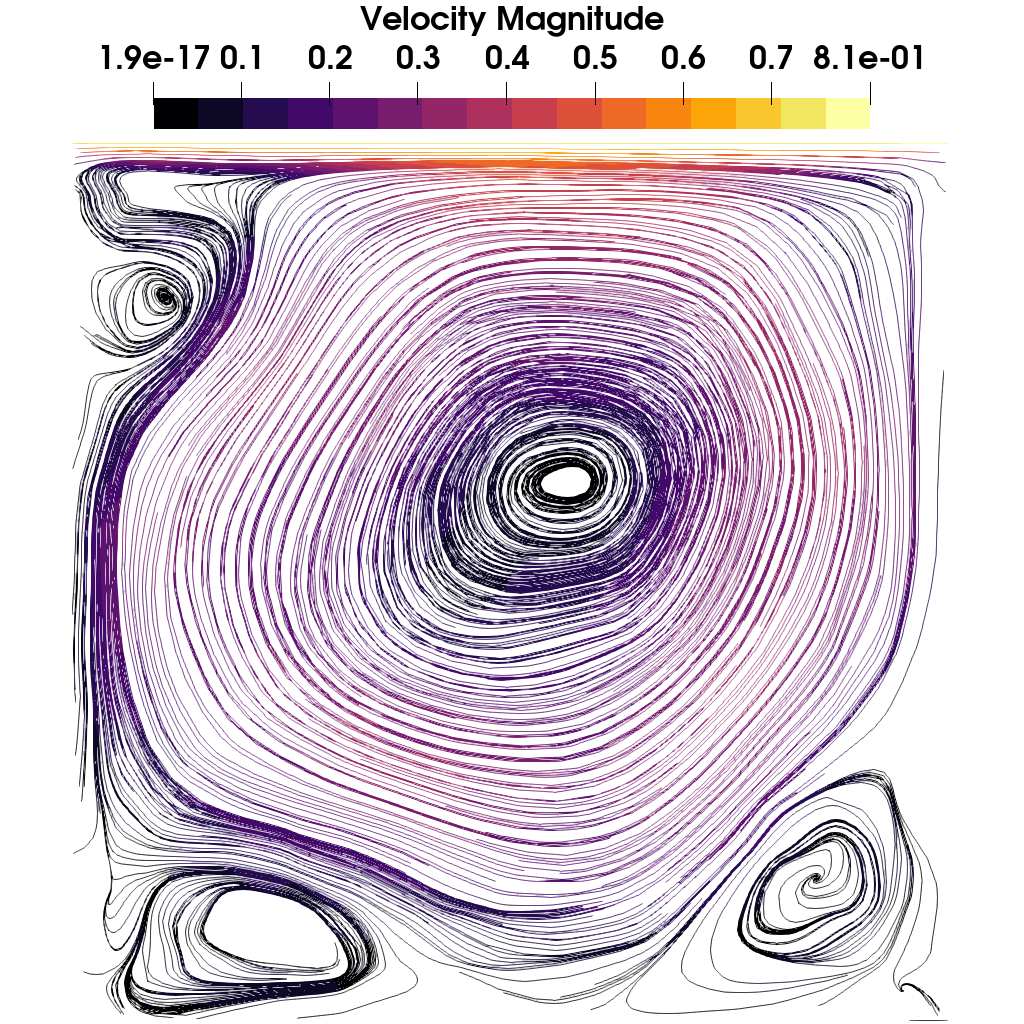}
    \end{subfigure}%
    ~ 
    \begin{subfigure}[t]{0.33\textwidth}
        \centering
        \includegraphics[width=1.0\textwidth]{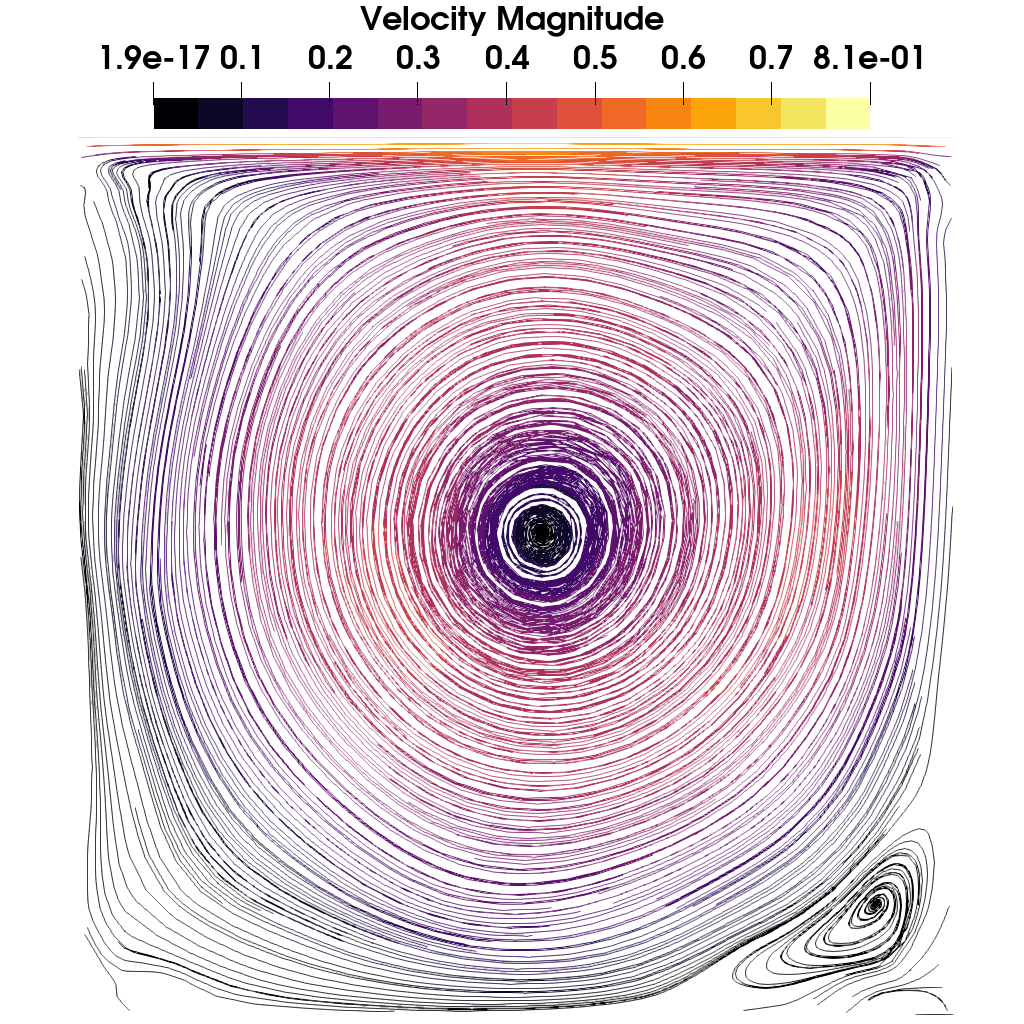}
    \end{subfigure}%
    \begin{subfigure}[t]{0.33\textwidth}
        \centering
        \includegraphics[width=1.0\textwidth]{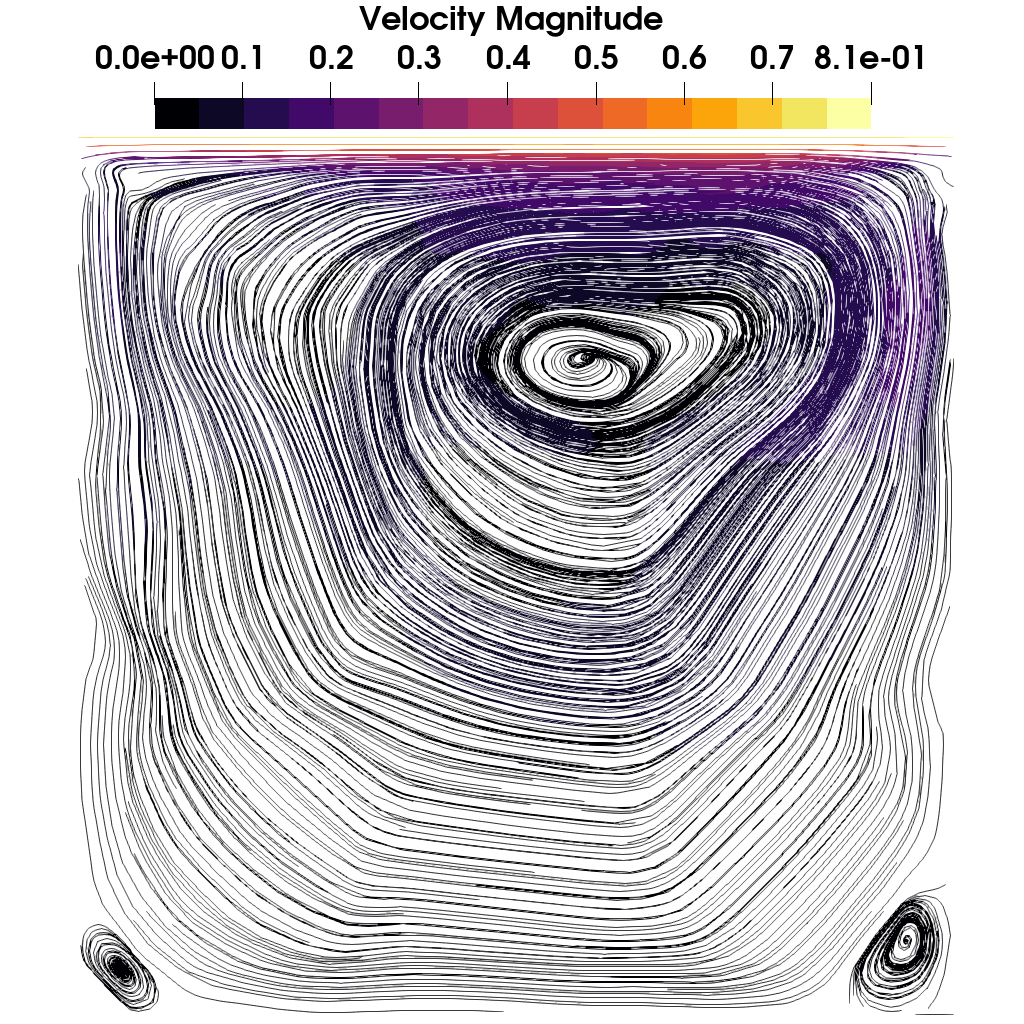}
    \end{subfigure}
    \caption{Snapshots at a common time of the expected (the average over all realizations) velocity field streamlines for the lid-driven cavity experiment (Example~\ref{ex:lid-driven}) forced by additive noise (left), multiplicative noise (middle) and transport noise (right).} 
    \label{fig:compareVelocity_mean}
\end{figure*}

\begin{figure*}[t!]
    \centering
    \begin{subfigure}[t]{0.33\textwidth}
        \centering
        \includegraphics[width=1.0\textwidth]{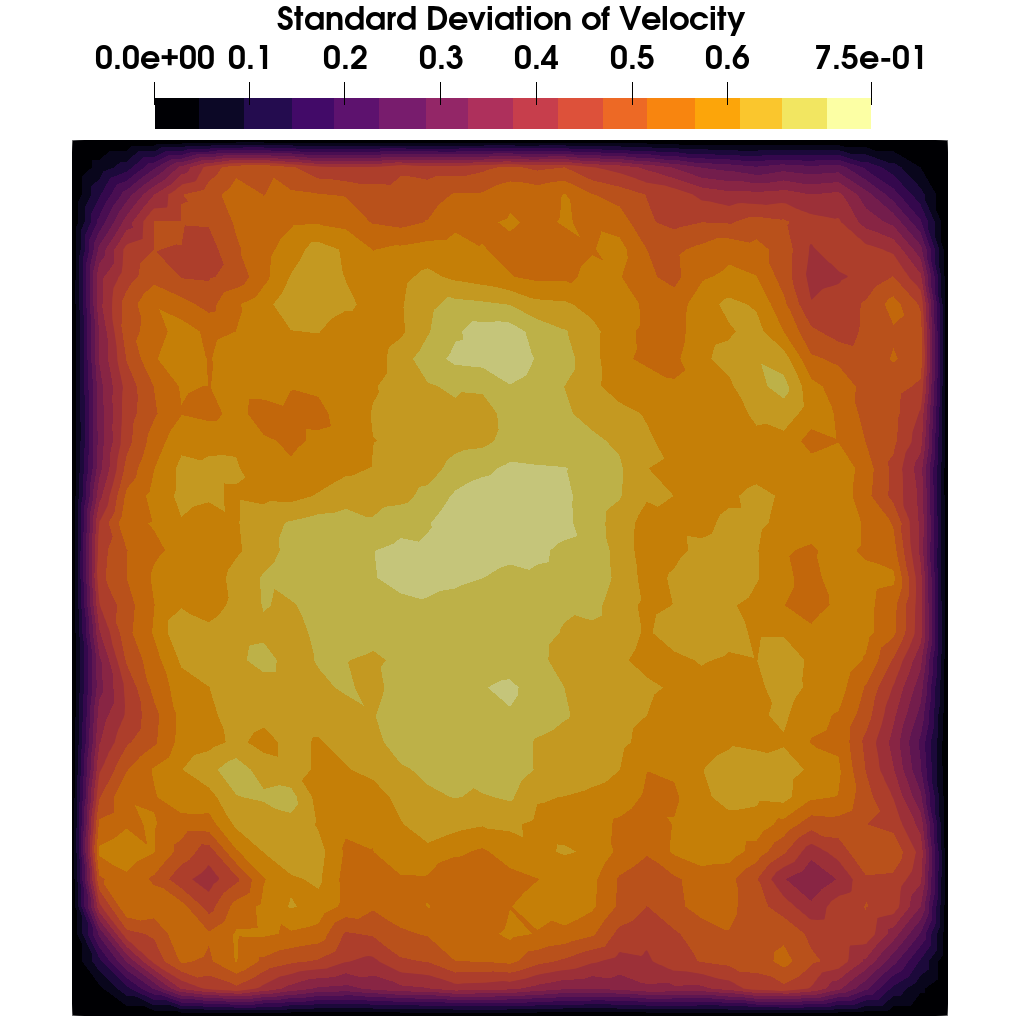}
    \end{subfigure}%
    ~ 
    \begin{subfigure}[t]{0.33\textwidth}
        \centering
        \includegraphics[width=1.0\textwidth]{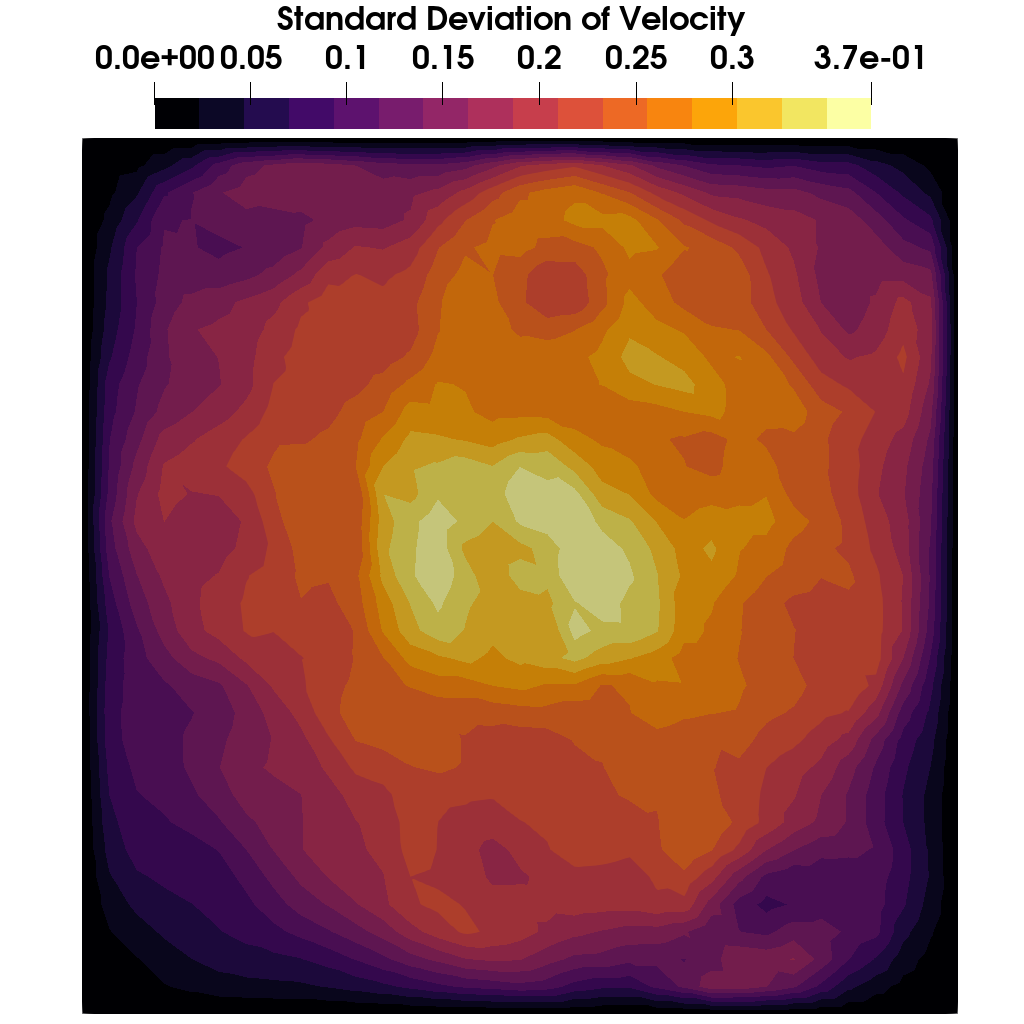}
    \end{subfigure}%
    \begin{subfigure}[t]{0.33\textwidth}
        \centering
        \includegraphics[width=1.0\textwidth]{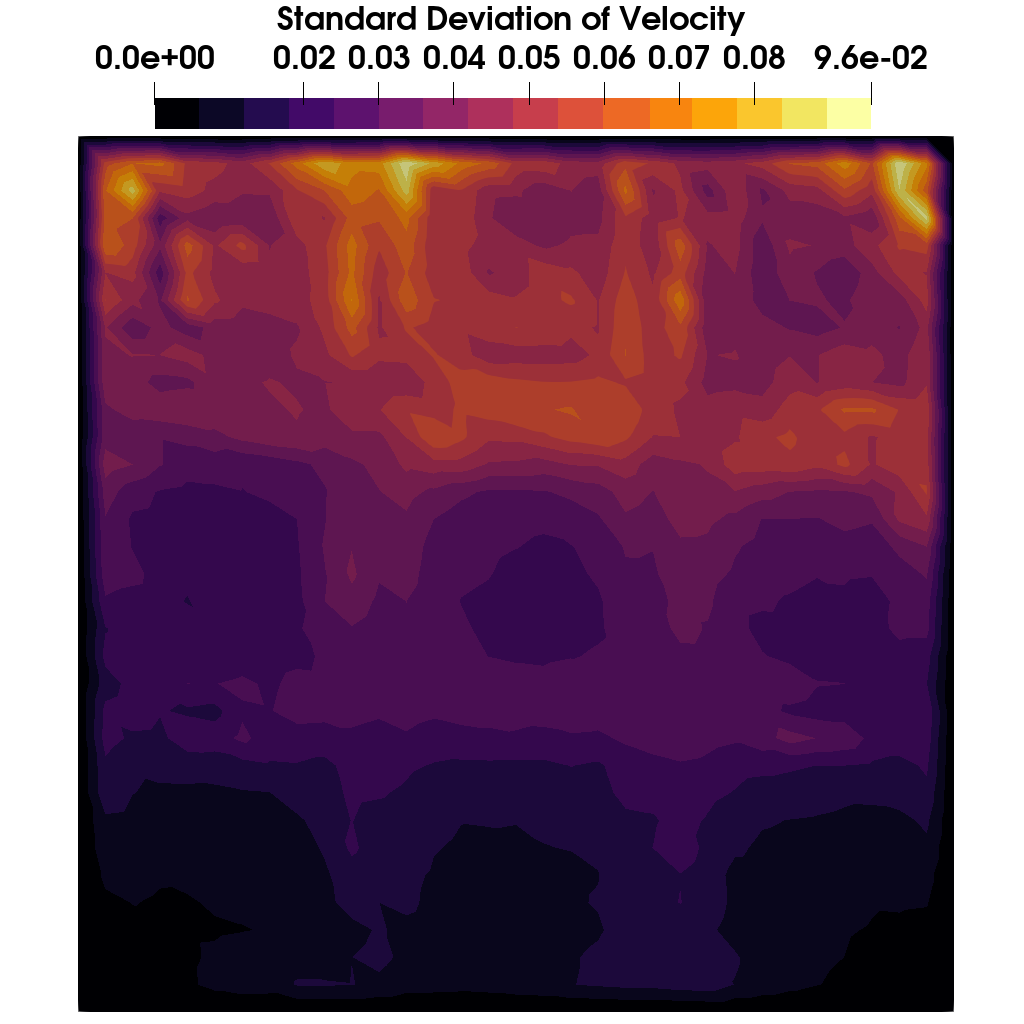}
    \end{subfigure}
    \caption{Snapshots at a common time of the standard deviation of the velocity field for the lid-driven cavity experiment (Example~\ref{ex:lid-driven}) forced by additive noise (left), multiplicative noise (middle) and transport noise (right). Colour encodes the magnitude of the standard deviation of velocity.} 
    \label{fig:compareVelocity_SD}
\end{figure*}


\textbf{Marginal distributions} (\emph{i.e.}, the distribution at a particular location) reduce the dimension of the problem by fixing the location. It can be considered as placing a probe into the fluid at this particular location. Consequently, the distribution of a $2$-dimensional random vector can be plotted easily (\emph{cf.} Figure~\ref{fig:velocity_marginals}). An advantage of this approach is that the full distribution is visualized, but only at one spatial point. 

\begin{figure*}[t!]
    \centering
    \begin{subfigure}[t]{0.33\textwidth}
        \centering
        \includegraphics[width=1.0\textwidth]{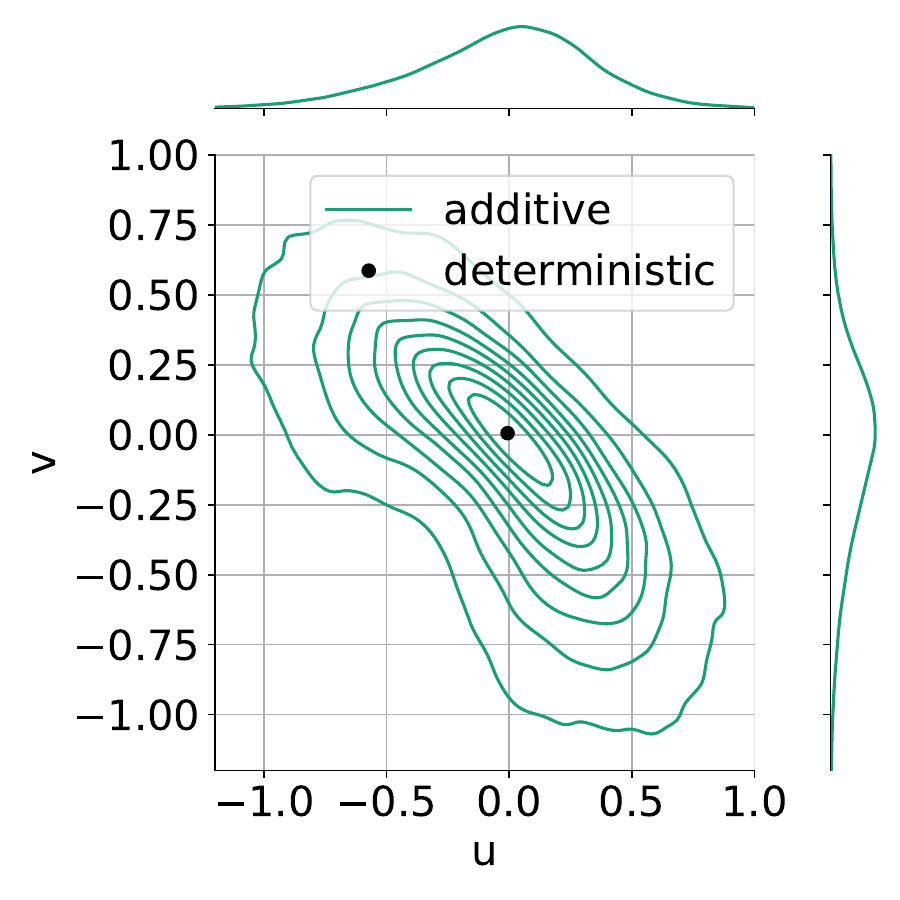}
    \end{subfigure}%
    ~ 
    \begin{subfigure}[t]{0.33\textwidth}
        \centering
        \includegraphics[width=1.0\textwidth]{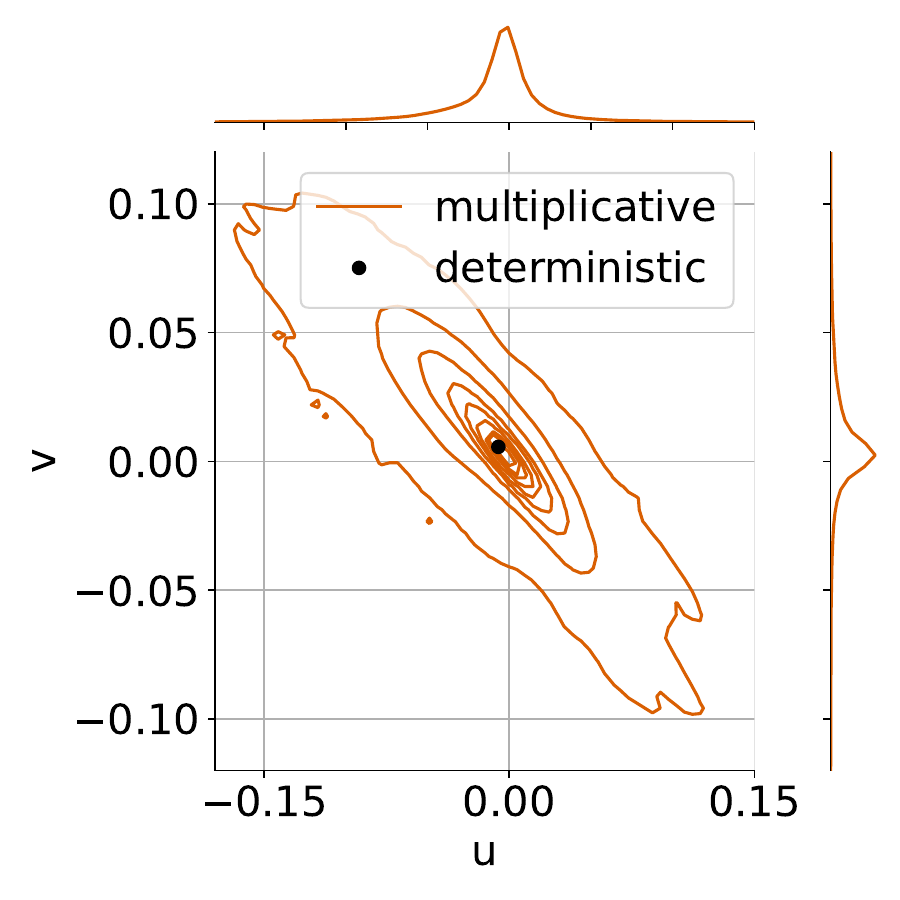}
    \end{subfigure}%
    \begin{subfigure}[t]{0.33\textwidth}
        \centering
        \includegraphics[width=1.0\textwidth]{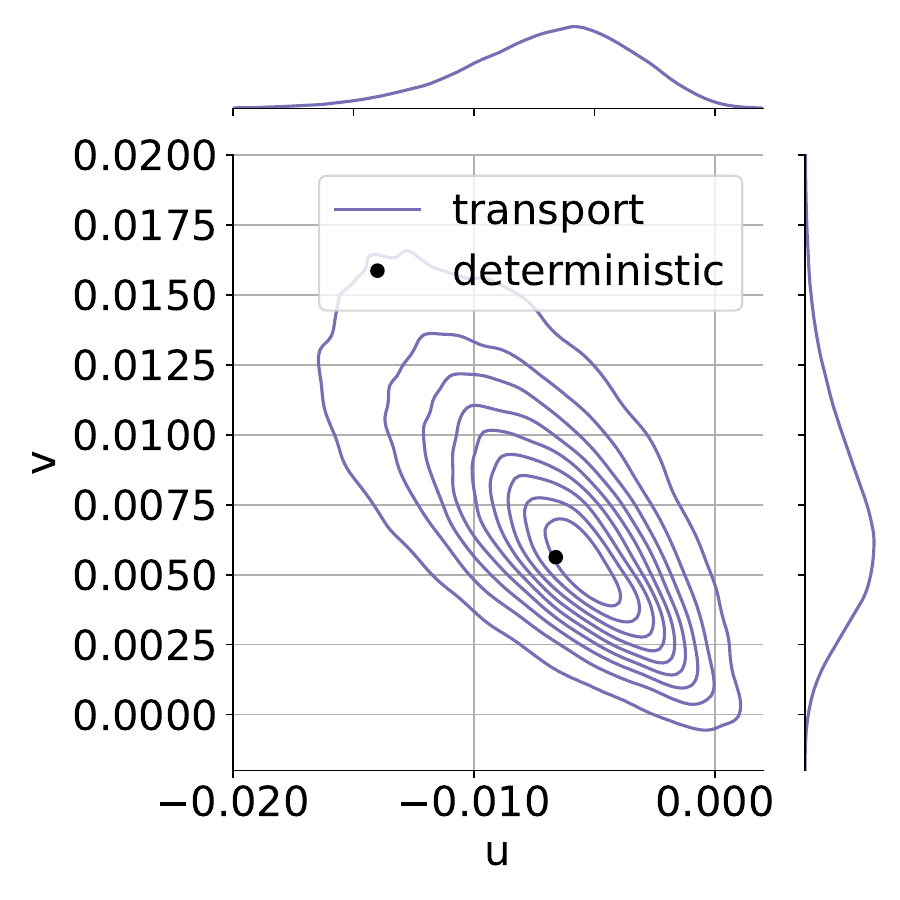}
    \end{subfigure}
    \caption{Probability density function (visualised by level sets) of the 2D marginal distribution at $(0.875,0.875)$ and a common time of the velocity field~$\bfu = (u, v)^\top$ for the lid-driven cavity experiment (Example~\ref{ex:lid-driven}) forced by additive noise (left), multiplicative noise (middle) and transport noise (right). Black dots indicate the deterministic velocity vector. The 1D probability density functions of each velocity component are displayed at the top and right, respectively. 
    } 
    \label{fig:velocity_marginals}
\end{figure*} 

\textbf{The action against observables} is a generalization of the statistical quantities to include physical quantities of interest (\emph{e.g.}, the kinetic energy). In general, one specifies a function (the observable). Then, one evaluates the function at the (random) velocity field and computes the expected value. Choosing a particular observable, such as the kinetic energy, enables the physical interpretation of these results (\emph{cf.} Figure~\ref{fig:evolution_energy}): \emph{e.g.}, in the lid-driven cavity experiment, additive noise and multiplicative noise are accelerating the velocity field, while transport noise is slowing it down: in the additive case the velocity increases rapidly until it reaches its steady state (approximately $20$-times higher than the deterministic level), while for multiplicative noise it accumulates slowly, and for transport noise its steady state, which is below the deterministic level, is reached the fastest. 

\begin{figure*}[t!]
    \centering
    \begin{subfigure}[t]{0.5\textwidth}
        \centering
        \includegraphics[width=1.0\textwidth]{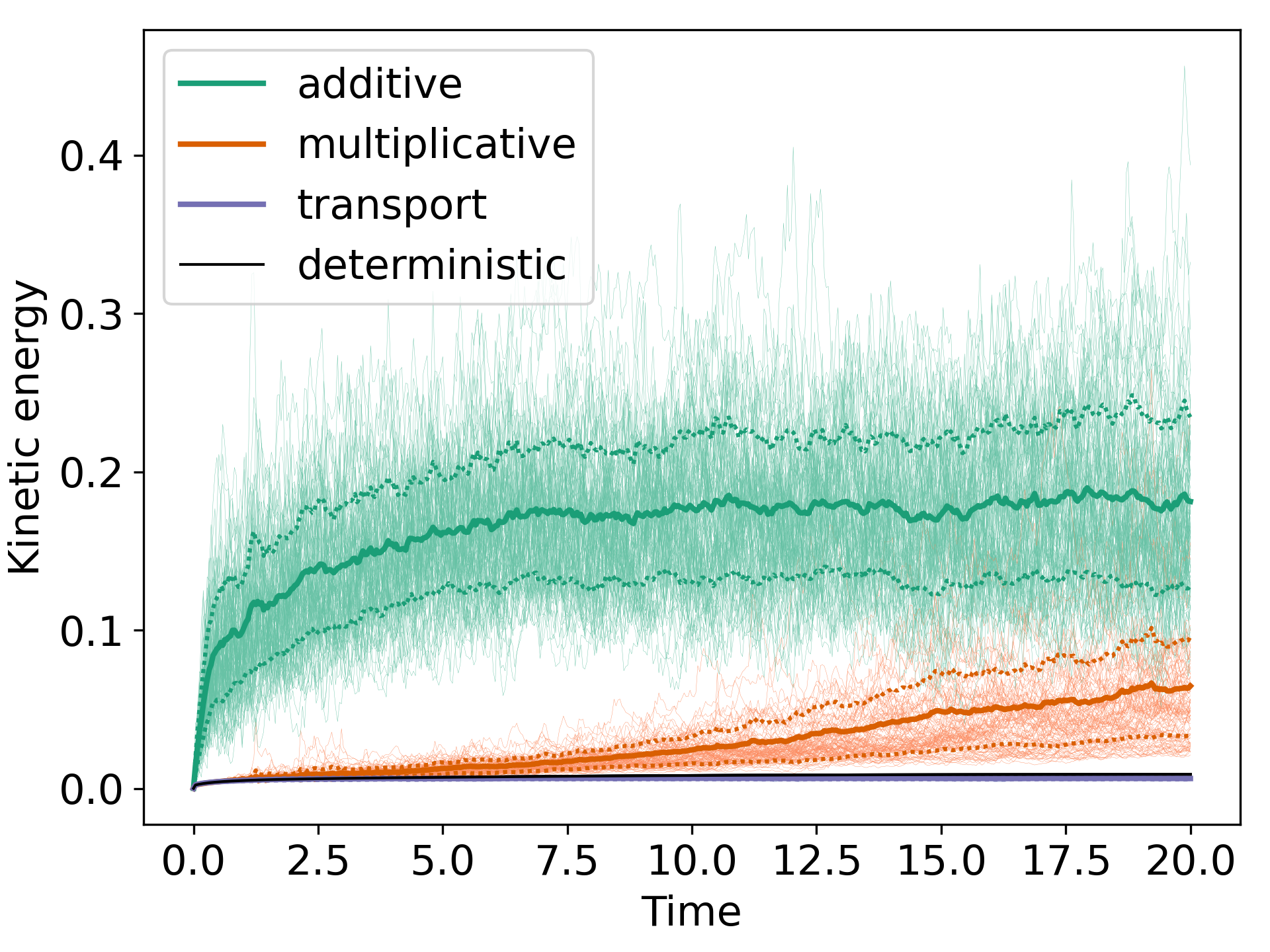}
    \end{subfigure}%
    \begin{subfigure}[t]{0.5\textwidth}
        \centering
        \includegraphics[width=1.0\textwidth]{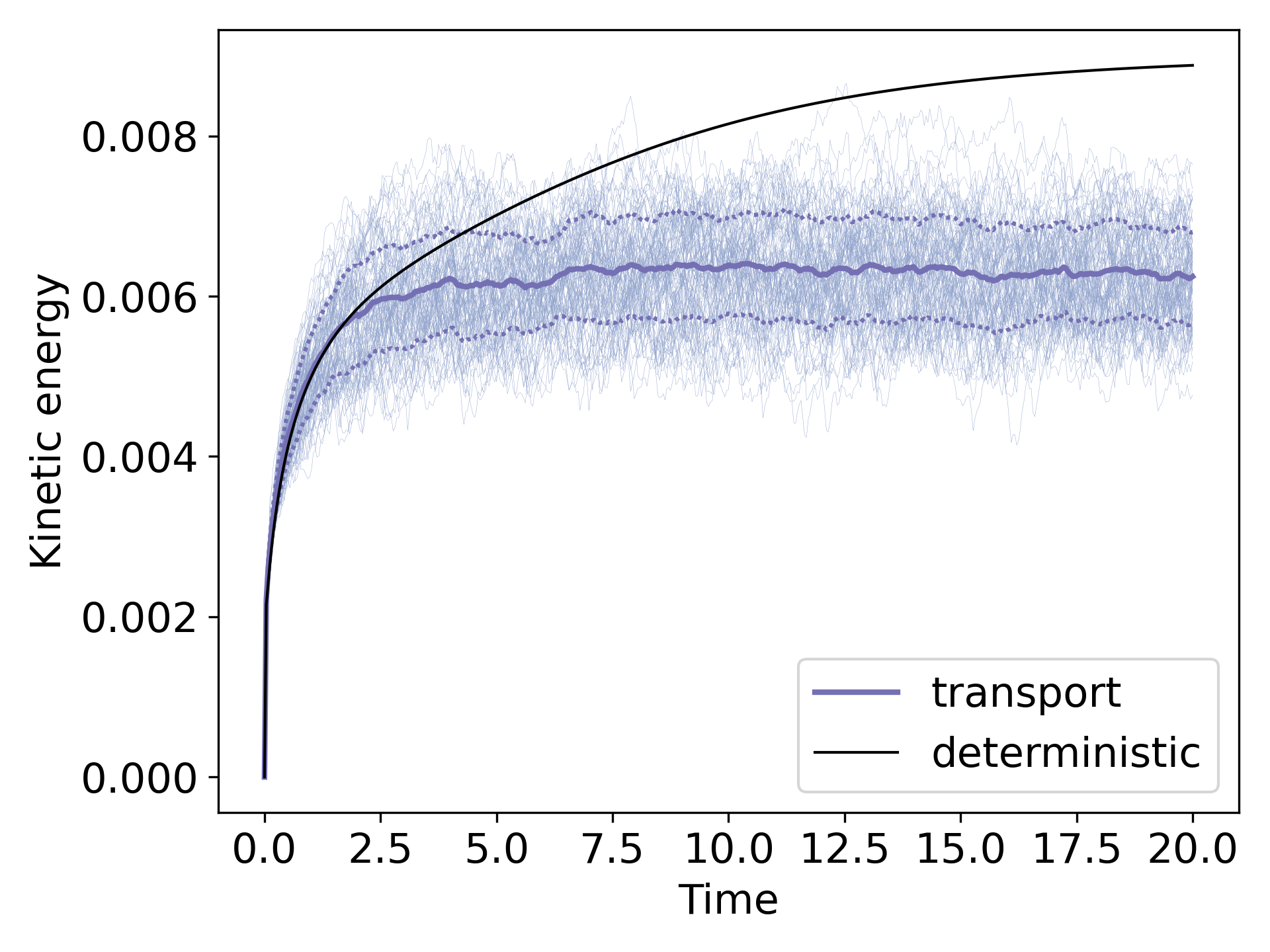}
    \end{subfigure}%
    \caption{Time evolution of the kinetic energy for the lid-driven cavity experiment (Example~\ref{ex:lid-driven}). Thick lines and dotted lines show the mean energy and the mean energy plus or minus one standard deviation, respectively. Individual energy trajectories are shown in pale colours -- left: all noises; right: transport noise only. } 
    \label{fig:evolution_energy}
\end{figure*}

\subsection{A short review about the mathematical theory}
The first mathematical contribution dates back to 1973: in \cite{Bensoussan1973} a particular noise was chosen which allowed the transformation to a random PDE. The latter was solved without tools from stochastic analysis but the equivalence to the original problem remained unclear. A truly stochastic theory was introduced in \cite{Flandoli1995}. General multiplicative noise, possibly nonlinear with respect to the velocity field, could be included. The current focus is on noise with physically motivated structure and with possible regularisation effects. A recent overview can be found in \cite{Flandoli2023}.

In several interesting applications (such as atmospheric flows) one of the spatial directions is of minor importance and can be neglected. In this case one considers
the two-dimensional problem. For the two-dimensional stochastic Navier--Stokes equations  a counterpart of the classical deterministic theory by Ladyshenskaya is known:
solutions are unique and as spatially regular as the data. First contributions are \cite{MR1140758,MR1233785} and an exhaustive picture can be found in \cite{MR3443633}. 
%
%

\subsection{The aims of this article}
Since quantitative predictions in turbulent fluid flows are of particular importance in applications, algorithms that perform reliably for a broad class of relevant noises must be designed and, most importantly, rigorously analysed. While some algorithms, that were originally developed for the deterministic Navier--Stokes equations, have been successfully generalised to the stochastic case; it is that others fail due to the noise-induced irregularity of solutions. 

In this article, we particularly focus on a systematic introduction into various time-discretisations of the stochastic Stokes and Navier--Stokes equations, highlighting the original motivation for the algorithms, as well as their individual advantages and disadvantages. For this reason, we explain key differences between the deterministic and the stochastic fluid flow model for the Stokes problem, and why numerical tools from the deterministic setting fail to perform optimally in the stochastic case.\footnote{The numerical approximation of the stochastic Stokes equations is only a first step towards the approximation of more realistic stochastic fluid flow equations, such as the stochastic  Navier--Stokes equations.} Based on the analytic insights, we provide modified numerical schemes which restore the optimal performance in the stochastic setting as well. For the sake of readability, we keep most arguments on an informal level, highlighting conceptual ideas and referring to the related literature for their detailed discussion. For the stochastic Navier--Stokes equations, we only mention related results in the literature to keep the technical presentation at an easily accessible level.  

Another goal in this article is to improve the {\em comparability} of different numerical schemes by proposing a collection of benchmark problems that test, \emph{e.g.}, the performance of these schemes in critical parameter settings or in situations when the exact solution is known. In the deterministic setting, it is that benchmarks do exist, where academic problems as well as physically relevant fluid-flow problems are assembled to compare different algorithms in terms of accuracy, computing times, and range of applicability. Unfortunately, a corresponding benchmark is still missing in the current setting where the dynamics is also driven by different noises, which is why we propose a set of problems on which we test and compare different algorithms.
To increase their visibility, we present a detailed derivation of these exact solutions for simple noises, and propose concrete parameter configurations with all the necessary information in one place to implement them.  

We use this benchmark for comparative simulations to not only motivate interesting different non-linear dynamics for varying noises (see Figure~\ref{fig:compareVelocity} to Figure~\ref{fig:evolution_energy}), but also to compare complexities of methods as well as induced errors which arise if different numerical methods are used for their approximation. This ranges from consistent methods up to efficient splitting methods, which significantly reduce the computational times.

\subsection{Structure of the article}
After briefly introducing the notation and commenting on how the illustrations of this article are created in Section~\ref{sec:preliminaries}, we split the main part of the article into three, mostly separate sections: 

\textbf{Section~\ref{sec:Num-Stokes}: Stokes equations.} In this section we focus on time-discretisations for the stochastic Stokes equations. We start by recalling analytic results about the regularity of the velocity and the pressure. Afterwards, we provide an overview over the existing literature for the numerical approximation of the equations with a special emphasize on the historical developments. Lastly, in dedicated subsection, we discuss four time-discretisations in detail.

\textbf{Section~\ref{sec:Num-NSE}: Navier--Stokes equation.} We outline the difficulties which arise in the discretisation of the Navier--Stokes equations due to the nonlinearity additionally to those already present for the Stokes equations. The known results for the two-dimensional problem are surveyed and compared to the case of the Stokes system. We close the section by listing some problems which remain open so far.

\textbf{Section~\ref{sec:benchmarks}: Benchmark problems.} To ensure the comparability of different algorithms not only in this article but in between other articles as well, we collect multiple benchmark problems that either can be used to validate the theoretical findings or to challenge the algorithm in demanding more general parameter configurations. For this, we present two approaches: firstly, for each of the three noises we construct explicit solutions so that the exact approximation errors can be computed; and secondly, we provide two examples that have already been used but for which no explicit solutions are known. All benchmark problems are self-contained and hopefully will serve as
a common set of problems in the future
to better compare the performance of different algorithms.  

Finally, in Section~\ref{sec:conclusion} we close the article by reflecting the previous three sections in regards to the aims of the article.

\section{Preliminaries} \label{sec:preliminaries}

\subsection{Notation} \label{sec:Notation}

Let $D =(0,1)^d$, for $d \in \{2,3\}$, be the considered domain, let $T>0$ be the final time, and let $(\Omega, {\mathcal F}, \{{\mathcal F}_t\}, {\mathbb P})$ be a filtered probability space. All probabilities will be expressed with respect to this probability space. For a random variable~$X$, we denote by $\mathbb{E}[X] = \int_{\Omega} X(\omega) \dd \mathbb{P}(\omega)$ its expectation. Let $W$ be a real-valued $\{ {\mathcal F}_t\}$-Wiener process.

We distinguish vector-valued functions from scalar ones by using bold symbols. For two vectors~$\bfa$ and $\bfb$, we denote by $\bfa \cdot \bfb$ the Euclidean inner product, and by $\abs{\bfa}$ its induced norm. Let $c$ and $C$ denote generic non-negative constants that may change their value from line to line. 

Lebesgue and Sobolev spaces over $D$ with integrability~$p$ and differentiability~$k$ for scalar and vector-valued functions are denoted by: $L^p_x$ and $W^{k,p}_x$, and $\mathbb{L}^p$ and $\mathbb{W}^{k,p}$, respectively. In the special case $p=2$ these spaces are abbreviated by~$H^k_x$ and $\mathbb{H}^k$, respectively. We denote the inner product in $\mathbb{L}^2$ by~$\left(\bfa, \bfb \right)$ for $\bfa, \bfb \in \mathbb{L}^2$ and its corresponding norm by $\norm{\bfa}_{\mathbb{L}^2}$.

For a Banach space~$Y$, we similarly define $L^p(0,T;Y)$ and $W^{k,p}(0,T;Y)$ as the space of Bochner measurable functions with values in $Y$, integrability~$p$ and differentiability~$k$. From time to time, we use the shorter notation $L^p_t Y$ and $W^{k,p}_t Y$ instead. Analogously, we define~$L^p(\Omega;Y)$ (or in short $L^p_\omega Y$) as the space of $Y$-valued random variables with finite $p$-th moments. We denote by~$L^{\infty-}_\omega = \bigcap_{p < \infty} L^p_\omega$.

\subsection{Discretisation parameters and implementational details}

Let $N \in \mathbb{N}$ be the number of time steps, $\tau := \frac{T}{N}$ be the time step size, $\{ t_n := n\tau\}_{n=0}^N \subset [0,T]$ denote the collection of nodes of the equi-distant time grid, and $\{\Delta_n W:= W(t_{n+1}) - W(t_{n}) \}_{n=0}^{N-1}$ be the Wiener increments on this time grid.

Even though in this article we exclusively focus on the discussion of various time discretisation, we still need to introduce some details for the spatial discretisation: 

\textit{Mixed finite elements}\footnote{For an introduction into finite elements, we refer to \emph{e.g.} \cite{EG21}.} are defined in terms of an underlying mesh (also called triangulation) and functions (usually polynomials) defined on this mesh for both, velocity fields and pressure. Depending on the choice of tuples of (finite-dimensional) function spaces, some are \textit{stable pairings},\footnote{The stability corresponds to the famous Ladyzhenskaya–Babuška–Brezzi-condition (also called LBB- or inf-sup-condition) which is needed for the well-posedness of saddle point problems.} and hencewith well-suited for approximating velocity and pressure simultaneously:
one famous example of such a stable pair is the \textit{Taylor--Hood element} (\emph{cf}. \cite{Taylor1973}). The lowest order Taylor Hood element uses continuous, piecewise quadratic vector fields and continuous, piecewise linear scalar functions for the approximation of velocity field and pressure, respectively. 

To approximate the expectation and other non-computable quantities (\emph{e.g.}, the standard deviation) of random variables, we employ the \textit{Monte-Carlo method}: \emph{i.e.}, we repeat the same experiment with a new random sample multiple times; the results are then used to compute statistical approximations of the non-computable quantities of interest.

All illustrations in this article use the lowest order Taylor--Hood element defined on the quasi-uniform mesh of the unit square, which is generated from an equi-distant lattice with $16 \times 16$ vertices. For the Monte-Carlo method we compute $1000$ samples. Since the algorithm for the time discretization alters from figure to figure, we explicitly mention the used one in each of them.

The code used for generating the data is available at~\href{https://github.com/joernwichmann/Survey}{https://github.com/joernwichmann/Survey}; it uses the open-source finite element package \textit{Firedrake}~\cite{FiredrakeUserManual}, which itself heavily relies on \textit{PETSc}~\cite{petsc-user-ref}.


\section{Numerics of the stochastic Stokes equations} \label{sec:Num-Stokes}
While the numerical approximation of the stochastic Stokes equations is only a first step towards the approximation of more realistic stochastic fluid flow equations, such as the stochastic  Navier--Stokes equations, its detailed understanding is a prerequisite for the understanding of the latter. 

In this section we summarise and comment on existing algorithms for the approximation of the stochastic Stokes equations, provide a theoretical reasoning why some efficient algorithms that are borrowed from the deterministic setting do not perform optimally, and how these time-splitting methods may be changed to do so again. We accommodate the discussion with computational experiments, some of which suggest further phenomena which so far are not understood theoretically.

We begin by introducing the stochastic Stokes equations forced by a generic noise. Following that, we briefly discuss theoretical results about solutions to the stochastic Stokes equations with a particular focus on the regularity of the pressure. Afterwards, we review the state of the art of numerical algorithms. Finally, in dedicated subsections we introduce some of these algorithms in detail, highlighting their advantages and disadvantages in particular.

\subsection{Model and theoretical results} \label{sec:model-and-theory}
The stochastic Stokes equations are given by:
\begin{framed}
\begin{subequations} \label{def:gen-stokes}
\begin{eqnarray} \label{eq:gen-stokes1}
{\rm d}{\bf u} - \bigl( \mu \Delta {\bf u} - \nabla p \bigr){\rm d}t &=&
{\bf f} {\rm d}t  + \dd \bfZ_{\bfu}  
\qquad \mbox{on } \Omega \times (0,T) \times D,\\  \label{eq:gen-stokes2}
{\rm div} \, {\bf u} &=& 0 \qquad \qquad \qquad  \mbox{on } \Omega \times (0,T) \times D, \\
{\bf u}(0) &=& {\bf u}_0 \qquad \qquad \quad \,\,\,  \mbox{on } \Omega  \times D,
\end{eqnarray}
\end{subequations}
\end{framed}
which are supplemented by periodic boundary conditions. Here, $\mu > 0$ is a parameter that quantifies the strength of the diffusion (for fluids the value is determined by the particular viscosity of the fluid), ${\bf f}: (0,T) \times D \rightarrow {\mathbb R}^d$ is an external body force, $\bfu_0: \Omega  \times D \rightarrow \mathbb{R}^d$ is a given initial vector field (a snapshot of the fluid's velocity at time $t = 0$), and $\bfZ_{\bfu}: \Omega \times (0,T) \times D \rightarrow \mathbb{R}^d$ is a generic noise that generally depends on the solution; its structure is defined by the particular model under consideration. The solution consists of two parts: a vector field $\bfu: \Omega \times (0,T) \times D \rightarrow \mathbb{R}^d$ and a scalar function $p: \Omega \times (0,T) \times D \rightarrow \mathbb{R}$ which we commonly refer to as the velocity and pressure in the fluid, respectively. 

We introduce the generic process $\bfZ_{\bfu}$ to emphasize that the following theoretical discussion does not depend on the particular structure of the noise. However, later we will exclusively choose $\bfZ_{\bfu}$ to be either \textit{additive noise}, \textit{multiplicative noise} or \textit{transport noise} since these structures are strongly motivated by physics and mathematics as we discussed in the introduction.

The analysis of the stochastic Stokes equations -- addressing existence and regularity of solutions -- usually uses the Hilbert space
\begin{equation} \label{eq:Space-exact-div-free}
{\mathbb V} = \{ \pmb{\varphi} \in {\mathbb H}^1|\, {\rm div} \, \pmb{\varphi} = 0\}
\end{equation}
of {\em exactly} divergence-free functions to {\em first} construct velocity without considering pressure. Then, \textit{a posteriori}, a pressure is reconstructed from the already assembled velocity by utilising the Helmholtz (also sometimes called Helmholtz--Hodge) decomposition or de Rham's theorem; see, \emph{e.g.},~\cite{Langa2003} and the references therein. Since we will use this decomposition frequently, we provide more details on its definition.

\textbf{Helmholtz decomposition.} Any vector field~$\bfv \in \mathbb{L}^2(D)$ can be decomposed into an $\mathbb{L}^2(D)$-orthogonal sum of two vector fields; the first one is incompressible and the second one is the gradient of a scalar function. The Helmholtz--Leray projection $P_{\tt HL}: \mathbb{L}^2(D) \to \mathbb{L}^2(D)$ enables an explicit representation of this decomposition, \emph{i.e.}, 
\begin{equation} \tag{HD} \label{def:Helmholtz}
\bfv = P_{\tt HL}\bfv + P_{\tt HL}^\perp \bfv.
\end{equation} 
This projection is defined by $P_{\tt HL} \bfv := \bfv - \nabla \Delta^{-1} \Div \bfv$, where $\Delta^{-1}$ denotes the solution operator to the Poisson equation. Notice that it is linear and trivially stable on $\mathbb{L}^2(D)$. Moreover, it inherits higher-order stability, such as $\mathbb{H}^k(D)$-stability, $k> 0$, from well-known regularity theory for the Laplace equation.

\textbf{Velocity construction.}
As pointed out before, the velocity is constructed independently of the pressure. To do this, one first applies~$P_{\tt HL}$ to~\eqref{eq:gen-stokes1}. By noticing that gradients belong to the kernel of the Helmholtz--Leray projection (\emph{i.e.}, $P_{\tt HL} \nabla = 0$), one arrives at the decoupled velocity evolution equation:
\begin{equation} \label{eq:decoupled-velocity}
{\rm d}{\bf u} + \mu \bfA {\bf u} {\rm d}t =
P_{\tt HL} {\bf f} {\rm d}t  + P_{\tt HL} \dd \bfZ_{\bfu}  
\qquad \mbox{on } \Omega \times (0,T) \times D,
\end{equation}
where $\bfA := -P_{\tt HL} \Delta$ is called the \textit{Stokes operator}. To construct solutions to~\eqref{eq:decoupled-velocity} by abstract analytical tools is a standard procedure: \emph{e.g.}, one typically uses a finite system of orthonormal functions that is generated by the eigenfunctions of the Stokes operator; while its {\em existence} is known from the spectral theory for the Stokes operator, its de-facto {\em construction} is complicated on general domains. 
This is is why such a system of functions is often avoided for computing approximative solutions.

\textbf{Pressure reconstruction.} Once the velocity field has been constructed, a pressure can be reconstructed by utilizing the complementary Helmholtz--Leray projection~$P_{\tt HL}^\perp$. More specifically, applying~$P_{\tt HL}^\perp$ to~\eqref{eq:gen-stokes1} provides us with the reconstruction formula for pressure: 
\begin{equation} \label{eq:reconstruct-pressure}
 \nabla p {\rm d}t = - P_{\tt HL}^\perp {\rm d}{\bf u} + \mu P_{\tt HL}^\perp\Delta {\bf u} {\rm d}t +
P_{\tt HL}^\perp{\bf f} {\rm d}t  + P_{\tt HL}^\perp \dd \bfZ_{\bfu}  
\qquad \mbox{on } \Omega \times (0,T) \times D.
\end{equation}
Notice that, since in general both projections of velocity and noise are non-differentiable in time, the pressure has to be interpreted in a time-integrated manner (as the original pressure is merely a distribution in time; see, \emph{e.g.},~\cite{Langa2003,Wichmann2024} for more details). This is in contrast to the deterministic case, in which the pressure is a well-defined function -- and the irregular pressure will conceptually affect the way how reliable and efficient schemes can be constructed for the approximation of the stochastic Stokes equations. We will write $P(t) := \int_0^t p(s) \dd s$ and call it the \textit{time-integrated pressure}. 

It is equation~\eqref{eq:reconstruct-pressure} that, at least formally, allows the time-integrated pressure to be decomposed into individual sources: 
\begin{itemize}
\item (acceleration) $P_{\tt HL}^\perp {\rm d}{\bf u}$ encapsulates pressure contributions due to the time-integrated acceleration in the fluid. In general, this term is non-differentiable, limiting the regularity of pressure; but in most cases, the velocity is sufficiently regular so that this term vanishes due to the incompressibility condition.

\item (diffusion) $\mu P_{\tt HL}^\perp\Delta {\bf u} {\rm d}t $ corresponds to the pressure generated by the diffusion. In domains without boundary (as we have assumed), the incompressibility of the velocity field ensures that this term vanishes. In domains with boundary, it has a non-trivial contribution; but its regularity is similar to the deterministic case, meaning that it is not diminishing the regularity of the pressure.
  
\item (body force) $P_{\tt HL}^\perp{\bf f} {\rm d}t$ is the pressure that is created by an external force. It behaves as in the deterministic case. For sufficiently regular forces, this term does not create any difficulties. However, if the force is too rough, it can lead to the ill-posedness of the pressure; see, \emph{e.g.},~\cite{Simon1999} and~\cite{Langa2003} for the deterministic and stochastic cases, respectively.

\item (noise) $P_{\tt HL}^\perp \dd \bfZ_{\bfu}$ is the \textit{stochastic pressure} that is created by the noise. On the one hand, as soon as the noise violates the incompressibility condition, this crucially affects the regularity of the pressure. On the other hand, if the noise complies with the constraint, then this term vanishes. Consequently, in this case the stochastic pressure does not have an effect on the time-integrated pressure, and therefore the original pressure could be studied instead. 
\end{itemize}

\begin{framed}
In summary, solutions to the stochastic Stokes equations are typically constructed in two stages: firstly, the velocity is determined independently of the pressure by solving the evolution equation restricted to the space of \textit{exactly divergence-free} vector fields; and secondly, the pressure is reconstructed a-posteriori by de Rham's theorem. In contrast to the deterministic Stokes equations, the temporal regularity of the pressure is reduced substantially, which motivates the use of the \textit{time-integrated pressure}.
\end{framed}
This analytical approach is different from the numerical one where the approximate velocity and pressure need to be computed at the same
instant. When devising numerical schemes, related finite element spaces for both entities are exposed to proper restrictions to lead to stable schemes, in particular. Only algorithms that i) can be implemented with reasonable effort; ii) terminate within reasonable time; and iii) approximate the desired state with reasonable accuracy are useful in applications. 

In the remaining part of this section, we discuss some commonly used discretizations in detail. But before we do so, we review the mathematical literature on the numerical approximation of the stochastic Stokes equations. To facilitate the comparison between different results as well as to enable the fast look-up of individual results, we give a summary of the algorithms and their respective error decay rates in Table~\ref{table:Stokes}.

\subsection{State of the art: algorithms and estimates} \label{sec:state-Stokes}
To construct {\em numerical schemes} for the stochastic Stokes equations (\emph{cf.}~\eqref{def:gen-stokes}), one has to proceed differently in comparison to the previous section. In the space-time discrete setting, it is customary to compute the approximate velocity and approximate pressure {\em simultaneously} rather than one after the other since the practical use of the discrete Helmholtz projection (which could be used to decouple the approximate velocity from the approximate pressure) reduces the efficiency of schemes.  

As a result, numerical schemes are typically in use, for which the incompressibility constraint is only met {\em approximately} rather than {\em exactly}. Violating the incompressibility in computations causes a subtle interplay of error sources that affect the related approximate velocity and pressure, which is most delicate
in the presence of {\em general ${\mathbb L}^2$-valued} noise. But this kind of noise is relevant to {\em e.g.}~address extended effects such as buoyancy, or electric effects in complex fluid scenarios -- and hence is the most relevant model, for which efficient and provably convergent algorithms need to be found.

\subsubsection{Historical developments}
As observed in~\cite{CHP1}, the interplay of ${\mathbb L}^2$-valued noise and approximate pressure is delicate: due to their different scaling, the approximate pressure {\em cannot} be estimated uniformly in time. An exception is ${\mathbb V}$-valued noise; in this case, the approximate pressure is controlled uniformly since the noise acts on the approximate velocity only. The reason for this observation is linked to the decomposition \eqref{eq:reconstruct-pressure} of the pressure.

\smallskip
Especially time-stepping algorithms, such as the Chorin method (see, \emph{e.g.},~\cite{Chorin1968,Tmam1969} for its original definition), rely on the regularity of the pressure. For the deterministic (Navier-)Stokes equations, the pressure is well-behaved and thus, the reliability of the scheme is guaranteed. Motivated by the reduced computational complexity of the scheme, it is frequently used in applications; however, as we have seen above, the pressure behaves irregularly in the stochastic case, raising the question whether the Chorin method behaves well for the stochastic Stokes equations, too.

\smallskip
A negative answer was provided by Carelli, Hausenblas, and Prohl in~\cite{CHP1}: they conducted numerical simulations using the Chorin scheme for the stochastic Stokes equations when driven by ${\mathbb L}^2$-valued noise and noticed a {\em reduced} convergence order from ${\mathcal O}(\tau^{1/2})$ to ${\mathcal O}(\tau^{1/4})$, which they traced back to the limited regularity of the approximate pressure -- whose original analogue is decomposed in~\eqref{eq:reconstruct-pressure} into differently regular parts.

They noticed that the dependence of the approximate velocity on the irregular approximate pressure -- which in~\eqref{eq:reconstruct-pressure} would correspond to the last term -- has to be removed in order to recover an optimally convergent scheme. Therefore, they introduced a {\em pressure decomposition}: in the beginning of every time-step, they removed the rough (in time) gradient part of the noise, which required the computation of a discrete Helmholtz projection; and then, in a second step, they used the corrected noise instead of the original one in the Chorin method. 
For this pressure-corrected Chorin method (see Algorithm~\ref{algo:semi-discrete_ChorinCorrected} below), numerical simulations now indicate its optimal convergence even for ${\mathbb L}^2$-valued noise.

A proof of this optimal convergence was later provided in \cite{FV2022} through the reformulation of the pressure-corrected Chorin method as a semi-explicit pressure stabilization method, which now exploits the {\em regular} (in time) pressure part for the stabilization. In~\eqref{eq:reconstruct-pressure} this would correspond 

\begin{landscape}
\begin{table}
\centering
\begin{tblr}{
  cells = {c},
  row{1} = {Silver},
  row{2} = {Silver},
  cell{1}{1} = {r=2}{},
  cell{1}{2} = {r=2}{},
  cell{1}{3} = {c=2}{},
  cell{1}{5} = {c=2}{},
  cell{1}{7} = {r=2}{},
  cell{6}{1} = {r=2}{},
  cell{6}{2} = {r=2}{},
  cell{8}{1} = {r=2}{},
  cell{8}{2} = {r=2}{},
  cell{13}{1} = {r=2}{},
  cell{13}{2} = {r=2}{},
  cell{13}{3} = {r=2}{},
  cell{13}{6} = {r=2}{},
  cell{13}{7} = {r=2}{},
  vlines,
  hline{1,17} = {-}{0.08em},
  hline{2} = {3-6}{},
  hline{3-13} = {-}{},
  hline{15-17} = {-}{},
}
\textbf{Ref.} & \textbf{ Noise} & \textbf{Discretization} &                                 & \textbf{Convergence Rates} (Space in which the error is measured: Rate)                                                                                                                                                               &                                                                                                                                                                                                 & \textbf{Exp.} \\
                    &                 & \textbf{Time} (-stepping)           & \textbf{Space} (FEM)                  & \textbf{Velocity}                                                                                                                                                                        & \textbf{Time-int. Pressure}                                                                                                                                                                               &                      \\
\cite{LMS1}               & mult.        & impl. Euler          & stable     & $L^{\infty}_t L^2_\omega \mathbb{L}^2:\, \tau^{\frac{1}{2}} + h^2$ & $L^{\infty}_t L^2_\omega L^2_x:\,\tau^{\frac{1}{2}} + h^2$ & Yes  \\ 
\cite{FQ1}               & mult.        & impl. Euler          & stable  & $L^{\infty}_t L^2_\omega \mathbb{L}^2 + L^2_\omega L^2_t \mathbb{H}^1: \, \tau^{\frac{1}{2}} + \tau^{-\frac{1}{2}}h$ & $L^{\infty}_t L^2_\omega L^2_x:\, \tau^{\frac{1}{2}} + \tau^{-\frac{1}{2}}h$ & Yes    \\
\cite{CHP1}               & mult.        & Chorin method          & arbitrary  & $L^{\infty}_t L^2_\omega \mathbb{L}^2 + L^2_\omega L^2_t \mathbb{H}^1:\, \tau^{\frac{1}{2}} + h + \tau^{-\frac{1}{2}}h^2$ & -- & Yes \\ 
\cite{FPV1}              &  mult.        & proj., impl. Euler & stable & $L^{\infty}_t L^2_\omega \mathbb{L}^2 + L^2_\omega L^2_t \mathbb{H}^1:\, \tau^{\frac{1}{2}} + h$ & $L^{\infty}_t L^2_\omega L^2_x:\,\tau^{\frac{1}{2}} + h$ & Yes   \\
\cite{FPV1}              &  mult.        & $\cdots$ + $\varepsilon$-quasi compr.  & arbitrary & $L^{\infty}_t L^2_\omega \mathbb{L}^2 + L^2_\omega L^2_t \mathbb{H}^1:\, \tau^{\frac{1}{2}} + h + \varepsilon^{-\frac{1}{2}} h^2$ & $L^{\infty}_t L^2_\omega L^2_x:\, \tau^{\frac{1}{2}} + h + \varepsilon^{-\frac{1}{2}} h^2$ & Yes   \\
\cite{FV2022}              &  mult.        & Chorin method  & arbitrary & $L^{2}_\omega L^2_t \mathbb{L}^2:\, \tau^{\frac{1}{4}} + \tau^{-\frac{1}{2}}h$ & $ L^2_\omega L^{2}_t L^2_x:\, \tau^{\frac{1}{4}} + \tau^{-\frac{1}{2}}h$ & Yes   \\
\cite{FV2022}              &  mult.        & $\cdots$ + proj. & arbitrary & $L^{\infty}_t L^2_\omega \mathbb{L}^2 + L^2_\omega L^2_t \mathbb{H}^1:\, \tau^{\frac{1}{2}} + h + \tau^{-\frac{1}{2}}h^2$ & $L^{\infty}_t L^2_\omega L^2_x:\,\tau^{\frac{1}{2}} + h + \tau^{-\frac{1}{2}}h^2$ & Yes   \\
\cite{Gyngy2007}            &  mult.        & impl. Euler & -- & $ L^2_\omega L^{\infty}_t \mathbb{L}^2 + L^2_\omega L^2_t \mathbb{H}^1:\, \tau^{\frac{1}{2}}$ & -- & No   \\
\cite{Gyngy2008}            &  mult.        & impl. \& expl. Euler & abstract & $ L^2_\omega L^{\infty}_t \mathbb{L}^2 + L^2_\omega L^2_t \mathbb{H}^1:\, \tau^{\frac{1}{2}} + h$ & -- & No   \\
\cite{Vo2023}            &  mult.        & impl. Milstein & stable & $ L^2_\omega L^{\infty}_t \mathbb{L}^2 + L^2_\omega L^2_t \mathbb{H}^1:\, \tau + h$ & $L^{\infty}_t L^2_\omega L^2_x:\,\tau + h$ & Yes   \\
\cite{Le2024}            &  mult.        & mod., impl. Euler & stable & $ L^2_\omega L^{\infty}_t \mathbb{L}^2 + L^2_\omega L^2_t \mathbb{H}^1:\, \tau^{\frac{1}{2}} + h$ & -- & Yes   \\
         &   &  & + div-free & $ L^2_\omega W^{\frac{1}{2},2}_t \mathbb{L}^2:\, \tau^{\frac{1}{2}} + h$ &  &    \\
\cite{2024arXiv241214316D}            &  tran.        & Crank--Nicolson & stable & -- & -- & Yes   \\
\cite{Vo2025}               & mult.        & impl. Euler          & stable  & $L^{\infty-}_\omega L^{\infty}_t \mathbb{L}^2 + L^{\infty-}_\omega L^2_t \mathbb{H}^1:\, \tau^{\frac{1}{2}} + h + \tau^{-\frac{1}{2}}h$ & $L^{\infty-}_\omega L^{\infty}_t L^2_x:\, \tau^{\frac{1}{2}} + h + \tau^{-\frac{1}{2}}h$ & Yes \\ 
\end{tblr}
\caption{Summary of the available algorithms for the numerical approximation of the stochastic Stokes equations, including their respective convergence rates for the velocity and time-integrated pressure errors, as well as whether numerical experiments have been conducted. Here, `stable' refers to the inf-sup-stability, needed for mixed finite element pairs. For splitting methods `arbitrary' pairings can be used. The notation of the spaces, in which the errors are measured, is explained in Section~\ref{sec:Notation}. Further details on the respective literature are presented in Section~\ref{sec:state-Stokes}.}
\label{table:Stokes}
\end{table}
\end{landscape}

to erasing the last term. We also refer to \cite{FPV1} for a related error analysis, as well as supporting simulations.

\smallskip
To continue our discussion, we recall that -- due to the merely ${\mathbb L}^2$-valued noise -- the pressure takes values in some Sobolev space with negative (temporal) differentiability only; see, \emph{e.g.}, \cite{Langa2003,Wichmann2024}. It is this negative differentiability that led Feng and Qui (\emph{cf.}~\cite{FQ1}) to study the convergence of the discretely time-integrated approximate pressure towards the time-integrated pressure. They showed that the time-integrated pressure error converges with optimal rate. In addition,
they conducted computational studies to evidence the failure of convergence for the (non-integrated) pressure error; see also Figure~\ref{fig:failure_chorin_inOne} (right).

\smallskip
We interpret these results as a further indication to stay cautious when generalising presumably efficient methods, that were originally designed and analysed for the deterministic Stokes equations, to the stochastic case.

\medskip
An early error analysis for a full discretisation in time {\em and} space with general ${\mathbb L}^2$-valued noise is \cite{CP} (for the stochastic Navier--Stokes equations), which consists of two steps: {\em firstly}, a temporal semi-discretization is introduced as an auxiliary problem to study temporal discretization effects,
leading to an optimal error estimate; {\em then} the second step addresses spatial discretization effects. The crucial interplay of the (discrete in time) Lagrange multipliers (the pressure) with the general ${\mathbb L}^2$-valued noise is avoided by the use of an {\em exactly} divergence free finite element method (known as `Scott-Vogelius elements'; see
\cite{SV1985}) to derive the error decay rate ${\mathcal O}(\sqrt{\tau} + h)$. 

\smallskip
As mentioned above, it is known that this discretisation in space bears certain practical disadvantages, which is why stable mixed finite element methods are studied in \cite{FQ1}. In their analysis, the {\em failure of the pointwise} incompressibility condition for the computed velocity created another error term which they estimated to be of order ${\mathcal O}(\frac{h}{\sqrt{\tau}})$. However, it remained questionable whether this time-space coupling term in the error estimate is sharp. 

\smallskip
Since the limited regularity of the (discrete) Lagrange multiplier was again the source of sub-optimality, the idea in \cite{FPV1} was to `cut off' the rough gradient part from the noise by means of a Helmholtz projection in every time step: as a consequence, the analysis then again avoids a coupling term as of
${\mathcal O}(\frac{h}{\sqrt{\tau}})$ for a spatial discretization of the implicit Euler scheme that is based on stable mixed finite element methods. However, and in comparison to the original scheme, this modified scheme is now more costly due to the additionally required Helmholtz projections of the noise in every iteration step. 

\smallskip
A negative answer on the necessity of the time-space coupling term in the error estimate proposed in~\cite{FQ1} is given in~\cite{LMS1}: the authors used semigroup theory to derive an optimal order of convergence~${\mathcal O}(\sqrt{\tau} + h^2)$ for velocity approximates in a stable mixed finite element method in combination with the implicit Euler scheme, thus avoiding the practically restrictive space-time coupling term in the error estimate. Their method of proof avoids the {\em separate} consideration of temporal and spatial errors and thus the auxiliary (semi-discrete) problem by means of a {\em direct} space-time error analysis. This approach also allows an improvement of the convergence rate with respect to the spatial discretization parameter by sharp results for the fully discrete Stokes semigroup associated to the spatial discretization by stable mixed finite elements.

\smallskip
An alternative possibility to prove error estimates without coupling terms for stable mixed finite element methods is detailed in~\cite{BD1} (though this analysis already covers the Navier--Stokes equations as well): they computed the pressure (by solving certain elliptic problems) leading to an additional stochastic integral. The pressure is still very irregular in time as noted above, but one can benefit from the martingale structure. Martingale estimates, such as the Burkholder-Davis-Gundy inequality, in combination with a refined analysis of the projection error of the pressure to the discrete pressure space lead to the convergence rate ${\mathcal O}(\sqrt{\tau} + h)$ for the velocity error without assuming any condition on the temporal and spatial discretization parameters. Moreover, a corresponding error estimate for the pressure error can be deduced by means of the equation itself. 

\subsubsection{Further related results}
In~\cite{Gyngy2007,Gyngy2008}, Gy\"ongy and Millet investigated time-discrete and fully discrete algorithms that combine the implicit or explicit Euler methods with a generic spatial discretization (\emph{e.g.}, wavelets or finite differences) for abstract stochastic evolution equations forced by multiplicative noise. They showed that these algorithms converge at certain rates under certain conditions on the data. Since the scope of both articles was the development of an error analysis for a broad class of equations and their discretisation, no effort was spent to specific problems such as \eqref{def:gen-stokes} which involves a constraint as well.

\smallskip
In~\cite{Vo2023}, Vo analysed the Milstein scheme\footnote{The Milstein scheme is a higher-order time-stepping algorithm which was originally introduced for stochastic differential equations; see \emph{e.g.}~\cite{Milshtejn1975}. For these equations, its error decays with rate $1$ (in time).} in combination with mixed finite elements for approximating the Stokes equations driven by sufficiently smooth multiplicative noise. Provided that the data are regular enough (in particular, the noise coefficient needs to be differentiable in the direction of velocity), he established that the velocity and pressure errors converge with rate~$1$ in time and space, doubling the convergence rate in time in comparison to the implicit Euler and Chorin methods. Moreover, he conducted numerical experiments to validate the theoretical results. 

\smallskip
In~\cite{Le2024}, Le and Wichmann studied a slightly modified implicit Euler scheme in combination with a generic spatial discretisation for the generalised Stokes equations\footnote{The generalized Stokes equations model power-law fluids. These are fluids for which the shear stress tensor and strain rate tensor are in a non-linear (power-law) relation.}.  This modification enabled them to reduce the regularity requirements of the target solution while preserving the convergence rate ${\mathcal O}(\sqrt{\tau} + h)$. In addition to the classical error measure, they showed that the velocity error converges in some time-fractional Sobolev space with the same convergence rate, too. The theoretical results are supported by numerical experiments.

\smallskip
In~\cite{2024arXiv241214316D}, Droniou, Le, and Wichmann studied a fully discrete algorithm -- that is based on the Crank--Nicolson method in combination with the Gradient Discretization Method\footnote{The Gradient Discretization Method is a mathematical tool that enables a unified analysis for a broad class of discretisations, such as continuous and discontinuous finite elements, and hybrid higher-order methods; see, \emph{e.g.},~\cite{Droniou2018} for details.} for the temporal and spatial discretisation, respectively -- for the approximation of the generalised Stokes equations forced by transport noise. 

They observed that transport noise and, in particular, the Crank--Nicolson method enable a global-in-time control on the approximate velocity. Motivated by this, they studied two families of measures that are both promising candidates for constructing the {\em invariant measure} (\emph{i.e.}, the stationary state for stochastic equations) numerically. For no-slip boundary conditions, they proved that both families converge to the unique invariant measure. However, for non-trivial boundary conditions, this convergence has remained unresolved. 

Their theoretical investigation is accompanied by a series of computational studies, specifically targeting the convergence of the distribution of the kinetic energy and of the velocity at particular locations. In all cases, they found that these distributions converge even for non-trivial boundary conditions. 

\smallskip
In~\cite{Vo2025}, Vo studied the implicit Euler method in combination with mixed finite elements with a special emphasize on arbitrary high moments of the error and the pathwise error: the author showed that, if the initial condition has finite moments of arbitrary order, then arbitrary moments of the errors of velocity and pressure decay with rate $\mathcal{O}( \tau^{\frac{1}{2}} +  \tau^{-\frac{1}{2}}h)$, which extends to the pathwise error too. Both convergence rates were validated by a series of computational studies.

 \medskip
In most of the cited works above, the error analyses  the interplay of general noise, regularity of the Lagrange multiplier, and space-time discretization based on mixed finite elements were done on a torus with periodic boundary conditions, and related works in the more involved setting of bounded domains with Dirichlet conditions are still rare. The goal in the remainder of this section is to detail the sketched effects that do {\em not} appear in the deterministic setting of the Stokes problem, to get prepared to additional ones in the case of the stochastic Navier--Stokes equations in Section~\ref{sec:Num-NSE}.

\subsection{Efficient time discretizations of the stochastic Stokes equations}\label{sec:EfficientTimeDisc}
We now consider time discretization strategies for the stochastic Stokes equations for the aforementioned three noises. In particular, we present three algorithms for the time discretization of~\eqref{def:gen-stokes} forced by additive and multiplicative noises (the implicit Euler method, Chorin method, and pressure-corrected Chorin method); and one algorithm for transport noise (the Crank--Nicolson method). 

These algorithms are time-stepping algorithms, which means that their definition consists of two parts: a) an initialisation rule, and b) an iteration rule. However, since all algorithms are one-step algorithms, their initialisation is done trivially; \emph{i.e.}, the given initial velocity defines the initial state of the time discretization: ${\bf u}^0 := {\bf u}_0$.

The particular noises, interpreted as generic ones to connect them to the analytic section, are given by:
\begin{itemize}
\item (additive noise)~$\dd \bfZ_{\bfu} = \bfsigma \dd W$;
\item (multiplicative noise)~$\dd \bfZ_{\bfu} = \bfsigma(\bfu) \dd W$; and
\item (transport noise)~$\dd \bfZ_{\bfu} =(\bfsigma \cdot \nabla) \bfu \circ \dd W$,
\end{itemize}
where we have restricted the randomness to act via one real-valued Wiener process~$W$ for simplicity.

\subsubsection{Implicit Euler method} \label{ie1}
The first method we will discuss is the implicit Euler scheme, which is defined as follows: 

\begin{framed}
\begin{algorithm}[implicit Euler] \label{algo:Semi-disc--ImplicitEuler}
Let $n\geq 0$. Given the previous velocity ${\bf u}^{n}$, compute the next velocity and pressure iterates $({\bf u}^{n+1}, p^{n+1})$ from
\begin{subequations}\label{stokes-2}
\begin{eqnarray}\label{stokes-2a}
\bigl({\bf u}^{n+1} - {\bf u}^n\bigr) - \tau \bigl(\Delta {\bf u}^{n+1} - \nabla p^{n+1} - {\bf f}^{n+1}\bigr) &=& \pmb{\sigma}(t_n, {\bf u}^{n})\Delta_n W\qquad \, \mbox{on }  \Omega \times D, \\  \label{stokes-2b}
{\rm div}\, {\bf u}^{n+1} &=& 0 \qquad \qquad\qquad\qquad \mbox{on } \Omega \times D.
\end{eqnarray}
\end{subequations}
\end{algorithm}
\end{framed}

From an analytical viewpoint, the solution consists of {\em two} components $({\bf u}^{n+1}, p^{n+1})$, whose construction can be done similarly as for the non-discretised stochastic Stokes equations described in Section \ref{sec:model-and-theory}. 
A relevant property that distinguishes the scheme is the following stability result (for simplicity, we set $\{{\bf f}^n\}_{n=1}^N \equiv {\bf 0}$):

\begin{lemma}[velocity stability] \label{lem:IE-energyEstimate} For $0\leq n \leq N$, it holds
\begin{equation}\label{disc_ener0}
{\mathbb E} \left[ {\mathscr E}({\bf u}^n)\right] +  {\mathbb E}\left[ \tau \sum_{\ell=0}^n \Vert \nabla {\bf u}^{\ell+1}\Vert^2_{{\mathbb L}^2}\right]\leq C(t_n)\Big(1+{\mathbb E} \left[ {\mathscr E}({\bf u}^0)\right]\Big),
\end{equation}
where the energy functional is again (\emph{cf.}~\eqref{eq:energy-preservation}) given by ${\mathscr E}({\bf v}) = \frac{1}{2} \Vert {\bf v}\Vert^2_{{\mathbb L}^2}$ ({\em i.e.}, the kinetic energy).
\end{lemma}
The verification of Lemma~\ref{lem:IE-energyEstimate} utilizes standard arguments only. To ensure a self-contained presentation and to familiarise ourselves with the novelties due to the driving noise, we provide a proof.  
\begin{proof}
The formal derivation of~\eqref{disc_ener0} rests on considering (\ref{stokes-2}) in variational form for a single $\omega \in \Omega$, and using ${\bf u}^{n+1}(\omega)$ as test function. In order to apply independence  of the increment $\Delta_n W$ of the $\sigma$-algebra ${\mathcal F}_{t_n}$, standard calculations lead to the following identity which holds ${\mathbb P}$-a.s.:
\begin{eqnarray}\nonumber
&&\frac{1}{2} \bigl( \Vert {\bf u}^{n+1}\Vert^2_{{\mathbb L}^2} - \Vert {\bf u}^{n}\Vert^2_{{\mathbb L}^2} + \Vert {\bf u}^{n+1} - {\bf u}^{n}\Vert^2_{{\mathbb L}^2}\bigr)
+ \tau \Vert \nabla {\bf u}^{n+1}\Vert^2_{{\mathbb L}^2} \\ \nonumber
&&\qquad \qquad \qquad = \bigl(\pmb{\sigma}(t_n, {\bf u}^{n})\Delta_n W, {\bf u}^{n+1} - {\bf u}^{n}\bigr) + \bigl(\pmb{\sigma}(t_n, {\bf u}^{n})\Delta_n W,  {\bf u}^{n}\bigr) \\ \label{discr_energy1}
&&\qquad \qquad \qquad \leq \frac{1}{2} \Vert {\bf u}^{n+1}-{\bf u}^n\Vert^2_{{\mathbb L}^2} + \Vert \pmb{\sigma}(t_n, {\bf u}^{n}) \Delta_n W\Vert^2_{{\mathbb L}^2} + \bigl(\pmb{\sigma}(t_n, {\bf u}^{n})\Delta_n W,  {\bf u}^{n}\bigr)\, .
\end{eqnarray}
The reason why we proceed like this on the right-hand side of the equality sign is that stochastic increments $\Delta_n W$ that are obtained on the mesh $\{t_n\}_{n=0}^N$ are of order
${\mathcal O}(\sqrt{\tau})$, so a standard estimation of the term $(\pmb{\sigma}(t_n, {\bf u}^{n})\Delta_n W, {\bf u}^{n+1})$ on the right-hand side by Cauchy-Schwarz inequality would spoil a discrete energy bound if estimated directly.
Now the leading term on the right-hand side in (\ref{discr_energy1}) may be absorbed on the left-hand side. For the last term on the right-hand side, we have that
${\mathbb E}\bigl[ \bigl(\pmb{\sigma}(t_n, {\bf u}^{n})\Delta_n W,  {\bf u}^{n}\bigr)\bigr] = 0$ since the random variables $\Delta_n W$ and $\pmb{\sigma}^\top(t_n, {\bf u}^n){\bf u}^n$ are independent, and increments $\Delta_n W$ are normally distributed, with mean value $0$. For the remaining term on the right-hand side we use independence of $\pmb{\sigma}(t_n, {\bf u}^n)$ and $\Delta_n W$, and we also use ${\mathbb E}[ \abs{ \Delta_n W }^2] = \tau$, as well as that $\pmb{\sigma}$ is Lipschitz-continuous to arrive at the estimate 
\begin{equation}\label{term1}
{\mathbb E}[\Vert \pmb{\sigma}(t_n, {\bf u}^{n}) \Delta_n W\Vert_{{\mathbb L}^2}^2] \leq C \tau \bigl(1+ {\mathbb E}[\Vert {\bf u}^{n}\Vert^2_{{\mathbb L}^2_2}] \bigr).\end{equation} This estimation is now suitable to apply the discrete form of Gronwall's lemma to deduce (\ref{disc_ener0}). 
\end{proof}


By energy-type arguments similar to those that were sketched to settle the {\em discrete stability} property (\ref{disc_ener0}) for (\ref{stokes-2}), it is then possible for the ${\mathbb V}$-valued solution $\{ {\bf u}^n\}_{n=1}^N$ of this semi-discretisation to verify convergence with strong rate $\frac{1}{2}$ for the velocity iterates (see, {\em e.g.}~\cite{Gyngy2007}), {\em i.e.},
\begin{equation}\label{stokes-3}
\max_{1 \leq n \leq N} {\mathbb E}\bigl[ \Vert {\bf u}(t_n) - {\bf u}^n\Vert_{{\mathbb L}^2}^2\bigr] \leq C  \tau.
\end{equation}
Its derivation is again based on working with {\em exactly divergence free} (test) functions, such that after application of the Helmholtz projection the pressure vanishes from (\ref{stokes-2}), and the optimality of this rate follows from well-known limited temporal regularity of the driving Wiener process; see {\em e.g.}~\cite{KS1}. 

To derive a corresponding estimate for the hydrostatic pressure $\{p^{n}\}_{n=1}^N$ in (\ref{stokes-2}) is a more delicate endeavor, which reveals a new difficulty for the construction of numerical schemes for (\ref{stokes-2})  that is not present in the related well-known theory for the {\em deterministic case} ({\em i.e.}, $\pmb{\sigma} \equiv 0$; see {\em e.g.}~\cite{HR1,HR2}): the interplay of the driving noise with the pressure $\{p^{n}\}_{n=1}^N$ -- which mathematically serves its role as the Lagrange multiplier to guarantee~\eqref{stokes-2b}. The following estimate has first been obtained in \cite{FQ1} (see also Theorem 3.3 in \cite{FPV1}) for the time-integrated pressure:
\begin{equation}\label{stokes-4}
\max_{1 \leq n \leq N} {\mathbb E} \bigg[ \bigl\Vert \int_0^{t_{n}} p(s)\, {\rm d}s - \tau \sum_{\ell=1}^n p^\ell \bigr\Vert^2_{{\mathbb L}^2/{\mathbb R}} \bigg] \leq C  \tau\, .
\end{equation}
In a sense, only the `anti-derivative' of the pressure may be well-approximated by the iterates $\{ p^n\}_{n=1}^N$ from (\ref{def:gen-stokes}) -- rather than $\max_{1 \leq n \leq N}{\mathbb E}\big[ \Vert p(t_n) - p^n\Vert^2_{{\mathbb L}^2/{\mathbb R}}\bigr]$ which is used in the deterministic case. As already mentioned, the reason for this weaker error estimate may again be motivated by a discrete stability result for the pressure iterates $\{ p^n\}_{n=1}^N$ from (\ref{def:gen-stokes}) (for simplicity, we assume that $\{{\bf f}^n\}_{n=1}^N = {\bf 0}$):
\begin{lemma}[pressure stability] \label{lem:pressure-stability} There hold
\begin{equation*}
\max_{1 \leq n \leq N} {\mathbb E}\bigl[ \Vert \nabla p^{n}\Vert^2_{{\mathbb L}^2}\bigr] \leq \frac{C}{\tau} \qquad \text{ and } \qquad \max_{1 \leq n \leq N} {\mathbb E}\bigg[ \Vert \sum_{\ell=1}^n \tau \nabla p^{\ell}\Vert^2_{{\mathbb L}^2}\bigg] \leq C.
\end{equation*}
\end{lemma}

\begin{proof}
Our argumentation proceeds pathwise, hence let $\omega \in \Omega$ be an arbitrary event. First, we apply the complementary Helmholtz--Leray projection~$P_{\tt HL}^\perp$ to~\eqref{stokes-2a}. Since we assumed that $D$ is the torus and $\bff^n = \bf0$, and by using the incompressibility of $\bfu$ we find that 
\begin{equation} \label{druck-0}
\nabla p^{n+1} (\omega) = \frac{1}{\tau} P_\mathtt{HL}^\perp \left[ \pmb{\sigma}(t_n, {\bfu^{n}(\omega)})\Delta_n  W(\omega) \right].
\end{equation}
Consequently, squaring both sides, integrating in space, together with taking expectations (here we use the independence of the Wiener increment to the previous velocity iterate) and maximum yield
\begin{equation}\label{druck-1}
\max_{1 \leq n \leq N} {\mathbb E}\bigl[ \Vert \nabla p^{n}\Vert^2_{{\mathbb L}^2}\bigr] = \frac{1}{\tau} \max_{0 \leq n \leq N}  {\mathbb E}\bigl[ \Vert P_\mathtt{HL}^\perp \pmb{\sigma}(t_n, {\bf u}^{n})\Vert^2_{{\mathbb L}^2}\bigr].
\end{equation}
It remains to use the stability of the projection, and a further invocation of (\ref{term1}) and (\ref{disc_ener0}) to derive
\begin{equation}\label{druck-2}
\max_{1 \leq n \leq N} {\mathbb E}\bigl[ \Vert \nabla p^{n}\Vert^2_{{\mathbb L}^2}\bigr] \leq \frac{C}{\tau} \bigl(1+ \max_{0 \leq n \leq N} {\mathbb E}[\Vert {\bf u}^{n}\Vert^2_{{\mathbb L}^2_2}]\bigr) \leq
\frac{C}{\tau}\, .
\end{equation}
This settles the first inequality of Lemma~\ref{lem:pressure-stability}. The second follows analogously by first summing up~\eqref{druck-0}.
\end{proof}

\begin{figure*}[t!]
    \centering
    \begin{subfigure}[t]{0.5\textwidth}
        \centering
        \includegraphics[width=1.0\textwidth]{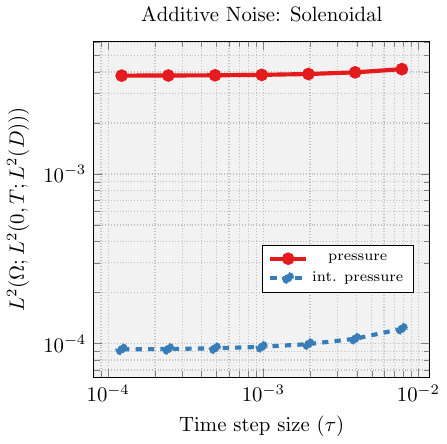}
    \end{subfigure}%
    ~ 
    \begin{subfigure}[t]{0.5\textwidth}
        \centering
        \includegraphics[width=1.0\textwidth]{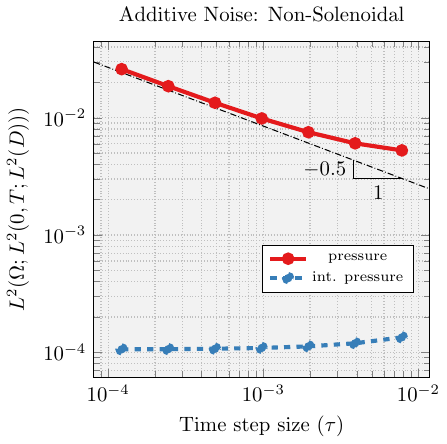}
    \end{subfigure}
    \caption{Norm evolution of approximate pressure and time-integrated (abbreviated by int.) pressure --- generated by the implicit Euler scheme (Algorithm~\ref{algo:Semi-disc--ImplicitEuler}) --- for solenoidal additive noise (left: Example~\ref{ex:additive-solenoidal}) and non-solenoidal additive noise (right: Example~\ref{ex:add-generic}).
    } 
    \label{fig:pressureStability}
\end{figure*}

This first estimation \emph{seems} clearly sub-optimal for numerical analysts used to work on {\em deterministic} fluid flow problems -- but in the stochastic setting it is a result from the interplay of noise and pressure term; in fact, it is identity~\eqref{druck-1} that shows as soon as the noise coefficient has a non-trivial gradient contribution, then the pressure measured without being time-integrated is singular in the time-discretization parameter, and, as such, it would blow up for vanishing time step sizes. In contrast, notice that an estimate which is uniform in $\tau$ is only possible if the {\em noise is solenoidal}; see Figure~\ref{fig:pressureStability} for a related simulation. This difference is not shared if the time-integrated pressure is considered there.

\begin{framed}
The implicit Euler method focuses on the stability of velocity and pressure, and the exact match of the incompressibility constraint for the approximate velocity. As an implicit method, it favours stability over computational complexity. Thus, in general computations require more time but are more stable and accurate in comparison to other methods.  
\end{framed}

Since in the stochastic setting not only time and space but also probability have to be discretized, reducing the computational complexity of the time discretization is a much appreciated asset of any algorithm. Next, we discuss two algorithms that address this issue. As always, the limited regularity of the pressure has to be taken into account in their construction and analysis. 

\subsubsection{Chorin method}\label{os1}
Splitting methods, which trace back to Chorin~\cite{Chorin1968} and Temam~\cite{Tmam1969} for the deterministic (Navier--)Stokes equations, lead to the desired reduction. In~\cite{CHP1}, the Chorin method was adapted to the stochastic Stokes equations. It computes {\em first} an artificial velocity which does not need to be incompressible and hence is easier to compute; and then, in a {\em second} step, it computes a Helmholtz decomposition to correct the artificial velocity by a pressure gradient to recover an incompressible vector field -- the new velocity. More formally, the scheme is defined as follows:
\begin{framed}
\begin{algorithm}[Chorin method] \label{algo:semi-discrete_Chorin}
Let $n\geq 0$. Given the previous velocity and artificial velocity $({\bf u}^{n},\widetilde{\bf u}^{n})$, compute the next velocity, artificial velocity, and pressure $({\bf u}^{n+1},\widetilde{\bf u}^{n+1}, p^{n+1})$ in two steps:
\begin{itemize}
\item[{\bf (A)}] First, compute $\widetilde{\bf u}^{n+1}$ from
\begin{subequations} \label{def:Chorin}
\begin{equation}\label{def:Chorin-first}
\bigl( \widetilde{\bf u}^{n+1} - {\bf u}^n\bigr) - \tau \bigl(\Delta \widetilde{\bf u}^{n+1} - {\bf f}^{n+1}\bigr) = \pmb{\sigma}\bigl(t_n, { \widetilde{\bf u}^{n}}\bigr)\Delta_n W \qquad \mbox{on }  \Omega \times D.
\end{equation}
\item[{\bf (B)}] Then, use $\widetilde{\bf u}^{n+1}$ to compute ${\bf u}^{n+1}$ and $p^{n+1}$ from
\begin{eqnarray}\label{def:Chorin-second} 
{\bf u}^{n+1} - \widetilde{\bf u}^{n+1} + \tau \nabla p^{n+1} &=& 0  \qquad \qquad\qquad\qquad  \mbox{on }  \Omega \times D, \\  \label{def:Chorin-third} 
{\rm div}\, {\bf u}^{n+1} &=& 0 \qquad \qquad\qquad\qquad \mbox{on } \Omega \times D.
\end{eqnarray}
\end{subequations}
\end{itemize} 
\end{algorithm}
\end{framed}
If we apply the divergence operator in (\ref{def:Chorin-second}) and use~\eqref{def:Chorin-third}, we then arrive at the following elliptic problem for the pressure:
\begin{equation}\label{stokes-4c}
-\Delta p^{n+1} = -\frac{1}{\tau} {\rm div}\, \widetilde{\bf u}^{n+1}\,;
\end{equation}
since then we are provided with $\widetilde{\bf u}^{n+1}$ and $p^{n+1}$, we 
obtain the solenoidal vector field ${\bf u}^{n+1}$ by the simple update~\eqref{def:Chorin-second}. Now, related computations are decoupled into successively solving $d+1$ Poisson-type problems per time step: in ${\bf (A)}$ we compute {\em components of} the new iterate
$\widetilde{\bf u}^{n+1}$ independent from each other, and {\bf (B)} amounts to solving (\ref{stokes-4c}). 

This gain in efficiency of the method sacrifices a discrete energy bound
(\ref{disc_ener0}) that is available for (\ref{stokes-2}). As is well-known, this splitting method of Chorin converges with the same rate as the (decoupled) implicit Euler method does in the {\em deterministic setting}; see {\em e.g.}~\cite{Prohl1997, Prohl2009}. The reason for it is the availability of sufficient regularity of the time derivative of the (gradient of the) pressure $\nabla \partial_t p$, which may be concluded from a corresponding property of the velocity field, so that the {\em decoupled} usage of pressure and velocity in the splitting scheme can be controlled. To see this in more detail, we 
shift backwards in time equation (\ref{def:Chorin-second}) and then insert it into (\ref{def:Chorin-first}), which leads to the following system with the term $\nabla \partial_t p$ appearing in discrete form as leading term on the right-hand side of the following equations for $(\widetilde{\bf u}^{n+1}, p^{n+1})$ (we denote by $d_t  {\varphi}^{n+1} = \frac{ {\varphi}^{n+1} -  {\varphi}^{n}}{\tau}$ scaled differences of iterates):
\begin{subequations} \label{eq:chorin-re0}
\begin{eqnarray}\label{eq:chorin-re1}
d_t \widetilde{\bf u}^{n+1}  - \Delta \widetilde{\bf u}^{n+1} + \nabla p^{n+1} &=&
\tau \nabla d_t p^{n+1} +
{\bf f}^{n+1} + \frac{1}{\tau}\pmb{\sigma}\bigl(t_n, { \widetilde{\bf u}^{n}}\bigr)\Delta_n W\,, \\ \label{eq:chorin-re2}
{\rm div}\, \widetilde{\bf u}^{n+1}  - \tau \Delta p^{n+1} &=& 0.
\end{eqnarray}
\end{subequations}
Again, $\pmb{\sigma} = \pmb{0}$ represents the deterministic case, and so the error analysis for Chorin's method in \cite{Prohl1997} rests on its re-interpretation as a semi-implicit pressure stabilization method, where the incompressibility constraint in~\eqref{eq:chorin-re2} is replaced by the \textit{quasi-compressibility} constraint:
\begin{equation}\label{pertu1}
{\rm div} {\bf u} - \varepsilon \Delta p = 0 \qquad \mbox{on } D \qquad \mbox{with} \qquad \varepsilon = {\mathcal O}(\tau)\, .
\end{equation}

\begin{figure*}[t!]
    \centering
    \begin{subfigure}[t]{0.5\textwidth}
        \centering
        \includegraphics[width=1.0\textwidth]{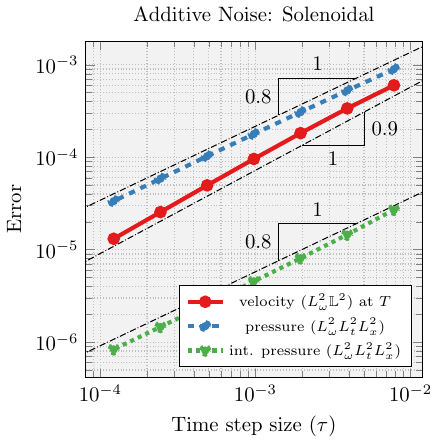}
    \end{subfigure}%
    ~ 
    \begin{subfigure}[t]{0.5\textwidth}
        \centering
        \includegraphics[width=1.0\textwidth]{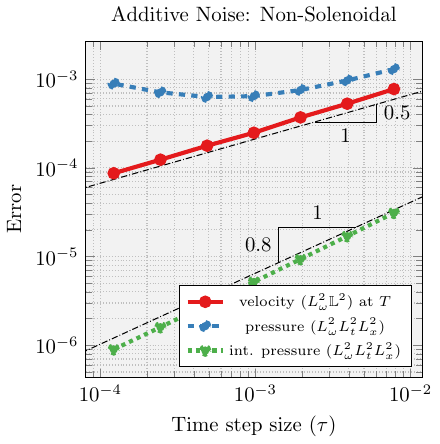}
    \end{subfigure}
    \caption{Evolution of velocity error, pressure error and time-integrated (abbreviated by int.) pressure error --- generated by the Chorin scheme (Algorithm~\ref{algo:semi-discrete_Chorin}) --- for solenoidal additive noise (left: Example~\ref{ex:additive-solenoidal}) and non-solenoidal additive noise (right: Example~\ref{ex:add-generic}).} 
    \label{fig:failure_chorin_inOne}
\end{figure*}

In the stochastic case ({\em i.e.}, $\pmb{\sigma} \neq \pmb{0}$) however, the {\em lack} of corresponding (uniform) stability bounds for the pressure for (\ref{stokes-2})  in
Lemma \ref{lem:pressure-stability}
now prevents an optimal stability {\em resp.}~error estimate for the stochastic Chorin method (see also part two of \cite[Lemma 3.1]{FV2022}) -- where to estimate perturbation effects of the incompressibility constraint, and the splitting-based explicit treatment of pressure iterates are the relevant sources of discretisation errors. In fact, the computational studies in \cite{CHP1} as well as 
\cite[Fig.~1 (left)]{FV2022}  for general ${\mathbb L}^2$-valued noise evidence a reduction of the rate in (\ref{stokes-3}) for the velocity iterates $\{ \widetilde{{\bf u}}^{n}\}_{n=1}^N$ from~\eqref{def:Chorin} to ${\mathcal O}(\tau^{1/4})$; and their theoretical roots
are identified in the convergence analysis in the proof of \cite[Lemma 3.2]{FV2022}. And it is not only that the 
rate of convergence for the velocity iterates deteriorate as indicated: also the time-integrated pressure iterates are affected, deteriorating to an order ${\mathcal O}(\tau^{1/4})$.

While corresponding convergence plots may be found in the cited works for multiplicative noise, we here provide related plots with {\em additive} noise; see Figure~\ref{fig:failure_chorin_inOne}. We observe a doubling of rates, and again a clear dependence of rates on the type of noise,
being either solenoidal or non-solenoidal; we are unaware of any related theoretical study which would explain the improved rates for the integrated pressure. However, the improved rates for the velocity are explained in~\cite{BrPr2} in a related (but different) setting. In addition, we observe that pressure iterates in the presence of non-solenoidal noise fail to converge, as illustrated in Figure~\ref{fig:failure_chorin_inOne} (right).

\begin{framed}
The Chorin method reduces the computational complexity by relaxing the incompressibility constraint for velocity. However, this relaxation has a price: the irregular pressure causes the velocity and pressure approximation errors to decay with a sub-optimal rate. Thus, approximations are fast to compute but lack accuracy. 
\end{framed}

\subsubsection{Pressure-corrected Chorin method} \label{sec:pres-cor-Chorin}
Because of this sub-optimal performance of the original Chorin method in the stochastic setting, a modified version of it was proposed in~\cite{CHP1}. The main idea for this modification is to eliminate the \textit{irregular stochastic pressure} (\emph{i.e.}, the bottleneck for the regularity of the pressure, which in~\eqref{eq:reconstruct-pressure} is the fourth term on the right-hand side) in the splitting scheme. Since solenoidal noises do \textit{not} affect the pressure and hence its regularity, restricting the noise to its solenoidal part by invoking the Helmholtz--Leray projection is better suited for the splitting scheme. However, to approximate the correct equation, the projection needs to be reverted afterwards, which corresponds to a pressure correction. The algorithm reads as follows:

\begin{framed}
\begin{algorithm}[pressure-corrected Chorin method] \label{algo:semi-discrete_ChorinCorrected} Let $n\geq 0$. Given the previous velocity and artificial velocity $({\bf u}^{n},\widetilde{\bf u}^{n})$, compute the next velocity and pressure $({\bf u}^{n+1}, p^{n+1})$, as well as the artificial velocity and pressure $(\widetilde{\bf u}^{n+1}, \tilde{p}^{n+1})$ in three steps:
\begin{itemize}
\item[{\bf (A')}] First, compute $\widetilde{\bf u}^{n+1}$ from
\begin{subequations} \label{def:mod-Chorin}
\begin{equation}\label{def:mod-Chorin-first}
\bigl( \widetilde{\bf u}^{n+1} - {\bf u}^n\bigr) - \tau \bigl(\Delta \widetilde{\bf u}^{n+1} - {\bf f}^{n+1}\bigr) = P_{\tt HL}\left[ \pmb{\sigma}\bigl(t_n, { \widetilde{\bf u}^{n}}\bigr)\Delta_n W \right] \qquad \mbox{on }  \Omega \times D.
\end{equation}
\item[{\bf (B')}] Second, use $\widetilde{\bf u}^{n+1}$ to compute ${\bf u}^{n+1}$ and $\tilde{p}^{n+1}$ from
\begin{eqnarray}\label{def:mod-Chorin-second-a}
{\bf u}^{n+1} - \widetilde{\bf u}^{n+1} + \tau \nabla \tilde{p}^{n+1} &=& 0  \qquad \qquad\qquad\qquad  \mbox{on }  \Omega \times D, \\  \label{def:mod-Chorin-second-b}
{\rm div}\, {\bf u}^{n+1} &=& 0 \qquad \qquad\qquad\qquad \mbox{on } \Omega \times D.
\end{eqnarray}
\item[{\bf (C')}] Finally, use $\tilde{p}^{n+1}$ to compute ${p}^{n+1}$ from
\begin{equation}\label{def:mod-Chorin-third}
 \nabla p^{n+1} = \nabla \tilde{p}^{n+1} + \frac{1}{\tau}  P_{\tt HL}^\perp\left[ \pmb{\sigma}\bigl(t_n, { \widetilde{\bf u}^{n}}\bigr)\Delta_n W \right] \qquad \, \,\mbox{on }  \Omega \times D.
\end{equation}
\end{subequations}
\end{itemize} 
\end{algorithm}
\end{framed}

\begin{figure*}[t!]
    \centering
    \begin{subfigure}[t]{0.5\textwidth}
        \centering
        \includegraphics[width=1.0\textwidth]{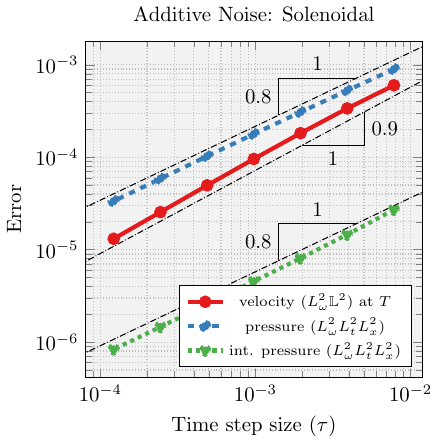}
    \end{subfigure}%
    ~ 
    \begin{subfigure}[t]{0.5\textwidth}
        \centering
        \includegraphics[width=1.0\textwidth]{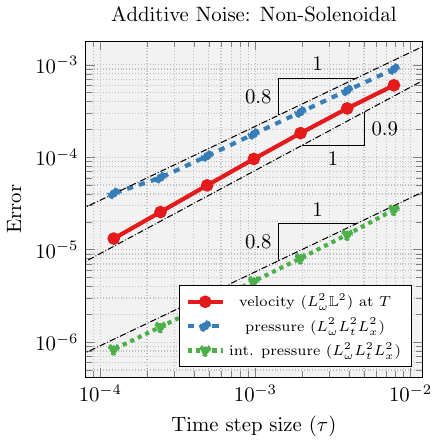}
    \end{subfigure}
    \caption{Evolution of velocity error, pressure error and time-integrated (abbreviated by int.) pressure error  --- generated by the pressure-corrected Chorin scheme (Algorithm~\ref{algo:semi-discrete_ChorinCorrected}) --- for solenoidal additive noise (left: Example~\ref{ex:additive-solenoidal}) and non-solenoidal additive noise (right: Example~\ref{ex:add-generic}).} 
    \label{fig:restoreChorin_inOne}
\end{figure*}

This algorithm relies on the Helmholtz--Leray projection and its complementary projection and hence, they must be computed. To achieve this, it is sufficient to  solve a Poisson problem in a preprocessing step; \emph{i.e.}, introducing $p^{n+1}_{\tt stoch}$ as the solution to
\begin{equation*}
\Delta p^{n+1}_{\tt stoch} = \Div \left[ \pmb{\sigma}\bigl(t_n, { \widetilde{\bf u}^{n}}\bigr)\Delta_n W\right],
\end{equation*}
allows the identifications:
\begin{subequations} \label{eq:pre-modChorin}
\begin{eqnarray} \label{eq:pre-modChorin-a}
P_{\tt HL} \left[ \pmb{\sigma}\bigl(t_n, { \widetilde{\bf u}^{n}}\bigr)\Delta_n W \right] &=&  \pmb{\sigma}\bigl(t_n, { \widetilde{\bf u}^{n}}\bigr)\Delta_n W - \nabla p^{n+1}_{\tt stoch}, \\ \label{eq:pre-modChorin-b}
P_{\tt HL}^\perp \left[ \pmb{\sigma}\bigl(t_n, { \widetilde{\bf u}^{n}}\bigr)\Delta_n W \right] &=& \nabla p^{n+1}_{\tt stoch}.
\end{eqnarray}
\end{subequations}
We then use~\eqref{eq:pre-modChorin-a} and~\eqref{eq:pre-modChorin-b} to compute the right-hand side of~\eqref{def:mod-Chorin-first} and~\eqref{def:mod-Chorin-third}, respectively. We note that this preprocessing step does not increase the complexity of the scheme in a significant way, since it requires to solve another Poisson problem only.

The computational studies for this algorithm in~\cite{CHP1} evidence its optimal convergence; a rigorous mathematical derivation of the convergence was proven in~\cite{FV2022} only later (see also~\cite{FPV1} for related results): it was shown that the approximate velocity and approximate time-integrated pressure of this modified Chorin method converge optimally (\emph{i.e.}, convergence with rate $\frac{1}{2}$) towards the analytic velocity and time-integrated pressure respectively. 

It is atep~{\bf (C')} which gives the algorithm its name, since the pressure is decomposed into the \textit{more regular artificial pressure} and the \textit{irregular stochastic pressure}: $p^{n+1} = \tilde{p}^{n+1} + \frac{1}{\tau} p^{n+1}_{\tt stoch}$. The regularity respective irregularity of each term is attested by their different scaling in $\tau$. Similar decompositions of pressure have been used in different settings for numerical purposes; see, \emph{e.g.},~\cite{BD1,CHP1, FPV1}.

Since only the artificial pressure influences the velocity and the irregular stochastic pressure is dismissed, it enables a uniform stability estimate (see next lemma), which is in sharp contrast to the singular behaviour of the pressure generated by the original Chorin method (see equation~\eqref{druck-1}). 

\begin{lemma}[artificial pressure stability] \label{lem:art-pres-stab}
It holds $\mathbb{P}$-a.s.: $\tau \sum_{n=1}^N \lVert \nabla \tilde{p}^{n}  \rVert_{\mathbb{L}^2}^2 \leq C$.
\end{lemma}
\begin{proof}
Using equation~\eqref{def:mod-Chorin-second-a}, we rewrite equation~\eqref{def:mod-Chorin-first} into an equation for $(\bfu^{n+1}, \tilde{p}^{n+1})$ only: 
\begin{equation}\label{def:mod-Chorin-first-alternative}
\bigl( \bfu^{n+1} - {\bf u}^n\bigr) - \tau \bigl(\Delta \bfu^{n+1} - \nabla \tilde{p}^{n+1} - {\bf f}^{n+1}\bigr) - \tau^2 \Delta \nabla \tilde{p}^{n+1} = P_{\tt HL}\left[ \pmb{\sigma}\bigl(t_n, { \widetilde{\bf u}^{n}}\bigr)\Delta_n W \right] \qquad \mbox{on }  \Omega \times D.
\end{equation}
Since we intend to argue pathwisely, let $\omega \in \Omega$ be arbitrary. Next, we multiply~\eqref{def:mod-Chorin-first-alternative} by $\nabla \tilde{p}^{n+1}(\omega)$, integrate in space and integrate by parts. Notice that all terms that depend on $\bfu^{n+1}$ and $\bfu^{n}$ vanish due to~\eqref{def:mod-Chorin-second-b}. Similarly, the projected noise vanishes since it is divergence free. Therefore, we arrive at
\begin{equation*}
\tau \lVert \nabla \tilde{p}^{n+1}  \rVert_{\mathbb{L}^2}^2 + \tau^2 \lVert \nabla^2 \tilde{p}^{n+1}  \rVert_{\mathbb{L}^2}^2  = \tau \left( \bff^{n+1}, \nabla \tilde{p}^{n+1} \right).
\end{equation*}
The claimed assertion now follows from Cauchy-Schwarz inequality, Young's inequality and summing over the time index~$n$.
\end{proof}
The improved pressure stability enables an error analysis based on the semi-implicit formulation of the pressure-corrected Chorin method, which reads as follows for the artificial velocity and pressure~$(\widetilde{\bf u}^{n+1}, \tilde{p}^{n+1})$:
\begin{subequations} \label{eq:reform-mod-Chorin0}
\begin{eqnarray}\label{eq:reform-mod-Chorin-1}
d_t \widetilde{\bf u}^{n+1}  - \Delta \widetilde{\bf u}^{n+1} + \nabla \tilde{p}^{n+1} &=&\tau  \nabla d_t \tilde{p}^{n+1}
 + {\bf f}^{n+1} + \frac{1}{\tau} P_{\tt HL} \bigl[\pmb{\sigma}\bigl(t_n,  \widetilde{\bf u}^{n}\bigr)\Delta_n W\bigr]\,, \\ \label{eq:reform-mod-Chorin-2}
{\rm div}\, \widetilde{\bf u}^{n+1} - \varepsilon \Delta \tilde{p}^{n+1} &=& 0 \, , \qquad \text{ where } \varepsilon = {\mathcal O}(\tau)\, . 
\end{eqnarray}
\end{subequations}
Again, we emphasize that -- different to~\eqref{eq:chorin-re0} -- now the formulation employs the more regular artificial pressure~$\tilde{p}^{n+1}$ only; and, thanks to its improved stability (see Lemma~\ref{lem:art-pres-stab}), the pressure-dependent term on the right-hand side may now be handled efficiently in an error analysis, leading to the optimal convergence of the scheme; see~\cite{FV2022} for further details.

\begin{framed}
The pressure-corrected Chorin method distinguishes between the regular artificial pressure and the irregular stochastic pressure:
the latter is removed from the discrete dynamics via the Helmholtz projection of the noise in a pre-processing step. Consequently, the error sources in the operator-splitting depend on the artificial pressure only, leading to optimal rates of convergence.
\end{framed}

\begin{remark}[The special role of \textbf{additive noise}]\label{rem:3.7}
A special case is the presence of additive noise, \textit{i.e.}, $\bfsigma(\cdot,\bfu) \equiv \bfk$ for some given vector field~$\bfk$ that depends on time and space. In this case, we have seen that (recall Figures~\ref{fig:failure_chorin_inOne} and~\ref{fig:restoreChorin_inOne}) the convergence rate for the velocity and pressure errors increases. This is motivated by the following observation:

Only for additive noise, we can use the transformation $\bfy = \mathbf{u} - \mathbf{\Xi}$, where $\mathbf{\Xi}(t) = \int_0^t \bfk(s) \dd W(s)$, to re-write the stochastic Stokes equations (\emph{cf.} Equations~\eqref{def:gen-stokes}) into the following \textbf{random Stokes equations}:  
\begin{framed}
\begin{subequations} \label{eq:random-stokes}
\begin{eqnarray} \label{eq:random-stokes1}
\partial_t \bfy - \bigl( \mu \Delta {\bfy} - \nabla p \bigr) &=&
\bff + \Delta \mathbf{\Xi}
\qquad  \qquad \mbox{on } \Omega \times (0,T) \times D,\\  \label{eq:random-stokes2}
{\rm div} \, {\bfy} &=& -\Div \mathbf{\Xi} \qquad \qquad \, \mbox{on } \Omega \times (0,T) \times D, \\ 
{\bfy}(0) &=& {\bf u}_0 \qquad \qquad \qquad \,\,\,  \mbox{on } \Omega  \times D.
\end{eqnarray}
\end{subequations}
\end{framed}
In contrast to the stochastic Stokes equations where one seeks $\bfu$ and $p$ as the solution to an integral equation forced by a stochastic integral; the random Stokes equations define the evolution of $\bfy$ and $p$ in terms of differential equations. Hereby, generic additive noise influences not only the evolution equation of~$\bfy$ but also its incompressibility constraint. But for solenoidal noise the constraint is not affected.

Now, not only is it possible to use PDE arguments for the analysis of $\bfy$ (and $p$), but classical numerical algorithms can be utilized for their approximation as well; see \emph{e.g.},~\cite{BrPr2} in which this technique has been used for the stochastic Navier--Stokes equations. For example, the implicit Euler method approximates $\bfy$ at rate $1$. Then, reverting the transformation provides an approximation of the original velocity field~$\bfu$. Provided that the additive noise (\emph{i.e.}, $\mathbf{\Xi}$) can be evaluated exactly, the original variable is approximated with rate $1$ as well -- which is twice as fast as for multiplicative noise.

We may also use this reformulation as the starting point of a numerical analysis to settle the improved convergence rates for the Chorin method and the pressure-corrected Chorin method, as we have observed in Figure~\ref{fig:failure_chorin_inOne} and Figure~\ref{fig:restoreChorin_inOne}, respectively -- whose theoretical justification is still open. 
\end{remark}

\subsubsection{Crank--Nicolson method}
\label{sec:3.3.4}
The Crank--Nicolson method~\cite{Crank1947} was originally introduced to approximate deterministic equations of heat-conduction type. As an equilibrated method, it does not introduce any numerically-induced energy dissipation nor energy concentration, enabling accurate long-term simulations. Moreover, at least for deterministic equations and sufficiently regular solutions, it is second order accurate in time, doubling the convergence rate in comparison to, \emph{e.g.}, the implicit Euler scheme. Due to the equilibration and the increased convergence speed, it is nowadays widely used for the approximation of hyperbolic conservation laws; see, \emph{e.g.},~\cite{Kadalbajoo2006,Pini2001}. Especially, its discrete energy identity (\emph{resp.} energy inequality for nontrivial body forces) makes it a perfect choice for transport-type equations and, as such, for transport noise. While in the stochastic case the low regularity of solutions no longer enables increased convergence rates, the method still balances the energy on the discrete level. Before we discuss this in more detail, we introduce the scheme in the context of transport noise in Stratonovich form: 

\begin{framed}
\begin{algorithm}[Crank--Nicolson] \label{algo:Semi-disc--CN}
Let $n\geq 0$. Given the previous velocity ${\bf u}^{n}$, compute the next velocity and pressure $({\bf u}^{n+1}, p^{n+1})$ from
\begin{subequations}\label{algo:CN}
\begin{eqnarray}\label{algo:CN1}
\bigl({\bf u}^{n+1} - {\bf u}^n\bigr) - \tau \bigl(\Delta {\bf u}^{n+\frac{1}{2}} - \nabla p^{n+1} - {\bf f}^{n+1}\bigr) &=& (\pmb{\sigma} \cdot \nabla) {\bf u}^{n+\frac{1}{2}}\Delta_n W\quad  \mbox{on }  \Omega \times D, \\  \label{algo:CN2}
{\rm div}\, {\bf u}^{n+1} &=& 0 \,\qquad \qquad\qquad\qquad \mbox{on } \Omega \times D,
\end{eqnarray}
\end{subequations}
where $\bfu^{n+\frac{1}{2}} := \frac{\bfu^{n+1} + \bfu^n}{2}$ is the arithmetic mean of previous velocity and next velocity.
\end{algorithm}
\end{framed}
It is the type of stochastic integration that determines the quadrature rule for its discretization: the It\^o integral must be discretized by approximating the integrand in an adapted manner (\emph{i.e.}, its evaluation time must be prior to the time that defines the stochastic increment), whereas the direct approximation of the Stratonovich integral requires the {\em midpoint rule}. At the continuous level, the It\^o and Stratonovich integrals, if evaluated at the solution~$\bfu$, are related by the following identity (for details, see, \emph{e.g.},~\cite{MiRo}):
\begin{equation}\label{strato-1}
\int_0^t (\pmb{\sigma}  \cdot \nabla){\bf u}\, {\rm d}W  = 
\int_0^t (\pmb{\sigma}  \cdot \nabla){\bf u} \circ {\rm d}W  - \frac{1}{2} \int_0^t  (\pmb{\sigma}  \cdot \nabla) (\pmb{\sigma}  \cdot \nabla) {\bf u} \,{\rm d}s ,
\end{equation}
for all $0 \leq t \leq T$ and $\mathbb{P}$-a.s. Discretizing this formula provides an alternative approach for the approximation of the Stratonovich integral: instead of approximating the Stratonovich integral directly, one first approximates the It\^o integral and the It\^o--Stratonovich corrector separately; then, one reconstructs the Stratonovich integral from~\eqref{strato-1}. For example, a good candidate for this approach is the following:
\begin{equation} \label{eq:alternative-Strato}
\int_{t_{n}}^{t_{n+1}}(\bfsigma \cdot \nabla) \bfu \circ \dd \bfW \quad \approx \quad (\pmb{\sigma} \cdot \nabla) {\bf u}^{n}\Delta_n W + \tau \,\frac{1}{2} (\pmb{\sigma}  \cdot \nabla) (\pmb{\sigma}  \cdot \nabla) {\bf u}^{n+1}.
\end{equation}
At least formally, after summing up~\eqref{eq:alternative-Strato} and the right-hand side of~\eqref{algo:CN1}, both approximate the same object asymptotically, since, for sufficiently smooth solutions,
\begin{equation}
\mathbb{E} \left[ \left| \sum_{n=1}^N \left( (\pmb{\sigma} \cdot \nabla) {\bf u}^{n}\Delta_n W + \tau\,\frac{1}{2} (\pmb{\sigma}  \cdot \nabla) (\pmb{\sigma}  \cdot \nabla) {\bf u}^{n+1} \right) - \sum_{n=1}^N(\pmb{\sigma} \cdot \nabla) {\bf u}^{n+\frac{1}{2}}\Delta_n W \right| \right] = \mathcal{O}(\sqrt{\tau}).
\end{equation}

We choose to present the Crank--Nicolson scheme with the direct approximation of the Stratonovich integral, as this allows a \textit{pathwise energy identity}, which we will derive below. However, if one is interested in stiff problems, the additional dissipation induced by the corrector might be advantageous. But this comes at a price: the energy is no longer preserved on each path.

As already emphasized, it is the unconditional stability of the scheme that distinguishes the Crank--Nicolson scheme from others. Next, we present its reason, for which we assume that $\{{\bf f}^n\}_{n=1}^N \equiv {\bf 0}$; otherwise, the next equality has to be replaced by an inequality.   

\begin{lemma}[pathwise energy identity]\label{lem:pathwise-energy-stability} For all $0 \leq n \leq N$, it holds~$\mathbb{P}$-a.s.
\begin{equation}\label{discr_energy2s}
 {\mathscr E}({\bf u}^n)  +   \tau \sum_{\ell=0}^n \Vert \nabla {\bf u}^{\ell+1/2}\Vert^2_{{\mathbb L}^2}  =  {\mathscr E}({\bf u}^0)\, ,
 \end{equation}
 where the energy functional is given by ${\mathscr E}({\bf v}) = \frac{1}{2} \Vert {\bf v}\Vert^2_{{\mathbb L}^2}$ ({\em i.e.}, the kinetic energy).
\end{lemma} 

\begin{proof}
Since we argue on the level of individual paths, we first fix $\omega \in \Omega$. Next, we multiply~\eqref{algo:CN1} by ${\bf u}^{n+1/2}(\omega)$ -- which is different to (\ref{discr_energy1}) where ${\bf u}^{n+1}(\omega)$ was used -- and integrate in space; the binomial formula now leads to
$$ \bigl( {\bf u}^{n+1} - {\bf u}^n, \frac{1}{2}[{\bf u}^{n+1} + {\bf u}^{n}] \bigr) = \frac{1}{2}\bigl( \Vert {\bf u}^{n+1}\Vert^2_{{\mathbb L}^2} - \Vert {\bf u}^n\Vert^2_{{\mathbb L}^2}\bigr).$$
There is no further term $\frac{1}{2} \Vert {\bf u}^{n+1} - {\bf u}^n\Vert^2$ as in (\ref{discr_energy1}) which corresponds to numerical energy dissipation. For multiplicative noise, this term is commonly used to eventually control the discretization of the It\^o integral; see the discussion around~\eqref{discr_energy1}. But this is not necessary for transport noise that we deal with here, since integration by parts and ${\rm div}\, \pmb{\sigma} = 0$ imply that
\begin{equation*}
\left( (\pmb{\sigma} \cdot \nabla) {\bf u}^{n+\frac{1}{2}}\Delta_n W, {\bf u}^{n+\frac{1}{2}}\right) = - \int_{D} \Div \bfsigma \frac{\abs{\bfu^{n+\frac{1}{2}}}^2}{2} \dd x \, \Delta_n W =0. 
\end{equation*}
The claimed pathwise energy identity follows immediately. 
\end{proof}

\begin{figure*}[t!]
    \centering
    \begin{subfigure}[t]{0.5\textwidth}
        \centering
        \includegraphics[width=1.0\textwidth]{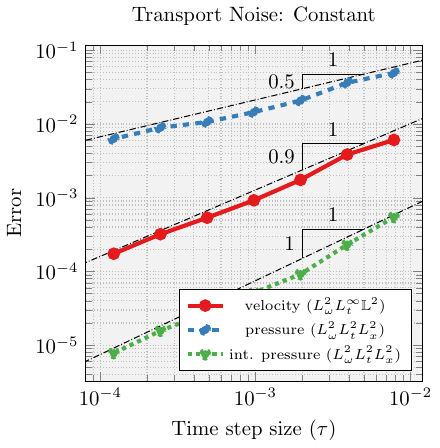}
    \end{subfigure}%
    ~ 
    \begin{subfigure}[t]{0.5\textwidth}
        \centering
        \includegraphics[width=1.0\textwidth]{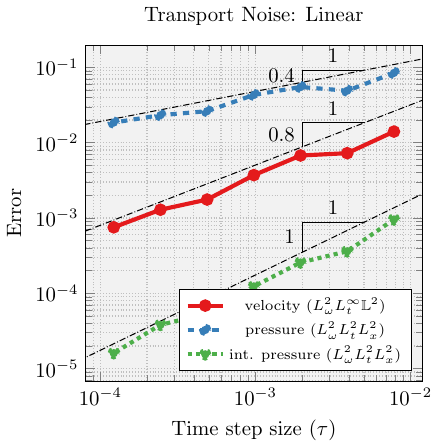}
    \end{subfigure}
    \caption{Evolution of velocity error, pressure error and time-integrated (abbreviated by int.) pressure error  --- generated by a modified Crank--Nicolson scheme (Algorithm~\ref{algo:Semi-disc--CN}), where the Laplacian is evaluated fully implicitly instead of semi-implicitly --- for constant transport noise (left: Example~\ref{ex:transportConstant}) and linear transport noise (right: Example~\ref{ex:transportLinear}).} 
    \label{fig:transportCompare}
\end{figure*}

An advantage of the Crank--Nicolson scheme is that it extends easily to, \emph{e.g.}, generalized Stokes equations, which are used for the modelling of non-Newtonian fluids. For example, in~\cite{2024arXiv241214316D} it is shown that the Crank--Nicolson scheme in combination with a generic spatial discretisation (called Gradient Discretization Method; \emph{cf.}~\cite{Droniou2018}) satisfies an energy inequality similar to Lemma~\ref{lem:pathwise-energy-stability} for the generalised Stokes equations, imposed on a general domain and complemented by non-trivial Dirichlet boundary conditions.

In contrast to additive and multiplicative noises, the numerical investigation of transport noise in fluid equations has only started recently. The only result we are aware of is~\cite{BMPW} in which the authors address a modified Crank--Nicolson scheme (\emph{i.e.}, the Laplacian is evaluated in a fully implicit manner rather than at the midpoint) for the Navier--Stokes equations: they show that the scheme converges optimally in time with rate $\frac{1}{2}$. We remark that the computational studies in Figure~\ref{fig:transportCompare} suggest twice this order in cases where $\bfsigma$ is special, and a theoretical explanation is open at the time. Another interesting, open problem is how an optimally convergent time-splitting scheme (\emph{e.g.}, a version of the pressure-corrected Chorin method) would look like to reduce the computational effort in the coupled Algorithm~\ref{algo:Semi-disc--CN}: the problem here is the Stratonovich interpretation of the noise. Since it depends on the unknown velocity at the new time, the noise cannot be projected in a preprocessing step, which is in conflict to the strategy used for the pressure-corrected Chorin method!

\begin{framed}
The Crank--Nicolson method focuses on the preservation of energy on the discrete level. For transport noise, it is unconditionally stable and, in contrast to other methods, this stability holds pathwisely. Its computational complexity and accuracy (at least for the approximate velocity) are similar to those of the implicit Euler method.  
\end{framed}


\section{Numerics of the stochastic Navier--Stokes equations} \label{sec:Num-NSE}

In the previous section we discussed numerical problems which need to be overcome to arrive at optimally convergent schemes for the stochastic Stokes equations with different driving noise. Additional problems arise when provably convergent schemes for the stochastic Navier-Stokes equations are our aim, and we will address some prominent ones in this section.
We survey the known results for the two-dimensional problem and outline the additional complications if compared to the Stokes system. Finally, we list some open problems.

\subsection{Model and theoretical results}
In this section we consider the stochastic Navier--Stokes equations in two dimensions: 
\begin{framed}
\begin{subequations} \label{def:gen-navierstokes}
\begin{eqnarray} \label{eq:gen-navierstokes1}
{\rm d}{\bf u}+ \bigl( \bfu\cdot\nabla\bigr)\bfu\,{\rm d}t - \bigl( \mu \Delta {\bf u} - \nabla p \bigr){\rm d}t &=&
{\bf f} {\rm d}t  + \dd \bfZ_{\bfu}  
\qquad \mbox{on } \Omega \times (0,T) \times D,\\  \label{eq:gen-navierstokes2}
{\rm div} \, {\bf u} &=& 0 \qquad \qquad \qquad  \mbox{on } \Omega \times (0,T) \times D, \\
{\bf u}(0) &=& {\bf u}_0 \qquad \qquad \quad \,\,\,  \mbox{on } \Omega  \times D,
\end{eqnarray}
\end{subequations}
\end{framed}
which are supplemented by periodic boundary conditions. For the physical meaning of these quantities, we refer to Section~\ref{sec:model-and-theory}.

The obvious difference between the stochastic Stokes equations and the stochastic Navier--Stokes equations is the convective term $\bigl( \bfu\cdot\nabla\bigr)\bfu$, rendering the latter equations non-linear -- and including this term allows a much more accurate description of real fluid flows, especially those with \textit{high} Reynolds numbers (or, equivalently, small viscosities) and turbulent behaviour. 

In contrast to linear SPDEs for which there are many results on its numerical approximation available in the literature, the results for genuinely nonlinear equations are sparse. If the nonlinearity has additional structure (\emph{e.g.}, it is a \emph{semi-linear} equation with Lipschitz-continuous nonlinearity, or it is a nonlinear equation with a suitable \textit{monotone} nonlinearity), then numerical approximations behave qualitatively the same as in the linear case -- although this is non-trivial to prove. A proto-typical example in this regard is the stochastic Allen-Cahn equation, which was studied, \emph{e.g.}, in~\cite{MR3776047}.

However, one easily checks that the mapping $\bfu\mapsto \bigl( \bfu\cdot\nabla\bigr)\bfu$ is neither globally Lipschitz continuous nor has it any useful monotonicity properties. Hence the stochastic Navier--Stokes equations do not fall into the class of SPDEs as just described, indicating that new methods had to be developed:

 A first analysis of such genuinely nonlinear problems has been performed in \cite{Pr}, where instead of considering the mean-square error the \textit{error in probability} was introduced. It is based on the convergence in probability of random variables which is a weaker notion than that of convergence in mean-square; convergence in mean-square corresponds to the convergence of the expectation of a certain (squared) norm, whereas convergence in probability only means that the norm vanishes with arbitrary high probability, see Section \ref{sec:largesamplesets} below for more explanation on this matter.

Next, we review the literature on the numerical approximation of the stochastic Navier--Stokes equations. As for the linear case, 
we summarise the existing algorithms in the literature and their respective error decay rates in Table~\ref{table:NSE} to facilitate the comparison between different results as well as to enable the fast look-up of individual results.
The focus in this review is on the numerical analysis of the 2D problem, and related strong rates of convergence; for the 3D case we refer to \cite{BCP1,OPW1}.

\subsection{State of the art: Navier--Stokes}  \label{sec:state-NSE}
First convergence rates for a spatio-temporal approximation, generated by the implicit Euler scheme in combination with {\em exactly} divergence-free finite elements, of the 2D stochastic Navier--Stokes equations were proved in \cite{CP}. In particular, it was shown that the error (in probability) for the velocity is $\mathcal O(\tau^{\frac{1}{4}}+h)$.  

\smallskip
Based on a stochastic pressure decomposition on the one hand and more restrictive regularity assumptions on the data on the other hand this has been improved in \cite{BD1}
to 
$\mathcal O(\tau^{\frac{1}{2}}+h)$ -- where now general stable mixed finite elements are admitted. 

\smallskip
The paper \cite{BrDo2} contains similar results for the 3D problem; but since here the existence of a unique solution is only known up to its \textit{blow-up-time}\footnote{The blow-up-time is a random time when certain norms of the solution become unbounded. It is known that the blow-up-time is strictly positive with full probability but its particular value is unknown in general.}, the convergence rates of the error hold locally in time, only. In all three papers generic nonlinear multiplicative noise with appropriate growth assumptions have been considered.

\smallskip
In \cite{BrPr2}, a new temporal and spatio-temporal discretisation for the stochastic Navier--Stokes equations with additive noise was introduced, which is based on the reformulation of the equations as {\em random PDEs}, using the transform ${\bf y} = {\bf u} - \bfZ_{\bfu}$ (see also Remark~\ref{rem:3.7}). The motivation is to derive improved convergence rates of approximates for ${\bf y}$ instead of $\bfu$. Importantly, the new variable has improved time regularity properties; in particular, one can show that it is differentiable in time, while ${\bf u}$ is only H\"older continuous up to the continuity index~$\tfrac{1}{2}$. By using these improved regularity properties it was shown that the error in probability converges with rate up to $1$ in time --- which doubles the convergence order in comparison to the multiplicative case discussed above.

\smallskip
A parallel numerical program has been performed in \cite{BeMi1,BeMi2,BeMi2b}: these works address the question, if convergence results for the  mean-square-error instead of the error in probability may be obtained by exploiting improved stability properties of the solution, such as finiteness of exponential moments. It turns out that this is the case, at least in particular data configurations:


%
The first result in this regard is \cite[Thm.~4.6]{BeMi1}, which 
proves {\em mean-square convergence} for the error with a quantified rate. This rate depends on the ratio of the viscosity and final time, and the noise -- for a fixed noise intensity, and in comparison either small time horizons or large viscosities the rate can be arbitrary close to $\tfrac{1}{4}$. This applies to the time-splitting scheme from \cite{BeBrMi} as well as the implicit Euler scheme used in \cite{CP}.

Another result in this direction is
asymptotic mean-square convergence of order up to $\frac{1}{2}$ in \cite[Thm. 3.3]{BeMi2}; \emph{i.e.}, provided that the strength of the additive noise is of order ${O}(\mu)$ (recall that $\mu$ denotes the viscosity of the fluid),  it is shown that the space-time discretisation is
$\mathcal O(\tau^{\gamma}+h)$, for some viscosity dependent $\gamma=\gamma(\mu)$.
This result is based on exponential moments. Without a restriction on the viscosity
and for multiplicative noise a logarithmic convergence rate is proved in \cite{BeMi2b} which is also based on exponential moment estimates.

A different aspect comes into play when transport noise is considered. It is due to its particular structure that transport noise is energy conservative along every path (see, \emph{e.g.}, Lemma~\ref{lem:pathwise-energy-stability}), which 

\begin{landscape}
\begin{table}
\centering
\begin{tblr}{
  cells = {c},
  row{1} = {Silver},
  vlines,
  hline{1,10} = {-}{0.08em},
  hline{2-9} = {-}{},
}
\textbf{Reference } & \textbf{ Noise} & \textbf{Error} & 
\textbf{Space disc.}                                 & \textbf{Convergence rate (velocity)}                                                                                                                                                               & \textbf{Feature}                                                                                                                             & \textbf{Exp.} \\
\cite{CP}                & multiplicative        & $\mathbb P$       & exact div & $\tau^{\frac{1}{4}} + h$ & regular data & Yes \\  
\cite{BD1}              & multiplicative       & $\mathbb P$        & mixed FEM & $\tau^{\frac{1}{2}} + h$ & regular data & Yes                 
 \\  
\cite{BrPr2}              & additive       & $\mathbb P$          & mixed FEM & $\tau + h$ & regular data & Yes      \\
\cite{BeMi1}              & multiplicative      & $L^2(\Omega)$          & -- &  $\tau^\gamma$, $\gamma=\gamma($viscosity$)$ & large viscosity & No      \\
\cite{BeMi2}              & additive       & $L^2(\Omega)$          & mixed FEM &  $\tau + h^2$ & large viscosity & No      \\
\cite{BeMi2b}              & multiplicative       & $L^2(\Omega)$          & mixed FEM &  $|\log(\tau + h^2)|^{-q}$\,\,\,$\forall q$ & -- & No      \\
\cite{BeMi2b}              & additive      & $L^2(\Omega)$          & mixed FEM &  $\exp(-\sqrt{|\log(\tau + h^2)|})$ & -- & No  \\
           \cite{BMPW}              & transport       & $L^2(\Omega)$          & -- & $ \tau^{\frac{1}{2}}$ & large viscosity, constant fields & Yes      
\end{tblr}
\caption{This table lists the known results on the numerical approximation of the two-dimensional stochastic Navier--Stokes equations described in Section \ref{sec:state-NSE}. For the temporal discretisation an implicit or semi-implicit Euler scheme is used. In some cases only the temporal error is studied.
 We indicate by $\mathbb P$ convergence in probability and by $L^2(\Omega)$ convergence in mean square. In the latter case the velocity error is measured in $L^{2}_\omega L^{\infty}_t \mathbb{L}^2 + L^{2}_\omega L^2_t \mathbb{H}^1$. In the former case convergence in probability for the $L^{\infty}_t \mathbb{L}^2 + L^2_t \mathbb{H}^1$-velocity error is considered.}
\label{table:NSE}
\end{table}
\end{landscape}

%
%
%

continues to hold for spatial derivatives of the solution, if the transport noise is defined in terms of constant vector fields only. Using this pathwise control on the solution, it was derived in~\cite{BMPW} that the mean-square error of a modified Crank--Nicolson time discretization converges with the optimal convergence rate~$\tfrac{1}{2}$ in mean-square sense. However, the result required again a restriction between the intensity of the noise and the viscosity.

\smallskip
The results mentioned so far do not treat the convergence rate for the pressure function explicitly -- but typically, the time-integrated pressure error decays with the same rates as the velocity error: for example, in~\cite{Q} it was shown for an implicit Euler scheme that the error in probability of the time-integrated pressure is $\mathcal O(\tau^{1/2}+h)$. As in the linear case, convergence of the original (non-time-integrated) pressure error can only be expected if the noise is solenoidal such that the temporal regularity of the pressure is unaffected by the noise.

\begin{framed} 
So far, for generic multiplicative noise and additive noise the respective optimal convergence rates $\mathcal O( \sqrt{\tau}+h)$ and $\mathcal O( \tau+h)$ were obtained for the \emph{error in probability} only. While it is proven that the \emph{mean-square error} converges at least logarithmically, faster convergence (up to the optimal ones) is typically observed in numerical simulations. In the case of transport noise, the optimal rates {\em in mean-square sense} are justified by theory. 
\end{framed}

\subsection{The problem in the error analysis and how it is overcome}
\label{sec:largesamplesets}
As alluded to above, we generally only expect convergence for the error in probability for discretizations of the stochastic Navier--Stokes equations. The main issue is that the solution of \eqref{def:gen-navierstokes} inherits some of the Gaussian character of the driving Wiener process: arbitrary high absolute values/norms of the solution are possible with small probability. This is a feature which applies to many stochastic evolution problems even if they are linear or finite dimensional. In the deterministic case, for a spatially smooth solution one can control certain norms of the velocity (or even its gradient) and thus one can treat the nonlinear term similarly to that of a Lipschitz-continuous perturbation. This is not possible any more in the stochastic case, given the unboundedness of solutions as just explained. 
 
In the following, we detail the problem in the error analysis for an implicit Euler discretization that originates from the quadratic nonlinearity and results in the weaker concept of \emph{convergence in probability} for the Navier-Stokes equations
rather than mean square convergence for the Stokes equations. 

First, we generalise the implicit Euler algorithm (Algorithm~\ref{algo:Semi-disc--ImplicitEuler}) to the stochastic Navier-Stokes equations by discretising the convective term in a semi-implicit way, leading to the following algorithm:

\begin{framed}
\begin{algorithm} (implicit Euler -- S-NSE) \label{algo:IE-NSE}
Let $n\geq 0$. Given the previous velocity ${\bf u}^n$, compute the next velocity and pressure $({\bf u}^{n+1}, p^{n+1})$ from
\begin{subequations}
\begin{eqnarray}\label{oc1}
 {\bf u}^{n+1}-{\bf u}^n - \tau \bigl(\Delta {\bf u}^{n+1} -\nabla p^{n+1} - {\bf f}^{n+1}\bigr) \hspace*{-0.25em}&=& \hspace*{-0.25em}-\tau({\bf u}^n \cdot \nabla) {\bf u}^{n+1} +  {\bfsigma}(t_n, {\bf u}^n) \Delta_n W, \\
 \label{oc2}
{\rm div}\, {\bf u}^{n+1}\hspace*{-0.25em} &=& \hspace*{-0.25em}0 .
\end{eqnarray}
\end{subequations}
\end{algorithm}
\end{framed}

Next, by comparing two solutions $\{ {\bf u}^n_i;\, 0 \leq n \leq N\}_{i=1,2}$ with differing initial data ${\bf u}^0_i$, we illustrate the problem in the error analysis due to the quadratic nonlinearity. We denote the velocity and pressure differences by ${\bf e}_{\bf u}^{n+1} := {\bf u}^{n+1}_1 - {\bf u}^{n+1}_2$ and $e_{p}^{n+1} := p^{n+1}_1 - p^{n+1}_2$, respectively. Notice that the velocity difference is solenoidal. Moreover, the evolution equation of the velocity difference is given by: 
 \begin{eqnarray*}
{\bf e}_{\bf u}^{n+1} - {\bf e}_{\bf u}^{n}\hspace*{-0.6em} &-& \hspace*{-0.6em}\tau \Delta {\bf e}_{\bf u}^{n+1}  + \tau \nabla e_{p}^{n+1} \\
&=& \hspace*{-0.6em}- \tau  ({\bf e}^n_{\bf u} \cdot \nabla) {\bf u}^{n+1}_1 - \tau ({\bf u}^n_2 \cdot \nabla){\bf e}_{\bf u}^{n+1}  +  \bigl({\bfsigma}(t_n, {\bf u}^n_1) - {\bfsigma}(t_n, {\bf u}^n_2)\bigr) \Delta_n W.
\end{eqnarray*}
We test this identity with ${\bf e}_{\bf u}^{n+1}$ and use standard arguments for handling the left-hand side; in particular, the fourth term vanishes for periodic boundary conditions. For the right-hand side of the equality we proceed as follows: the Lipschitz nonlinearity ${\bfsigma}(t, \cdot)$ is dealt with as in~\eqref{discr_energy1}; the second term vanishes due to integration by parts and the incompressibility of~$\bfu_2^n$; the `trouble-maker' is the first term on the right-hand side, which needs be to estimated: 

By applying H\"older's inequality, interpolating ${\mathbb L}^4$ between ${\mathbb L}^2$ and ${\mathbb W}^{1,2}$ (here, we use that $d=2$) and using a weighted Young's inequality (for arbitrary~$\delta>0$) yield
\begin{eqnarray*}
\big| \left( ({\bf e}^n_{\bf u} \cdot \nabla) {\bf u}^{n+1}_1, {\bf e}^{n+1}_{\bf u} \right) \big| &\leq &\Vert {\bf e}^n_{\bf u}\Vert_{{\mathbb L}^4} \Vert 
\nabla {\bf u}^{n+1}_1\Vert_{{\mathbb L}^2} \Vert {\bf e}^{n+1}_{\bf u}\Vert_{{\mathbb L}^4} \\
&\leq &  \Vert {\bf e}^n_{\bf u}\Vert_{{\mathbb L}^2}^{\frac{1}{2}}  \Vert \nabla {\bf e}^n_{\bf u}\Vert_{{\mathbb L}^2}^{\frac{1}{2}} \Vert 
\nabla {\bf u}^{n+1}_1\Vert_{{\mathbb L}^2} \Vert {\bf e}^{n+1}_{\bf u}\Vert_{{\mathbb L}^2}^{\frac{1}{2}} \Vert \nabla {\bf e}^{n+1}_{\bf u}\Vert_{{\mathbb L}^2}^{\frac{1}{2}}\\
&\leq & c_\delta  \Vert 
\nabla {\bf u}^{n+1}_1\Vert_{{\mathbb L}^2}^2 \left( \Vert {\bf e}^n_{\bf u}\Vert_{{\mathbb L}^2}^2 + \Vert {\bf e}^{n+1}_{\bf u}\Vert_{{\mathbb L}^2}^2 \right) + \delta \left( \Vert \nabla {\bf e}^n_{\bf u}\Vert_{{\mathbb L}^2}^{2} +   \Vert \nabla {\bf e}^{n+1}_{\bf u}\Vert_{{\mathbb L}^2}^{2} \right).
\end{eqnarray*}
Combining this inequality with the previous steps, we arrive at the following inequality:
 \begin{eqnarray}\nonumber
&&\frac{1}{2} \bigl( \Vert {\bf e}^{n+1}_{\bf u}\Vert^2_{{\mathbb L}^2} - \Vert {\bf e}^{n}_{\bf u}\Vert^2_{{\mathbb L}^2} \bigr) -\delta \tau \Vert \nabla {\bf e}^{n}_{\bf u}\Vert^2_{{\mathbb L}^2} + (1-\delta) \tau \Vert \nabla {\bf e}^{n+1}_{\bf u}\Vert^2_{{\mathbb L}^2} \\ \label{oc3}
&&\qquad\qquad\qquad\qquad \leq \tau c_\delta  \Vert 
\nabla {\bf u}^{n+1}_1\Vert_{{\mathbb L}^2}^2 \left( \Vert {\bf e}^n_{\bf u}\Vert_{{\mathbb L}^2}^2 + \Vert {\bf e}^{n+1}_{\bf u}\Vert_{{\mathbb L}^2}^2 \right) +  \Vert {\bf e}^n_{{\bfsigma}}  \Delta_n W\Vert^2_{{\mathbb L}^2} + \bigl( {\bf e}^n_{{\bfsigma}} \Delta_n W, {\bf e}^n_{\bf u}\bigr)\,,\nonumber
\end{eqnarray}
where ${\bf e}^n_{{\bfsigma}} := {\bfsigma}(t_n, {\bf u}^n_1) - {\bfsigma}(t_n, {\bf u}^n_2)$. 

To obtain a mean-square estimate, the next canonical step is to apply expectations: for the second term on the right-hand side, the It\^o isometry and the Lipschitz-continuity of~$\bfsigma$ help to extract the time-step size from the Wiener increment and to estimate~$\bfe_{\bfsigma}^n$ in terms of $\bfe_{\bfu}^n$, respectively; due to the independence of the Wiener increment, the last term vanishes. It remains to estimate the `trouble-maker': this term is critical, since in general 
\begin{equation} \label{eq:TheDream}
{\mathbb E}\bigl[\Vert \nabla {\bf u}^{n+1}_1\Vert^2_{{\mathbb L}^2} \bigl(
\Vert {\bf e}^n\Vert^2_{{\mathbb L}^2} + \Vert {\bf e}^{n+1}\Vert^2_{{\mathbb L}^2}\bigr)\bigr] \neq {\mathbb E}\bigl[ \Vert \nabla {\bf u}^{n+1}_1\Vert^2_{{\mathbb L}^2}\bigr] \cdot {\mathbb E}\bigl[ \Vert {\bf e}^n\Vert^2_{{\mathbb L}^2} + \Vert {\bf e}^{n+1}\Vert^2_{{\mathbb L}^2}\bigr)\bigr].
\end{equation}

For the moment let us suppose that~\eqref{eq:TheDream} {\em would} hold (\emph{e.g.}, this would be true if the random variables $\Vert \nabla {\bf u}^{n+1}_1\Vert^2_{{\mathbb L}^2}$ and $ 
\Vert {\bf e}^n\Vert^2_{{\mathbb L}^2} + \Vert {\bf e}^{n+1}\Vert^2_{{\mathbb L}^2}$ would be independent). Then, we could complete the strong error analysis with the help of the implicit version of the discrete Gronwall lemma. However, we want to stress that there is no particular reason why the starting assumption (\emph{i.e.},~\eqref{eq:TheDream}) should be satisfied. 

Thus, instead of working with the right-hand side of~\eqref{eq:TheDream}, we must proceed with the left-hand side, requiring us to control the gradient. To do this, three methods have been applied so far, which we describe in the following three subsections. They all rely on the same idea to exclude certain small sample sets where the critical term $\Vert \nabla {\bf u}^{n+1}_1\Vert^2_{{\mathbb L}^2}$ is unbounded\footnote{Since we could have swapped the roles of $\bfu_1$ and $\bfu_2$, the critical term is actually: $\min\{\Vert \nabla {\bf u}^{n+1}_1\Vert^2_{{\mathbb L}^2},\Vert \nabla {\bf u}^{n+1}_2\Vert^2_{{\mathbb L}^2}\}$.}. They all lead to a proof of convergence in probability but differ in their technical implementations. All of them are just a tool for the theoretical error analysis and do not effect the implementation.

\subsubsection{The truncated problem}
The first idea is to truncate the nonlinearity, rendering it Lipschitz continuous with a large Lipschitz constant, say $R\gg1$. Then, one approximates the truncated problem instead of the original one, an approach which originates from
\cite{Pr}. After tracing explicitly the dependence of the error analysis with respect to the truncation parameter, one balances the blow-up of solutions and the error decay by coupling the truncation parameter to the discretization parameters (\emph{e.g.}, $R = \log(\sqrt{\tau} + h)$). By noticing that in the event that the solution stays below the truncation threshold, this solution also solves the non-truncated problem, it is possible to transfer the error decay rates to the original problem but only for this restricted event.
Finally, by passing from the mean square error to the error in probability, one can remove this restriction to obtain the desired result.

\subsubsection{Localisation of the sample space}
The error analysis from \cite{CP} (see also \cite{BD1})
is based on 
estimates in $L^2(\Omega)$, which are localised with respect to the sample set. The size of the neglected sets shrinks asymptotically with respect to the discretisation parameters and is consequently not seen
when measuring the error in probability. On a technical level one multiplies the equation for the error in each time step with the indicator function of the event that the solution from the previous step stays (in a certain norm) below some threshold.\footnote{It would be desirable to control it up to the current step -- rather than the previous one -- but this creates problems with measurability related to stochastic integrals. }
The final estimate relies then on an iterative argument.
The exceptional set of the sample space is not reached with high probability, leading to convergence rates for the error in probability.

 \subsubsection{Discrete stopping times}\label{sec:4.3.3}
A new method based on discrete stopping times has been introduced in \cite{BrPr1}. We consider it superior
to the one by localisation of the sample space described above since it allows us to control all quantities even in the current time step and, at the same time, preserves the properties of the stochastic integrals.
As a result one obtains `global-in-$\Omega$' but `local-in-time' (\emph{i.e.}, up to the discrete stopping time only) convergence rates for the error. Since the discrete stopping times are constructed such that they converge (for vanishing discretization parameters) to $T$, where $T$ can be any given end-time, the convergence rates for the error in probability follow.

\begin{framed}
In general, the solution to stochastic Navier--Stokes equations is unbounded with positive probability, hindering the control of the convective term and, hence, prohibiting a classical error analysis of numerical approximations. To recover convergence rates, one has to remove the events where the solution reaches large values; \emph{e.g.}, this is achieved by using truncation, localisation or stopping times. Since these events do not occur asymptotically, the convergence rates transfer to the {\em error in probability}. 
\end{framed}

\subsection{Open questions}

\subsubsection{Weak error analysis}
For stochastic evolution problems it is common to consider, as an alternative to the pathwise error discussed so far, the \emph{weak error} between the continuous solution and its numerical approximation. The latter describes -- roughly speaking --- how well the probability distribution of the solution can be approximated. For many questions of interest this is sufficient. One expects that the convergence regarding the temporal error is twice the rate from the strong/pathwise error.

Let us first describe this for a finite dimensional problem. We approximate the stochastic differential equation
\begin{equation*}
\dd X_t=\mu(X_t)\dt+\sigma(X_t)\, \dd W,\quad X_0=\eta,
\end{equation*}
via an implicit Euler scheme, obtaining iterates $\{X_n\}_{n=1}^N$ to approximate $X(t_n)$. In the weak error analysis one analyses the quantity
\begin{equation*}
|\mathbb E[\varphi(X(T))-\varphi(X_N)]|,
\end{equation*}
where $\varphi:\R\rightarrow\R$ is smooth and bounded (approximating powers by truncations one can then control the error for all moments of $X_t$). Typically, one obtains
 \begin{equation}\label{eq:weak}
|\mathbb E[\varphi(X(T))-\varphi(X_N)]|\leq C \tau^{2\alpha},
\end{equation}
for any $\alpha<1/2$, see \emph{e.g.} \cite{KlPl}.
On the other hand it holds only
\begin{equation}\label{eq:strong}
\big(\mathbb E[|X(T))-X_N|^2]\big)^{\frac{1}{2}} \leq C \tau^\alpha
\end{equation}
for the strong error.\footnote{Note that \eqref{eq:strong} trivially implies a weak error of order $\tau^\alpha$.}
The proof of \eqref{eq:weak} heavily depends on the Kolmogorov equation, which describes the evolution of the probability density of $X_t$. In the easiest situation, that is if $\mu=0$ and $\sigma=1$, the solution is given by the driving Wiener process (\emph{i.e.}, $X_t=W_t$) and therefore its probability density is known explicitly; \emph{i.e.},
\begin{equation*}
\mathscr U(t,x)=\frac{1}{\sqrt{2\pi}t}\exp(-\tfrac{x^2}{2t}).
\end{equation*}  
Clearly, this is the fundamental solution to the heat equation. In the general case the Kolmogorov equation is the following parabolic PDE that defines $\mathscr U$:
\begin{equation}\label{eq:kol}
\partial_t\mathscr U(t,x)=\tfrac{\sigma^2(x)}{2}\partial_x^2\mathscr U(t,x)+\mu(x)\partial_x \mathscr U(t,x),
\end{equation}
where the coefficients depend on $\mu$ and $\sigma$. Standard results concerning the regularity of solutions to this parabolic equation apply. They are crucial for the proof of \eqref{eq:weak}.

\begin{figure*}[t!]
    \centering
    \begin{subfigure}[t]{0.5\textwidth}
        \centering
        \includegraphics[width=1.0\textwidth]{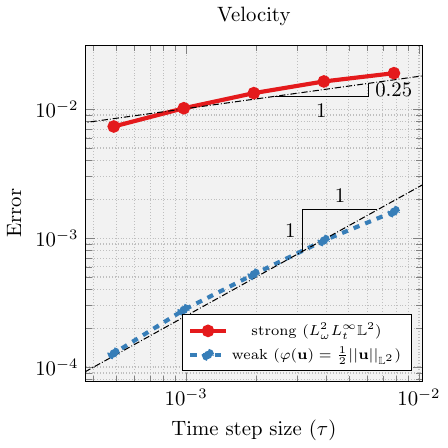}
    \end{subfigure}%
    ~ 
    \begin{subfigure}[t]{0.5\textwidth}
        \centering
        \includegraphics[width=1.0\textwidth]{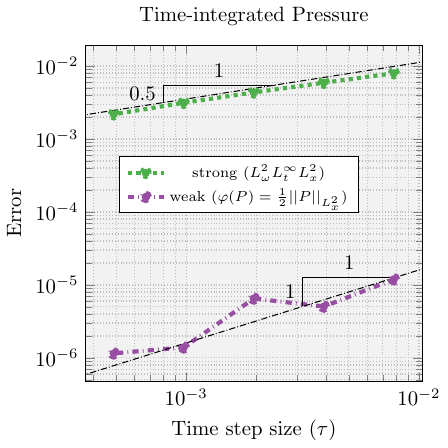}
    \end{subfigure}
    \caption{Evolution of the strong and weak errors for velocity (left) and time-integrated pressure (right) --- generated by a modified version of the implicit Euler scheme (Algorithm~\ref{algo:IE-NSE}) where the convective term is symmetrised --- for non-linear multiplicative noise (Example~\ref{ex:nonlinearMultiplicative}).
    } 
    \label{fig:weakError}
\end{figure*}

First results for SPDEs can be found in \cite{DB1,DB2}, where linear and semilinear problems with Lipschitz-continuous nonlinearities are considered (see also \cite{DB3} for more recent results).  In this case one can derive an infinite-dimensional version of \eqref{eq:kol} (where usually the `spatial' variable belongs to $\mathbb L^2$) and obtain corresponding estimates, \emph{cf.} \cite{Cerrai}. 
Eventually, a version of \eqref{eq:weak} can be derived, where the function $\varphi$ is now defined on the infinite-dimensional path space (with values in $\mathbb R$).
Note that the main difficulty is the error arising from the temporal discretisation. One can split the error between exact solution and space-time discretisation into the error between exact solution and temporal discretisation and that between temporal discretisation and spatio-temporal discretisation. For the latter we expect linear convergence with respect to the spatial discretisation parameter in the strong sense (as described in Section~\ref{sec:state-NSE}) and thus the same for the weak error. Hence the focus is typically only on the weak error analysis for the temporal discretisation.

SPDEs with non-Lipschitz nonlinearities have been considered in \cite{BrPr24b} and \cite{Bre} where a corresponding weak error analysis for the stochastic Allan-Cahn equation has been performed. For the latter the nonlinearity is of lower order and monotone. Being equipped with this additional property of the nonlinearity, it is then possible to adapt the program in~\cite{DB1,DB2,DB3}, and to prove the desired estimates for the Kolmogorov equation, resulting in~\eqref{eq:weak} for the Allen-Cahn equation\footnote{$X(T)$ and $X_N$ have to be replaced by the solution to the Allen-Cahn equation and its implicit Euler discretization in time, respectively}.

A weak error analysis for the stochastic Navier--Stokes equations is currently completely open. One may wonder if one can work with the idea of ``large sample sets'' as discussed before in Section~\ref{sec:largesamplesets} in connection with the strong error analysis. We believe that it is most convenient to employ a cut-off for the nonlinearity controlling a certain norm by the parameter $R$ (since indicator functions and stopping times cause issues with progressive measurability).
As a matter of fact, we believe that it is possible to obtain estimates for the Kolmogorov equation for the truncated problem (in terms of $R$) and to prove
\begin{equation}\label{eq:weakR}
|\mathbb E[\varphi(\bfu^R(T))-\varphi(\bfu^R_N)]|\leq\, c(R)\tau^{2\alpha}
\end{equation}
for all $\alpha<\frac{1}{2}$ via a modification of the approach from \cite{BrPr24b}. This corresponds to the estimate for the strong error and for the latter one can eventually obtain the error in probability by Markov's inequality. Unfortunately, we are not aware of a counterpart of that idea and thus we do not know how to extract any useful information from  \eqref{eq:weakR}.

\begin{framed}
For the stochastic Navier--Stokes equations, a theoretical foundation of the weak error analysis for both the velocity and time-integrated pressure is still missing. At least numerically,
for both errors a rate of order~$1$ (in time) is observed (\emph{cf.} Figure~\ref{fig:weakError}), but this is not supported by any theory yet.  
\end{framed}

\subsubsection{Transport noise}
The results from~\cite{BMPW} leave several questions open for the stochastic Navier--Stokes equations with transport noise.

Firstly, transport noise is always energy conservative (if the driving vector fields are solenoidal) -- a property which can be preserved on the time-discrete level. If the driving vector fields are constants the cancellation of the noise even applies after differentiating the momentum equation in space. Hence estimates in higher order Sobolev space can be obtained. This applies to the exact as well as the time-discrete solution. However, it is unclear if one can obtain similar results for non-constant vector fields. The same problem appears if one considers the equations in a bounded domain supplemented with no-slip boundary conditions instead of the space-periodic problem: the cancellation property of the noise for the differentiated system is lost and it is questionable if the same estimates can be proved.

Secondly, a central assumption in their error analysis is that the viscosity dominates the strength of the noise in a certain way. This is crucial to close the estimate in the theoretical error analysis. Interestingly, our computational studies suggest that this assumption might not be needed. Hence it would be very interesting to understand if temporal error estimates can be obtained without this assumption.

Finally, as alluded to above the temporal error analysis depends on higher order energy estimates for the continuous as well as the time-discrete solution. A corresponding analysis of a spatial (or spatio-temporal) discretisation requires analogous estimates for the discrete solution: in particular, it would involve uniform bounds for fully discrete
iterates of the velocity in high norms which may be obtained with the help of the discrete Stokes
operator. To our knowledge, such a program has not been carried through in the literature, and thus
whether optimal rates for such a space-time discretization may be reached remains open.

\begin{framed}
In summary, it is open to consider the following data settings in the temporal error analysis for the stochastic Navier--Stokes equations with transport noise: non-constant driving vector fields, bounded domains with no-slip boundary conditions, and strong noise which is not dominated by the viscosity. Moreover, an error analysis for a fully discrete algorithm has not been derived yet.
\end{framed}

\subsubsection{Discretisations via efficient time splitting}

It is due to the nonlinear term that solving the stochastic Navier--Stokes equations numerically is far more costly than solving the (linear) stochastic Stokes equations. Typically, a Newton method is used
on this behalf, and therefore the obvious question arises, whether efficient splitting methods, such as generalisations of the
Chorin method (\emph{cf.} Section~\ref{os1}) and the pressure-corrected Chorin method (\emph{cf.} Section~\ref{sec:pres-cor-Chorin}), are also reliable for the Navier--Stokes equations. In case they would be reliable (which is not supported by any theory at the moment), they could be used to accelerate the computation.

The main obstacle that is hindering a successful convergence analysis at the moment is to properly balance the demand for regularity/stability estimates in higher norms for (discrete) solutions with a flexible decoupling strategy. Here, the relevant question which needs to be confirmed for a (new) splitting method is whether iterates for velocity and pressure are stable in a \textit{sufficiently strong} sense. For example, as we have seen for splitting methods devised for approximating the Stokes equations, especially uniform bounds for the pressure in high norms plays an important role (see the discussion around~\eqref{eq:chorin-re0}) in the derivation of their convergence rates. In the non-linear case however the quantification of `\textit{sufficiently strong}' is unclear at the moment. 

Another difficulty comes into play, if the non-linearity is linearised. While this is much desired from an applied perspective since it accelerates the computation even further, it imposes another error, making the theoretical analysis much more complex. For the deterministic Navier--Stokes equations, splitting methods that use a linearisation have been analysed, \emph{e.g.}, in \cite{Prohl1997}; but neither the algorithms nor the analyses can be generalised to the stochastic case in a straightforward manner.

To account for more realistic settings, it is important to also study the effect of {\em non}-periodic boundary conditions for splitting methods (see also the next section). However, already for the stochastic Stokes equations the generalisation of, \emph{e.g.}, the Chorin method to the context of boundary value problems typically requires to impose non-physical boundary conditions for the pressure. These artificial boundary conditions create an additional source of error that must be analysed. While this error is well-understood in the deterministic setting (see, \emph{e.g.},~\cite{Prohl1997}), its impact in the stochastic case is unknown.

\begin{framed}
Currently, it is not known how to lower computational complexities by efficient time-splitting algorithms
 for the stochastic Navier--Stokes equations and guaranteeing optimal convergence at the same time. This is since the algorithms, which is devised for the deterministic equations, cannot be generalised straightforwardly; also, the considered three noises might lead to different algorithms. Moreover, the existing error analyses of splitting methods for the stochastic Stokes equations are limited to periodic boundary conditions, and hence, it is not clear how they perform for more realistic data, such as non-periodic boundary conditions. 
\end{framed}

\subsubsection{General domains and boundary conditions}
Most analyses of algorithms for the approximation of the stochastic Navier--Stokes equations only hold for simple domains and periodic boundary conditions due to the limited regularity of solutions. But as soon as one wants to reliably use these algorithms in more realistic situations, more general domains as well as non-periodic boundary conditions must be considered in their construction and analysis.

For no-slip boundary conditions (\emph{i.e.}, homogeneous Dirichlet boundary data) and in 2D, a convergence analysis of the implicit Euler scheme is performed in~\cite{BrPr1}, guaranteeing the reliability of this method for no-slip boundary conditions as well. A corresponding result so-far seems not available for efficient time-splitting methods to approximate with rates the solution of the stochastic Navier--Stokes equations in 2D. And we are not aware either of any results for them which would settle mere convergence in 3D. 

\begin{framed}
At the moment, the error analysis of numerical algorithms for the stochastic Navier--Stokes equations on general domains with non-periodic boundary conditions is mainly open. From a practical viewpoint, it would be relevant to perform a corresponding analysis and to compare different algorithms by related computational studies for practically relevant fluid flow problems.
\end{framed}


\section{Benchmark problems} \label{sec:benchmarks}
The theoretical study of algorithms to simulate the stochastic Stokes or Navier-Stokes equations in {\em academic} fluid flow scenarios is an important step to evaluate their approximation behaviour, but the mathematical tools as well as their practical value are very often limited to small numerical parameters, short times, or simply hinge on restrictive data assumptions ({\em e.g.}, periodic boundary conditions).
And, by evidence, the methods that we used so far do not provide any guidance on computational complexities, or proper
choices for numerical parameters to achieve a reasonable resolution of the solution — which is a relevant
practical criterion to compare different algorithms. As a result, we devise different problems here — both, of academic kind and of practical relevancy — to compare different algorithms in terms of accuracy and computational complexities to complement their evaluation. While there
is a plethora of benchmarks available in the case of deterministic fluid-flow problems ({\em e.g.}, the lid driven cavity problem) that are well-accepted in the community, we here propose a selection of those to compare algorithms for the stochastic case.

In fact, in the current literature on numerical methods for the stochastic Stokes and Navier--Stokes equations it is that the computational experiments are usually set up on a case-by-case basis, which makes the direct comparison of different algorithms impossible. Therefore, the aim of this section is to initiate a common collection of relevant examples, so that old and new algorithms can be tested within the same parameter configurations, therefore enabling a better comparability of their performance.

Each example is self-contained, focusing on the needs for its implementation; this enables software developers to conduct it easily. In between the examples, we motivate the different problems and explain what purpose they serve for. In this way, mathematicians can choose the benchmark problems which are suitable for them, depending on their research hypotheses.

The benchmark problems are designed to either provide supporting, numerical evidence for theoretical results or to bridge the gap between theory and applications. For the former, regularity of the (typically unknown) solution is of utmost importance; and for the latter, numerical algorithms have to be tested in real-world (or at least realistic) situations. In the next two subsections, we will construct solutions with prescribed regularity (\emph{i.e.}, the \textit{direct approach}) and we will present two examples in which data are prescribed and the solutions are unknown (\emph{i.e.}, the \textit{non-direct approach}).




\subsection{The direct approach: prescribing solutions} 
A major advantage of the direct approach is that it allows to prescribe arbitrary solutions, including regular and irregular ones. In this way, it is possible to numerically test the sensitivity of algorithms with respect to the regularity/irregularity of solutions. Due to its simplicity it is used in the deterministic case extensively. In contrast, in the stochastic case it is widely unknown and thus, it has not been used much yet. To change this, we will provide explicit solutions for the stochastic Stokes equations, which can be used to compute the exact approximation error in numerical simulations.

The section is structured as follows: First, we consider the deterministic case to illustrate the method. After that, we extend the construction to additive, then linear and multiplicative, and finally transport noises. In each case, we provide examples that can serve as benchmark problems for testing algorithms.

\subsubsection{Deterministic Stokes equations} 
The deterministic Stokes equations with no-slip (\emph{i.e.}, homogeneous Dirichlet) boundary conditions are given by:
\begin{subequations} \label{ben-eq:det-Stokes}
\begin{eqnarray}
\partial_t \bfu  -  \mu \Delta {\bf u} + \nabla p  &=&   {\bf f} \quad \qquad \qquad \ \, \quad \qquad \mbox{on } (0,T) \times D,\\ 
{\rm div} \, {\bf u} &=& 0 \quad \qquad \qquad \ \, \quad \qquad \mbox{on } (0,T) \times D, \\
{\bf u} &=& 0 \quad \qquad \qquad \ \, \quad \qquad \mbox{on } (0,T) \times \partial D, \\
{\bf u}(0) &=& {\bf u}_0 \qquad \qquad \qquad \qquad  \mbox{on } D.
\end{eqnarray}
\end{subequations}

It is a simple exercise to manufacture exact solutions for these equations: Given an incompressible vector field~$\bfu: (0,T)\times D \to \mathbb{R}^2$, a mean-value-free pressure~$p:(0,T)\times D \to \mathbb{R}$ and a viscosity~$\mu > 0$, one computes the acting body force~$\bff_\mu: (0,T)\times D \to \mathbb{R}^2$ by solving:
\begin{equation} \label{eq:manufactured-forcing}
\bff_\mu := \partial_t \bfu - \mu \Delta \bfu + \nabla p.
\end{equation}
Additionally evaluating the velocity at the initial time and boundary, one finds the corresponding initial and boundary conditions, respectively. 

We start with the following example: 
\begin{framed}
\begin{example}\label{ex:detStokes} Consider the data:

 $\bullet$ domain $D = (0,1)^2 \qquad \qquad\bullet$ time $T=1 \qquad \qquad \bullet$ viscosity $\mu =1$ 
 
 $\bullet$ body force $\bff(t,x,y) = 0 - \mu \begin{pmatrix}
 f(x,y) \\
 -f(y,x)
\end{pmatrix} + 2t \begin{pmatrix}
x \\ y
\end{pmatrix}$ where
\begin{equation}\nonumber
f(x,y) = 4x^2y(2y^2 - 3y + 1) + 16xy(x - 1)(2y^2 - 3y + 1)   + 4y(x - 1)^2(2y^2 - 3y + 1) + 4x^2(6y - 3)(x - 1)^2
\end{equation} 

$\bullet$ initial condition~$\bfu_0(x,y) =  \begin{pmatrix}
2x^2\left(1-x\right)^2 y(y-1)(2y-1) \\ -2y^2\left(1-y\right)^2x(x-1)(2x-1)
\end{pmatrix} $\\
In this case, the unique solution to~\eqref{ben-eq:det-Stokes} is given by:
\begin{equation}\nonumber
\bfu(t,x,y) \equiv  \begin{pmatrix}
2x^2\left(1-x\right)^2 y(y-1)(2y-1) \\ -2y^2\left(1-y\right)^2x(x-1)(2x-1)
\end{pmatrix} \quad \text{ and } \quad p(t,x,y) = \left(x^2 + y^2 - \frac{2}{3} \right) t.
\end{equation}
\end{example}
\end{framed} 
This example has been specifically designed to test the influence of exploding time derivatives of the pressure (for $t \downarrow 0$) on the performance of numerical algorithms. Especially for the analysis of the Chorin method, the blow-up of $\nabla \partial_t p $ plays an important role. In the above example, we can scale the effect by replacing every $t$ by $t^\gamma$ (where $\gamma \in (0,1]$ measures the temporal regularity). 

To have available solutions of certain fluid flow problems extends to the stochastic case for additive, linear and multiplicative, and transport noises. With suitable modifications, it would be possible to cover the stochastic Navier--Stokes equations, too. Recall that $W$ denotes a real-valued Wiener process.

\subsubsection{Stochastic Stokes equations with additive noise} 
Let $\bfg: \Omega \times (0,T) \times D \to \mathbb{R}^d$ be a given, stochastically integrable vector field. The stochastic Stokes equations with no-slip boundary conditions forced by additive noise are given by:

\begin{subequations}\label{bench-eq:Additive}
\begin{eqnarray}
{\rm d}{\bf u} - \left( \mu \Delta {\bf u}- \nabla p \right)  \dd t &=&   {\bf f}{\rm d}t + \bfg \dd W 
\qquad \qquad \, \, \mbox{on }  \Omega \times (0,T) \times D, \\
{\rm div} \, {\bf u} &=& 0 \quad \qquad \qquad \ \,\, \quad \qquad \mbox{on } \Omega \times (0,T) \times D, \\
{\bf u} &=& 0 \quad \qquad \qquad \ \,\, \quad \qquad \mbox{on } \Omega \times (0,T) \times \partial D, \\
{\bf u}(0) &=& {\bf u}_0 \qquad \qquad \qquad \qquad\,  \mbox{on } \Omega \times D.
\end{eqnarray}
\end{subequations}

For the construction of a solution, we start with the ansatz (recalling that $P_{\tt HL}$ and $P_{\tt HL}^\perp$ denote the Helmholtz projection and its orthogonal complement, respectively):
\begin{eqnarray} \nonumber
\bfu_{\mathrm{add}}(t) &:= &\bfu(t) + P_{\tt HL} \left[ \int_0^t \bfg(s) \dd W(s) \right], \\ \nonumber
\nabla P_{\mathrm{add}}(t) &:= &\int_0^t \nabla p(s) \dd s + P_{\tt HL}^\perp \left[ \int_0^t \bfg(s) \dd W(s) \right];
\end{eqnarray}
that is, we construct solutions to the stochastic Stokes equations as a linear combination of the deterministic solution and suitably decomposed noise; indeed, the Helmholtz decomposition:
\begin{equation} \nonumber
\int_0^t \bfg(s) \dd W(s) =  P_{\tt HL} \left[ \int_0^t \bfg(s) \dd W(s) \right] +  P_{\tt HL}^\perp \left[ \int_0^t \bfg(s) \dd W(s) \right],
\end{equation}
splits the noise into two terms: the first term is an incompressible vector field; the second term corresponds to the gradient part. Each of them is solely influencing velocity and pressure, respectively. 
 
An explicit computation, in which we use the relation between $\bff_\mu$, $\bfu$ and $p$ (see identity~\eqref{eq:manufactured-forcing}), shows that the ansatz generates a solution to the following equations:
\begin{eqnarray}\nonumber
\dd \bfu_{\mathrm{add}}- \mu \Delta \bfu_{\mathrm{add}} \dd t + \dd \nabla P_{\mathrm{add}} & = &\left( \bff_{\mu} - \mu \Delta P_{\tt HL} \left[ \int_0^t g(s) \dd W(s) \right]\right)  \dd t + \bfg \dd W, \\ \nonumber
\Div \bfu_{\mathrm{add}} &=& 0.
\end{eqnarray}

Notice that non-solenoidal noises force us to consider the time-integrated pressure~$P_{\mathrm{add}}$ instead of the classical one. Moreover, the pressure decomposes into a time-regular (called artificial or deterministic) pressure and time-irregular (called stochastic) pressure; see also~\eqref{eq:reconstruct-pressure}, which is in agreement with the theoretical results outlined in Section~\ref{sec:model-and-theory}. The body force~$\bff_\mu$ is shifted in the direction of the divergence-free part of the noise.

We give a warning that the noise might have a non-trivial impact on the boundary conditions. Indeed, since the Helmholtz projection preserves vanishing boundary data only in the normal direction of the boundary, the noise can lead to the prescription of non-trivial boundary conditions in the tangential direction, even though $\bfg$ vanishes completely at the boundary. 

Next, we propose three examples designed for each of the cases: \textit{solenoidal noise}, \textit{potential noise}, and \textit{general noise}. 

For all of them, we use the following ansatz: we define~$\bfg$ in terms of its Helmholtz decomposition. In this way, we do not need to compute the projections; instead, we can simply read-off how the projections $P_{\tt HL}$ and $ P_{\tt HL}^\perp$ act. In particular, we can insist on prescribing zero Dirichlet boundary conditions for the divergence-free and gradient parts separately. Moreover, we choose $\bfg$ to be deterministic and constant in time which enables the explicit computation of the stochastic integrals.

\begin{framed}
\begin{example}[additive solenoidal noise] \label{ex:additive-solenoidal} Consider the data:

 $\bullet$ domain $D = (0,1)^2  \quad\bullet$ time $T=1 \quad \bullet$ viscosity $\mu =1 \quad \bullet$ 1D-Brownian motion $W$

 $\bullet$ noise coefficient $\bfg(\omega,t,x,y) \equiv  \begin{pmatrix}
2x^2\left(1-x\right)^2 y(y-1)(2y-1) \\ -2y^2\left(1-y\right)^2x(x-1)(2x-1)
\end{pmatrix}$ 
 
 $\bullet$ body force $\bff(\omega,t,x,y) = - \mu \begin{pmatrix}
 f(x,y) \\
 -f(y,x)
\end{pmatrix}\left(1 + W_t(\omega) \right) + 2t \begin{pmatrix}
x \\ y
\end{pmatrix}$ where
\begin{equation}\nonumber
f(x,y) = 4x^2y(2y^2 - 3y + 1) + 16xy(x - 1)(2y^2 - 3y + 1)   + 4y(x - 1)^2(2y^2 - 3y + 1) + 4x^2(6y - 3)(x - 1)^2
\end{equation} 

$\bullet$ initial condition~$\bfu_0(x,y) =  \begin{pmatrix}
2x^2\left(1-x\right)^2 y(y-1)(2y-1) \\ -2y^2\left(1-y\right)^2x(x-1)(2x-1)
\end{pmatrix} $\\
In this case, the unique solution to~\eqref{bench-eq:Additive} is given by:
\begin{equation}\nonumber
\bfu(\omega,t,x,y) = \begin{pmatrix}
2x^2\left(1-x\right)^2 y(y-1)(2y-1) \\ -2y^2\left(1-y\right)^2x(x-1)(2x-1)
\end{pmatrix}\left( 1+ W_t(\omega) \right)
\end{equation}
and 
\begin{equation}\nonumber
 p(t,x,y) = \left(x^2 + y^2 - \frac{2}{3} \right) t.
\end{equation}
\end{example}
\end{framed}

\begin{framed}
\begin{example}[additive potential noise] \label{ex:add-potential}Consider the data:

 $\bullet$ domain $D = (0,1)^2  \quad\bullet$ time $T=1 \quad \bullet$ viscosity $\mu =1 \quad \bullet$ 1D-Brownian motion $W$
 
 $\bullet$ noise coefficient~$\bfg(\omega,t,x,y) \equiv  \begin{pmatrix}
2y^2(1-y)^2 x(x-1)(2x-1) \\ 2x^2(1-x)^2y(y-1)(2y-1)
\end{pmatrix}$
 
 $\bullet$ body force~$
\bff(\omega,t,x,y) = - \mu \begin{pmatrix}
 f(x,y) \\
 -f(y,x)
\end{pmatrix} + 2t \begin{pmatrix}
x \\ y
\end{pmatrix}$ where
\begin{equation}\nonumber
f(x,y) = 4x^2y(2y^2 - 3y + 1) + 16xy(x - 1)(2y^2 - 3y + 1)   + 4y(x - 1)^2(2y^2 - 3y + 1) + 4x^2(6y - 3)(x - 1)^2
\end{equation} 

$\bullet$ initial condition~$\bfu_0(x,y) =  \begin{pmatrix}
2x^2\left(1-x\right)^2 y(y-1)(2y-1) \\ -2y^2\left(1-y\right)^2x(x-1)(2x-1)
\end{pmatrix} $\\
In this case, the unique solution to~\eqref{bench-eq:Additive} is given by:
\begin{equation}\nonumber
\bfu(\omega,t,x,y) = \begin{pmatrix}
2x^2\left(1-x\right)^2 y(y-1)(2y-1) \\ -2y^2\left(1-y\right)^2x(x-1)(2x-1)
\end{pmatrix} 
\end{equation}
and 
\begin{equation}\nonumber
 P(\omega,t,x,y) = \left(x^2 + y^2 - \frac{2}{3} \right) \frac{1}{2} t^2 +  \left( x^2(1-x)^2y^2(1-y)^2 - \frac{1}{900} \right) W_t(\omega).
\end{equation}
\end{example}
\end{framed} 

\begin{framed}
\begin{example}[additive general noise] \label{ex:add-generic} Consider the data:

 $\bullet$ domain $D = (0,1)^2  \quad\bullet$ time $T=1 \quad \bullet$ viscosity $\mu =1 \quad \bullet$ 1D-Brownian motion $W$
 
 $\bullet$ noise coefficient 
\begin{equation} \nonumber
\bfg(\omega,t,x,y) \equiv \begin{pmatrix}
2x^2\left(1-x\right)^2 y(y-1)(2y-1) \\ -2y^2\left(1-y\right)^2x(x-1)(2x-1)
\end{pmatrix} +  \begin{pmatrix}
2y^2(1-y)^2 x(x-1)(2x-1) \\ 2x^2(1-x)^2y(y-1)(2y-1)
\end{pmatrix}
\end{equation}
 
 $\bullet$ body force~$
\bff(\omega,t,x,y) = - \mu \begin{pmatrix}
 f(x,y) \\
 -f(y,x)
\end{pmatrix}\left( 1+ W_t(\omega) \right)  + 2t \begin{pmatrix}
x \\ y
\end{pmatrix}$ where
\begin{equation}\nonumber
f(x,y) = 4x^2y(2y^2 - 3y + 1) + 16xy(x - 1)(2y^2 - 3y + 1)   + 4y(x - 1)^2(2y^2 - 3y + 1) + 4x^2(6y - 3)(x - 1)^2
\end{equation} 

$\bullet$ initial condition~$\bfu_0(x,y) =  \begin{pmatrix}
2x^2\left(1-x\right)^2 y(y-1)(2y-1) \\ -2y^2\left(1-y\right)^2x(x-1)(2x-1)
\end{pmatrix} $\\
In this case, the unique solution to~\eqref{bench-eq:Additive} is given by:
\begin{equation}\nonumber
\bfu(\omega,t,x,y) = \begin{pmatrix}
2x^2\left(1-x\right)^2 y(y-1)(2y-1) \\ -2y^2\left(1-y\right)^2x(x-1)(2x-1)
\end{pmatrix} \left( 1+  W_t(\omega) \right) 
\end{equation}
and 
\begin{equation}\nonumber
 P(\omega,t,x,y) = \left(x^2 + y^2 - \frac{2}{3} \right) \frac{1}{2} t^2 +  \left( x^2(1-x)^2y^2(1-y)^2 - \frac{1}{900} \right) W_t(\omega).
\end{equation}
\end{example}
\end{framed} 
These examples underline that it is the Helmholtz decomposition of the noise which determines whether the velocity, the pressure, or both are affected. Solenoidal noise acts purely on the velocity; potential noise (\emph{i.e.}, noise that is the gradient of a scalar function) leaves an impact on the pressure only; and general noise influences both. Moreover, they show that in general the time-integrated pressure cannot behave better than the Brownian motion.

\subsubsection{Stochastic Stokes equations with linear multiplicative noise} Let $g: \Omega \times (0,T) \to \mathbb{R}$ be a given, stochastically integrable function. The stochastic Stokes equations with no-slip boundary conditions forced by multiplicative noise are given by:

\begin{subequations} \label{bench-eq:multi}
\begin{eqnarray}
{\rm d}{\bf u} - \left( \mu \Delta {\bf u} - \nabla p \right) \dd t &=&   {\bf f}{\rm d}t +  \bfu g \dd W 
\qquad \qquad  \mbox{on }  \Omega \times (0,T) \times D, \\
{\rm div} \, {\bf u} &=& 0 \quad \qquad \qquad \qquad \qquad \mbox{on } \Omega \times (0,T) \times D, \\
{\bf u} &=& 0\quad \qquad \qquad \qquad \qquad \mbox{on } \Omega \times (0,T) \times \partial D, \\
{\bf u}(0) &=& {\bf u}_0 \qquad \qquad \qquad \qquad\, \,\,  \mbox{on } \Omega \times D.
\end{eqnarray}
\end{subequations}

Notice that we assumed $g$ to be constant in space. This simplifies the computations substantially; but we also want to stress that the arguments below can be extended to a space-dependent $g$ with suitable modifications. From time to time, we hint on the modifications that would be needed.

We will frequently use the geometric Brownian motion~$W_g$, which is defined by:
\begin{equation*}
W_g(t):= \exp\left( \int_0^{t} g(s) \dd W(s) - \frac{1}{2} \int_0^{t} \abs{g(s)}^2 \dd s \right).
\end{equation*}
Importantly, an equivalent definition of~$W_g$ can be given in terms of stochastic calculus. Indeed, $W_g$ is the solution to the following stochastic differential equation:
\begin{equation} \label{eq:geo-BM}
\dd X (t) = g(t) X(t) \dd W(t) \qquad \text{ and } \qquad X(0) = 1.
\end{equation}

Our ansatz for computing an explicit solution for the stochastic Stokes equations forced by linear multiplicative noise is the following:
\begin{equation} \label{eq:ansatz-multi} 
\bfu_{\mathrm{mul}}(t) :=  \bfu(t) W_g(t) \quad \text{ and }  \quad p_{\mathrm{mul}}(t) := p(t) W_g(t),
\end{equation} 
\emph{i.e.}, we multiply both, the deterministic velocity and pressure by the geometric Brownian motion. 

We start by observing that $\bfu_{\mathrm{mul}}$ is solenoidal since $g$ and hence $W_g$ are constant in space. Indeed, applying the divergence operator in the ansatz shows: $\Div \bfu_{\mathrm{mul}}(t) = \bfu(t) \cdot \nabla W_g(t) = 0$. At this stage, we see clearly that a space-dependent~$g$ would have a non-trivial impact on the incompressibility condition, and as such it would affect the pressure substantially.

Next, we compute the evolution equation for $\bfu_{\mathrm{mul}}$. This corresponds to an application of It\^o's formula, which reveals that
\begin{eqnarray*}
\dd \bfu_{\mathrm{mul}}(t) &= & W_g(t) \dd \bfu(t)  + \bfu(t) \dd W_g(t) \\
&= &W_g(t)\big( \mu \Delta \bfu(t) - \nabla p (t) + \bff_\mu (t) \big) \dd t + \bfu(t) g(t) W_g(t) \dd W(t),
\end{eqnarray*}
where we have used a) that $\bfu$ is time-differentiable which causes the cross-variation to vanish; b) Identity~\eqref{eq:manufactured-forcing}; and c) the fact that $W_g$ solves~\eqref{eq:geo-BM}. It remains to rewrite the right-hand side as a function of $\bfu_\mathrm{mul}$, $p_\mathrm{mul}$, and a modified force.  

Using~\eqref{eq:ansatz-multi}, we observe that the stochastic integral is already in the correct form, \emph{i.e.}:
\begin{equation*}
\bfu(t) g(t) W_g(t) \dd W(t) = g(t) \bfu_{\mathrm{mul}}(t) \dd W(t).
\end{equation*}
To rewrite the drift, we recall the following formula: $\Delta[a b] = \Delta[a] b + 2 \nabla a \cdot \nabla b + a\Delta[ b]$, which is valid for arbitrary, scalar functions $a$ and $b$. Applying this formula component-wise and then using that $g$ is constant in space, we find that
\begin{equation*}
W_g(t) \Delta \bfu(t) = \Delta \bfu_{\mathrm{mul}}(t)  - \left( 2 \nabla W_g(t) \cdot \nabla \bfu(t) + \bfu(t) \Delta W_g(t) \right) = \Delta \bfu_{\mathrm{mul}}(t).
\end{equation*}
Similarly, we conclude that $W_g(t) \nabla p (t) = \nabla p_{\mathrm{mul}}(t)$. Thus, we arrive at
\begin{equation*}
\dd \bfu_{\mathrm{mul}}(t)  = \left(\mu  \Delta \bfu_{\mathrm{mul}}(t)   - \nabla p_{\mathrm{mul}} (t) + W_g(t) \bff_\mu (t) \right) \dd t + g(t) \bfu_{\mathrm{mul}}(t) \dd W(t).
\end{equation*}
In summary, the calculations above have shown that~$\bfu_\mathrm{mul}$ and $p_\mathrm{mul}$ solve
\begin{equation*}
\dd \bfu_{\mathrm{mul}}(t)  = \left(\mu  \Delta \bfu_{\mathrm{mul}}(t)   - \nabla p_{\mathrm{mul}} (t) + W_g(t) \bff_\mu (t) \right) \dd t + g(t) \bfu_{\mathrm{mul}}(t) \dd W(t), \qquad \Div \bfu_{\mathrm{mul}}(t) = 0.
\end{equation*}

Based on this insight, we construct an example that uses an even simpler noise coefficient which is, in addition to being constant in space, deterministic and constant in time. This allows us to compute the stochastic integrals explicitly.

\begin{framed}
\begin{example}[linear multiplicative noise] \label{ex:multi} Consider the data:

 $\bullet$ domain $D = (0,1)^2  \quad\bullet$ time $T=1 \quad \bullet$ viscosity $\mu =1 \quad \bullet$ 1D-Brownian motion $W$
 
 $\bullet$ noise coefficient $g(\omega,t,x,y) \equiv \lambda =1$
 
 $\bullet$ body force
\begin{equation*}
\bff(\omega,t,x,y) = \left[ - \mu \begin{pmatrix}
 f(x,y) \\
 -f(y,x)
\end{pmatrix}  + 2t \begin{pmatrix}
x \\ y
\end{pmatrix} \right]\exp\left( \lambda W_t(\omega) - \frac{\lambda^2}{2} t \right)
\end{equation*} 

where
\begin{equation}\nonumber
f(x,y) = 4x^2y(2y^2 - 3y + 1) + 16xy(x - 1)(2y^2 - 3y + 1)   + 4y(x - 1)^2(2y^2 - 3y + 1) + 4x^2(6y - 3)(x - 1)^2
\end{equation} 

$\bullet$ initial condition~$\bfu_0(x,y) =  \begin{pmatrix}
2x^2\left(1-x\right)^2 y(y-1)(2y-1) \\ -2y^2\left(1-y\right)^2x(x-1)(2x-1)
\end{pmatrix} $\\
In this case, the unique solution to~\eqref{bench-eq:multi} is given by:
\begin{equation}\nonumber
\bfu(\omega,t,x,y) = \begin{pmatrix}
2x^2\left(1-x\right)^2 y(y-1)(2y-1) \\ -2y^2\left(1-y\right)^2x(x-1)(2x-1)
\end{pmatrix} \exp\left( \lambda W_t(\omega) - \frac{\lambda^2}{2} t \right)
\end{equation}
and 
\begin{equation}\nonumber
 p(\omega,t,x,y) = \left(x^2 + y^2 - \frac{2}{3} \right) t \exp\left( \lambda W_t(\omega) - \frac{\lambda^2}{2} t \right) .
\end{equation}
\end{example}
\end{framed}
Since the noise depends linearly on the solenoidal solution, it is solenoidal itself and hence, there is no need to consider the time-integrated pressure. However, notice that the pressure is still influenced by the noise. This is due to the body force that has a randomly scaled gradient-part in its Helmholtz decomposition.

\subsubsection{Stochastic Stokes equations with transport noise} Let $\bfsigma: \mathbb{R}^d \to \mathbb{R}^d$ be a given, sufficiently regular transport field. The stochastic Stokes equations with inhomogeneous Dirichlet boundary conditions forced by transport noise are given by:
\begin{subequations} \label{bench-eq:transport}
\begin{eqnarray}
{\rm d}{\bf u} - \left( \mu \Delta {\bf u} - \nabla p \right) \dd t &=&   {\bf f}{\rm d}t + (\bfsigma \cdot \nabla) \bfu \circ \dd W 
\qquad   \mbox{on }  \Omega \times (0,T) \times D, \\
{\rm div} \, {\bf u} &=& 0 \quad \qquad \qquad \quad \,\,  \qquad \qquad \mbox{on }\Omega \times  (0,T) \times D, \\
{\bf u} &=& \bfu_{\tt BC}\quad \qquad \qquad \qquad \,\,\, \qquad \mbox{on } \Omega \times (0,T) \times \partial D, \\
{\bf u}(0) &=& {\bf u}_0 \qquad \qquad \qquad \qquad \quad \,\,\,\,\,  \mbox{on } \Omega \times D.
\end{eqnarray}
\end{subequations}

Transport noise is inherently linked to the stochastic flow~$\bfX$ defined by:
\begin{equation*}
\dd \bfX_t^\bfx = -\bfsigma(\bfX_t^\bfx) \circ \dd W(t), \qquad \bfX_0^\bfx = \bfx, \qquad ( t\geq 0, \, \bfx \in \mathbb{R}^d ).
\end{equation*}
Our ansatz for constructing an explicit solution to the stochastic Stokes equations force by transport noise is (here $\bfY$ is the inverse flow of $\bfX$):
\begin{equation*}
\bfu_{\mathrm{tra}}(t,\bfx) := \bfu(t, \bfY_t^x) \quad \text{ and }  \quad p_{\mathrm{tra}}(t,\bfx) := p(t, \bfY_t^x);
\end{equation*}
in other words, at time $t$ the point~$\bfx$ is first moved by the inverse flow to the random point~$\bfY_t^x$, and then the deterministic velocity and pressure are evaluated at this random point. For expressing the evolution dynamics of our ansatz, we need the Jacobian matrix~$\bfPsi$ and the Laplacian~$\bfpsi$ of the flow~$\bfX$ evaluated at the position of the inverse flow~$\bfY$, \emph{i.e.}, for each component $i,j \in \{1,\ldots, d\}$, they are given by
\begin{equation*}
\Psi_{i,j}(t,\bfx) := \partial_{i} \bfX_{t,j}^{\bfz} \big|_{\bfz = \bfY_t^{\bfx}} \qquad \text{ and } \qquad \psi_i(t,\bfx) := \Delta \bfX_{t,i}^{\bfz} \big|_{\bfz = \bfY_t^{\bfx}},
\end{equation*}
respectively. Moreover, we define the modified force $\bff_{\mathrm{tra}}(t,\bfx) := \bff_\mu(t,\bfY_t^x)$.

An application of It\^o's formula and a lengthy computation, which we omit for the sake of readability, reveal that~$\bfu_\mathrm{tra}$ and $p_\mathrm{tra}$ solve
\begin{equation*}
\dd \bfu_{\mathrm{tra}}^{\ell}  = \left(\mu \left[ \Psi_{i,j} \Psi_{i,k} \partial_j \partial_k \bfu_{\mathrm{tra}}^{\ell} + \psi_i \partial_i\bfu_{\mathrm{tra}}^{\ell}   \right]    - \Psi_{\ell,i} \partial_i p_{\mathrm{tra}}   + \bff_{\mathrm{tra}}^{\ell} \right) \dd t + \sigma_i \partial_i \bfu_{\mathrm{tra}}^{\ell} \circ \dd W(t),
\end{equation*}
for each component $\ell \in \{1,\ldots,d\}$; and 
\begin{equation*}
\Div \bfu_{\mathrm{tra}}(t,\bfx) =  \partial_i \bfu^{j}(t,\bfY_t^{\bfx}) \partial_{j} \bfY_{t,i}^{\bfx},
\end{equation*}
where we use the Einstein summation convention (\emph{i.e.}, we sum over repeated indices) in both equations.

The dependence of the stochastic flow~$\bfX$ with respect to the initial condition not only influences the weights in front of the Hessian tensor and gradient of velocity and pressure, respectively, but additionally determines whether and at what intensity advection takes place. In particular, if $\bfX$ depends (affine-)linearly on the initial condition, then the advective term vanishes.

Next, we will give two examples in which we compute the stochastic flow~$\bfX$ and its inverse~$\bfY$ explicitly: for the first example, we choose a \textit{constant} field; and for the second one, we choose the \textit{identity} field. In both cases, the dependence of the flow on the initial condition will be affine-linear. Consequently, the gradient and Hessian of the flows will be constant and zero, respectively.

\textit{Constant Transport Noise.} Let $\bfGamma \in \mathbb{R}^d$ and $\bfsigma(\bfx) \equiv \bfGamma$. Then the following identities hold:
\begin{equation*}
\bfX_t^{\bfx}(\omega) =  \bfx - \bfGamma W_t(\omega), \qquad \bfY_t^{\bfx}(\omega)  = \bfx + \bfGamma W_t(\omega), \qquad \nabla \bfX_t^{\bfx}(\omega) = \identity_{d} = \nabla \bfY_t^{\bfx}(\omega),
\end{equation*}
where $\identity_{d}$ denotes the $d$-dimensional unit matrix. Using these identities allows us to derive the following example:

\begin{framed}
\begin{example}[constant transport noise] \label{ex:transportConstant} Consider the data:

 $\bullet$ domain $D = (0,1)^2  \quad\bullet$ time $T=1 \quad \bullet$ viscosity $\mu =1 \quad \bullet$ 1D-Brownian motion $W$
 
 $\bullet$ noise coefficient $\bfsigma(x,y) \equiv \bfGamma =  \begin{pmatrix}
 \Gamma_1 \\ \Gamma_2
 \end{pmatrix} = \begin{pmatrix}
 1 \\ 1
 \end{pmatrix}$
 
 $\bullet$ body force
\begin{equation*}
\bff(\omega,t,x,y) =  - \mu \begin{pmatrix}
 f(\hat{x},\hat{y}) \\
 -f(\hat{y},\hat{x})
\end{pmatrix}  + 2t \begin{pmatrix}
\hat{x} \\ \hat{y}
\end{pmatrix} 
\end{equation*} 

where $\hat{x} = x + \Gamma_1 W_t(\omega)$, $\hat{y} = y + \Gamma_2 W_t(\omega)$, and
\begin{equation}\nonumber
f(x,y) = 4x^2y(2y^2 - 3y + 1) + 16xy(x - 1)(2y^2 - 3y + 1)   + 4y(x - 1)^2(2y^2 - 3y + 1) + 4x^2(6y - 3)(x - 1)^2
\end{equation} 

$\bullet$ boundary condition~$\bfu_{\tt BC}(\omega,t,x,y) =  \begin{pmatrix}
2\hat{x}^2\left(1-\hat{x}\right)^2 \hat{y}(\hat{y}-1)(2\hat{y}-1) \\ -2\hat{y}^2\left(1-\hat{y}\right)^2\hat{x}(\hat{x}-1)(2\hat{x}-1)
\end{pmatrix}  $

$\bullet$ initial condition~$\bfu_0(x,y) =  \begin{pmatrix}
2x^2\left(1-x\right)^2 y(y-1)(2y-1) \\ -2y^2\left(1-y\right)^2x(x-1)(2x-1)
\end{pmatrix} $\\
In this case, the unique solution to~\eqref{bench-eq:transport} is given by:
\begin{equation}\nonumber
\bfu(\omega,t,x,y) = \begin{pmatrix}
2\hat{x}^2\left(1-\hat{x}\right)^2 \hat{y}(\hat{y}-1)(2\hat{y}-1) \\ -2\hat{y}^2\left(1-\hat{y}\right)^2\hat{x}(\hat{x}-1)(2\hat{x}-1)
\end{pmatrix} 
\end{equation}
and 
\begin{equation}\nonumber
 p(\omega,t,x,y) = \left(\hat{x}^2 + \hat{y}^2 - \int_{D} (\hat{x}^2 + \hat{y}^2) \,\dd \bfx  \right) t,
\end{equation}
where 
\begin{equation*}
\int_{D} (\hat{x}^2 + \hat{y}^2) \,\dd \bfx  = \frac{1}{3} \left( \big[\big(1+ \Gamma_1 W_t(\omega) \big)^3 + \big(1+ \Gamma_2 W_t(\omega) \big)^3 \big] - \big[ \big(\Gamma_1 W_t(\omega) \big)^3 + \big(\Gamma_2 W_t(\omega) \big)^3 \big] \right).
\end{equation*}
\end{example}
\end{framed}
In contrast to Example~\ref{ex:detStokes} in which the velocity vanishes at the boundary of the unit square, the velocity above vanishes at the boundary of a randomly translated unit square. In fact, this boundary is completely characterised by the stochastic flow~$\bfX$ since it is given as the boundary of $\bfX_t^{D}(\omega) = D - \bfGamma W_t(\omega)$, \emph{i.e.}, the unit square is translated in the direction~$-\bfGamma$ with a randomly scaled magnitude. However, since we do not seek a solution on the randomly translated domain but the original one, non-trivial boundary conditions have to be imposed.

\textit{Linear transport noise.} Let $\lambda \in \mathbb{R}$ and $\bfsigma(\bfx) \equiv \lambda \bfx$. Then the following identities hold:
\begin{eqnarray*}
\bfX_t^{\bfx}(\omega) &= & \exp( - \lambda W_t(\omega) )  \bfx , \qquad \quad \,\,\, \bfY_t^{\bfx}(\omega)  =  \exp( \lambda W_t(\omega) )  \bfx, \\
 \nabla \bfX_t^{\bfx}(\omega) &=  &\exp( - \lambda W_t(\omega) )  \identity_{d}, \qquad \nabla \bfY_t^{\bfx}(\omega) = \exp( \lambda W_t(\omega) )  \identity_{d}.
\end{eqnarray*}
From them, we derive the next example.

\begin{framed}
\begin{example}[Linear Transport Noise] \label{ex:transportLinear} Consider the data:

 $\bullet$ domain $D = (0,1)^2  \quad\bullet$ time $T=1 \quad \bullet$ viscosity $\mu =1 \quad \bullet$ 1D-Brownian motion $W$
 
 $\bullet$ noise coefficient $\bfsigma(x,y) \equiv \lambda \begin{pmatrix}
 x \\ y
 \end{pmatrix}$ with $\lambda = 1$
 
 $\bullet$ body force
\begin{equation*}
\bff(\omega,t,x,y) =  - \mu \begin{pmatrix}
 f(\hat{x},\hat{y}) \\
 -f(\hat{y},\hat{x})
\end{pmatrix} e^{2 \lambda W_t(\omega)}  + 2t \begin{pmatrix}
\hat{x} \\ \hat{y}
\end{pmatrix} e^{\lambda W_t(\omega)}  
\end{equation*} 

where $\hat{x} = e^{\lambda W_t(\omega)}x$, $\hat{y} =e^{\lambda W_t(\omega)} y$, and
\begin{equation}\nonumber
f(x,y) = 4x^2y(2y^2 - 3y + 1) + 16xy(x - 1)(2y^2 - 3y + 1)   + 4y(x - 1)^2(2y^2 - 3y + 1) + 4x^2(6y - 3)(x - 1)^2
\end{equation} 

$\bullet$ boundary condition~$\bfu_{\tt BC}(\omega,t,x,y) =  \begin{pmatrix}
2\hat{x}^2\left(1-\hat{x}\right)^2 \hat{y}(\hat{y}-1)(2\hat{y}-1) \\ -2\hat{y}^2\left(1-\hat{y}\right)^2\hat{x}(\hat{x}-1)(2\hat{x}-1)
\end{pmatrix}  $

$\bullet$ initial condition~$\bfu_0(x,y) =  \begin{pmatrix}
2x^2\left(1-x\right)^2 y(y-1)(2y-1) \\ -2y^2\left(1-y\right)^2x(x-1)(2x-1)
\end{pmatrix} $\\
In this case, the unique solution to~\eqref{bench-eq:transport} is given by:
\begin{equation}\nonumber
\bfu(\omega,t,x,y) = \begin{pmatrix}
2\hat{x}^2\left(1-\hat{x}\right)^2 \hat{y}(\hat{y}-1)(2\hat{y}-1) \\ -2\hat{y}^2\left(1-\hat{y}\right)^2\hat{x}(\hat{x}-1)(2\hat{x}-1)
\end{pmatrix} 
\end{equation}
and 
\begin{equation}\nonumber
 p(\omega,t,x,y) = \left(\hat{x}^2 + \hat{y}^2 - \frac{2}{3}e^{2 \lambda W_t(\omega)}  \right) t .
\end{equation}
\end{example}
\end{framed}

While constant transport noise leads to a randomly scaled translation of the domain, linear transport noise leads to randomly scaled dilatation focused at the origin. 

One might ask whether the scaled identity map can be replaced by a more general linear transformation. However, if one allows a different scaling in each co-ordinate, then this would already lead to a violation of the incompressibility constraint. If one is willing to consider compressible vector fields, then the procedure above can be used to generate more examples.

The flows imposed by the transport fields, that were used in the examples above, change the underlying domain; and, consequently, they force us to impose non-trivial boundary conditions. Stochastic flows that do not change the geometry typically need to vanish at the boundary (or at least the transport field~$\bfsigma$ needs to belong to the tangent space at the boundary). Intuitively, this corresponds to particles, that originally start inside the domain, cannot leave the domain by crossing its boundary. Thus, the domain remains invariant. A rigorous derivation of this fact can be found, \emph{e.g.}, in~\cite[Section~3]{MR1124837}. However, in this situation it is impossible to find an exact expression of the stochastic flow.

\subsection{The non-direct approach: prescribing data}
The non-direct approach does not prescribe an exact solution; instead, the data are chosen such that they meet theoretical requirements and/or they are given by an application. On the one hand, this removes the access to the exact error (as the analytic solution is no longer known) and it possibly prohibits a direct application of the theory (as the data might be too irregular); but on the other hand, it connects academic examples and real-world applications. These applications demand a thorough testing of the numerical algorithms to guarantee accuracy of the quantitative predictions.

Next, we present a theoretically motivated example, the construction of particular noises and a real-world application.  

\subsubsection{Stochastic Stokes equations with non-linear multiplicative noise} The following example (or slightly modified versions of it) has been used in~\cite{FPV1,FV2022,LMS1}. The most important features are: (i) a non-linear dependence on the solution, and (ii) a variable scale of spatial regularity. In this way, it is possible to explore how spatial regularity of the noise moderates the convergence rate of algorithms. A summary of the data is given in the following example:

\begin{framed}
\begin{example}[non-linear multiplicative noise] \label{ex:nonlinearMultiplicative} Consider the data:

$\bullet$ domain $D = (0,1)^2  \qquad\bullet$ time $T=1 \qquad \bullet$ viscosity $\mu =1 $\\
The noise is given in terms of a cylindrical process: 
\begin{equation*} 
\bfsigma(\bfu) \dd W := \sum_{\ell_1, \ell_2 = 0}^\infty \sqrt{\mu_{\ell_1,\ell_2}^r}\begin{pmatrix}
 \sqrt{u_1^2 + 1} \\
 \sqrt{u_2^2 + 1}
\end{pmatrix} \phi_{\ell_1,\ell_2} \dd \beta_{\ell_1,\ell_2},
\end{equation*}
where
\begin{itemize}
\item (spatial colouring) $r \in (0,\infty)$ determines the spatial regularity and the coefficients are given by 
\begin{equation*}
\mu_{\ell_1,\ell_2} := \begin{cases}
0, & \ell_1 = \ell_2 = 0, \\
(\ell_1^2 + \ell_2^2)^{-1}, & \text{else};
\end{cases}
\end{equation*}
\item (spatial basis) the basis functions are given by $\phi_{\ell_1,\ell_2}(\bfx) = \cos(\pi \ell_1 x_1) \cos(\pi \ell_2 x_2)$;
\item (driving processes) each process $\beta_{\ell_1,\ell_2}$ is a 1D-Brownian motion independent of the others.
\end{itemize}
The task is to numerically approximate the following equations: 
\begin{eqnarray}\nonumber
{\rm d}{\bf u} - \left( \mu \Delta {\bf u} - \nabla p \right) \dd t &=&   {\bf1 }{\rm d}t + \bfsigma(\bfu) \dd W 
\qquad \quad \,\,  \mbox{on }  \Omega \times (0,T) \times D\,,\\ \nonumber
{\rm div} \, {\bf u} &=& 0 \quad \qquad \qquad  \,  \qquad \qquad \mbox{on }\Omega \times  (0,T) \times D\,, \\
\nonumber
{\bf u} &=& {\bf0} \quad \qquad \qquad \qquad \, \qquad \mbox{on } \Omega \times (0,T) \times \partial D\,, \\
\nonumber
{\bf u}(0) &=& {\bf0} \qquad \qquad \qquad \qquad \quad \,  \mbox{on } \Omega \times D\,.
\end{eqnarray}
\end{example}
\vspace*{-0.8em}
\end{framed}
In contrast to previous examples in which the noise was defined by a single Brownian motion, the noise here is given in terms of infinitely many Brownian motions. This causes an additional difficulty since one has to truncate the noise in order to implement an approximation. Typically the truncation index is chosen in accordance with the spatial resolution, as high frequencies can only be captured on sufficiently refined spatial scales. 

\subsubsection{Construction of particular noises on various scales} To explore the dependence of the stochastic Stokes equations with respect to the particular structure of the noise (\emph{e.g.} additive, multiplicative, or transport) and its spatial scale, we introduce a scale-dependent construction of a family of noises: First, we prescribe the shape of the spatial action of the noise by choosing a vector field. Then, depending on the desired scale, this vector field is copied and scaled. Afterwards, each copy is multiplied by a noise. Finally, superposing these copies disjointly, we obtain a vector field that repeats the shape function multiple times at different locations; for an illustration of this, we refer to Figure~\ref{fig:vortexes} where the following example is depicted.

\begin{framed}
\begin{example}[vortexes on equi-distant lattice]\label{ex:vortexes} Consider the data: $\bullet$ domain $D = (0,1)^2$

$\bullet$ scale $\mathrm{scl}~ (= 2)\in \mathbb{N}_0 \quad\bullet$
length $\lambda = \tfrac{1}{\mathrm{scl}+1} \quad\bullet$ centers $\mathrm{C} = \{\tfrac{2j+1}{2(\mathrm{scl}+1)}: j= 0,1, \ldots, \mathrm{scl}\}$

The shape function is given by a counter-clockwise vortex around the mid-point of~$D$:
\begin{equation}\nonumber
\bfg(x,y) =  100 \begin{pmatrix}
2x^2\left(1-x\right)^2 y(y-1)(2y-1) \\ -2y^2\left(1-y\right)^2x(x-1)(2x-1)
\end{pmatrix}\chi_{D}(x,y),
\end{equation}
where $\chi_{D}$ denotes the indicator function on $D$.

For $c = (x_0,y_0) \in C^2$, the small-scale vortex around $c$ is given by:
\begin{equation}\nonumber
\bfg_c^{\lambda}(x,y) =  g(\hat{x},\hat{y}), \qquad \text{ where } \qquad \hat{x} = \frac{2(x-x_0) + \lambda}{2\lambda} \quad \text{and} \quad \hat{y} = \frac{  2(y-y_0) + \lambda}{2\lambda}.
\end{equation}

Let $\{\beta_c\}_{c \in C^2}$ be a family of independent 1D-Brownian Motions. \\
We define the following noises for which $\bfv$ denotes a vector field:
\begin{itemize}
\item (additive) $\dd \bfZ_{\bfv} := \sum_{c\in C^2} \bfg_c^{\lambda} \dd \beta_c$;
\item (multiplicative, non-linear) $\dd \bfZ_{\bfv} := \sum_{c\in C^2} \bfg_c^{\lambda} \abs{\bfv} \circ \dd \beta_c$;
\item (transport) $\dd \bfZ_{\bfv} := \sum_{c\in C^2} (\bfg_c^{\lambda} \cdot \nabla) \bfv \circ \dd \beta_c$.
\end{itemize}
\end{example}
\end{framed}

\begin{figure*}[t!]
    \centering
    \begin{subfigure}[t]{0.33\textwidth}
        \centering
        \includegraphics[width=1.0\textwidth]{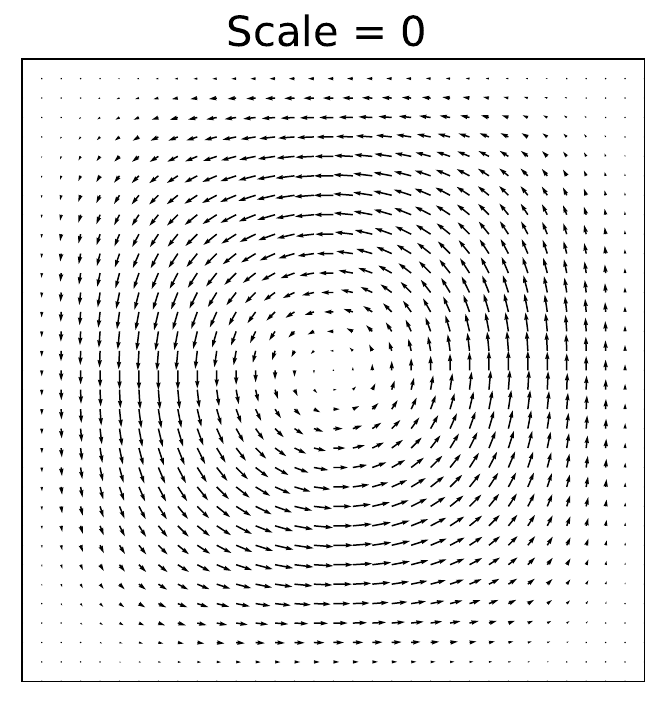}
    \end{subfigure}%
    \begin{subfigure}[t]{0.33\textwidth}
        \centering
        \includegraphics[width=1.0\textwidth]{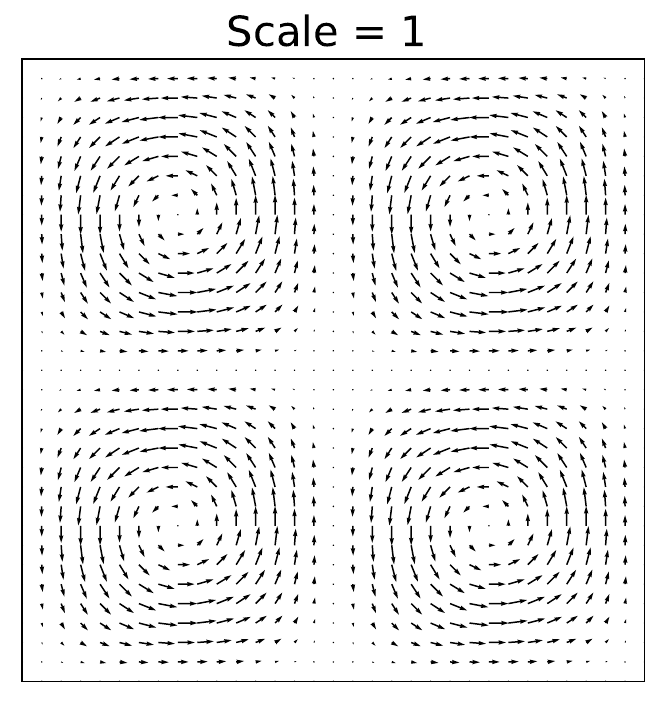}
    \end{subfigure}%
    \begin{subfigure}[t]{0.33\textwidth}
        \centering
        \includegraphics[width=1.0\textwidth]{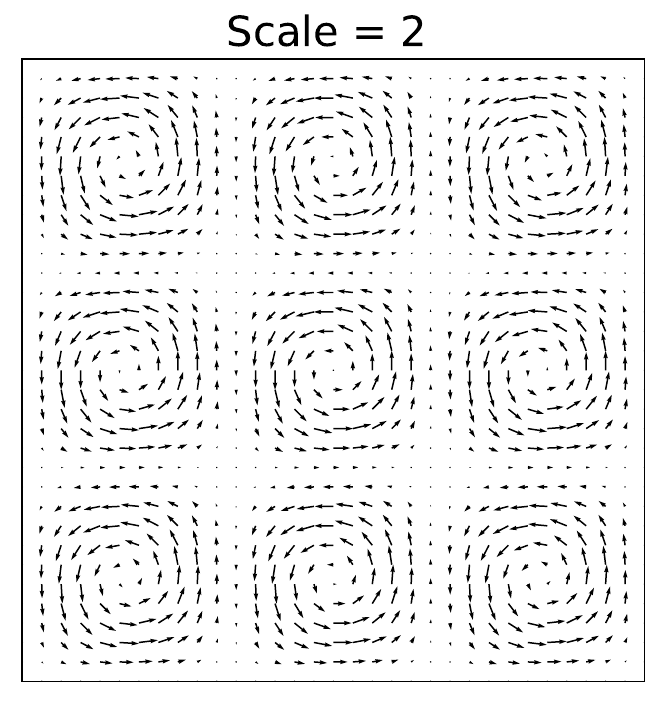}
    \end{subfigure}
    \caption{Illustration of the vortexes used for the construction of particular noises in Example~\ref{ex:vortexes}.} 
    \label{fig:vortexes}
\end{figure*}

\subsubsection{The lid-driven cavity problem} \label{sec:lid-driven} 
It is especially important to test numerical algorithms in realistic situations to eventually enable their use in real-world applications. However, currently most algorithms, developed for the stochastic Stokes and Navier--Stokes equations, are not validated in meaningful benchmarks. One of these benchmark problems is the so-called \textit{lid-driven cavity} experiment (\emph{cf.} Section~\ref{sec:intro} for various illustrations). The following introduction is taken from~\cite{2024arXiv241214316D}.

The lid-driven cavity experiment considers a resting fluid in a box (\emph{e.g.}, a lake), which is influenced by non-trivial boundary conditions acting solely at the lid of the container (\emph{e.g.}, wind blowing over the surface of the lake). The boundary conditions push the fluid at a constant rate into one direction (\emph{e.g.}, the wind blows in a fixed direction and at constant speed). Initially, the fluid starts at rest. Eventually, induced by the boundary conditions, the fluid displays versatile dynamics in the whole container; see, \emph{e.g.},~\cite{Zhu2020}.

Resolving the full range of dynamics in numerical simulations is a challenging task. Motivated by this, many authors used the lid-driven cavity experiment as a benchmark problem for numerical algorithms; see, \emph{e.g.},~\cite{Botti2019,Erturk2005,Zhu2020,kuhlmann2019lid} and the references therein. Unfortunately, numerical simulations for stochastic models of the lid-driven cavity experiment are still largely missing. Therefore, we want to encourage numerical analysts to also validate their algorithms using this benchmark problem. 

The lid-driven cavity experiment focuses on the challenges that are created by considering non-trivial boundary conditions. In the following example, we summarize the data needed for conducting the experiment: 

\begin{framed}
\begin{example}[lid-driven cavity] \label{ex:lid-driven} Consider the data:

 $\bullet$ domain $D = (0,1)^2  \hspace*{0.6em}\bullet$ time $T=20 \hspace*{0.6em} \bullet$ viscosity $\mu =\tfrac{1}{500} \hspace*{0.6em} \bullet$ noise $\bfZ_{\bfu}$ (from Example~\ref{ex:vortexes})
 
 $\bullet$ boundary condition~$\bfu_{\tt BC}(\omega,t,x,y) = \begin{pmatrix}
1 \\
0
\end{pmatrix} \chi_{\{y=1\}}(x,y) $\\
The task is to numerically approximate the following equations: 
\begin{eqnarray}\nonumber
{\rm d}{\bf u} - \left( \mu \Delta {\bf u} - \nabla p \right) \dd t &=&     \dd \bfZ_{\bfu} 
\qquad \qquad   \mbox{on }  \Omega \times (0,T) \times D\,,\\ \nonumber
{\rm div} \, {\bf u} &=& 0 \quad  \,\,  \qquad \qquad \mbox{on }\Omega \times  (0,T) \times D\,, \\
\nonumber
{\bf u} &=& \bfu_{\tt BC} \qquad \,\,\, \qquad \mbox{on } \Omega \times (0,T) \times \partial D\,, \\
\nonumber
{\bf u}(0) &=& {\bf0} \qquad \qquad \quad \,\,  \mbox{on } \Omega \times D\,.
\end{eqnarray}
\end{example}
\end{framed}

\subsection{Summary}
Numerical algorithms must be tested to ensure their accurate performance in various situations. These tests focus on different aspects, such as supporting theory and exploring theoretical extensions or their failure by systematically violating assumptions. For their construction, two methods are available:

The \emph{direct} method allows us to prescribe the solution exactly, giving us access to the exact approximation error. This method has been widely used for deterministic equations to verify the sharpness of theoretical results. However, at the moment it has not been used for stochastic equations. We have shown that this method is also available in the stochastic case as long as the noise is simple (\emph{i.e.}, solution-independent or linear) and we are free to choose the initial condition, boundary conditions, and, generally random, body force. Due to the success of the direct method for deterministic equation, we emphasise that it should be utilized for stochastic equations more often.

The \emph{non-direct} method prescribes the data rather than the solution. Consequently, the solution is unknown and the exact approximation error is itself a non-computable quantity. Therefore, the non-direct method is infeasible for supporting theoretically derived convergence rates of the approximation error. Instead, it focuses on bridging the gap between theory and applications. Here, it is important to conduct simulations in mathematically demanding but realistic situations. Attesting the algorithms success also in these challenging situations enables practitioners to use the quantitative predictions for particular applications.


\section{Conclusions} \label{sec:conclusion}
The aim of this work is to survey available numerical schemes and related convergence analyses for the incompressible stochastic Stokes and, and the 2D stochastic Navier--Stokes equations. We made an effort to identify key problems
in this context, and to motivate ideas. The presentation is kept at an informal level to serve
this purpose, and we provide an extensive list of references for a precision of technical arguments needed for it. Most existing
results in this area hold for
prototypical academic flow problems ({\em e.g.}, a torus, and supplemented by periodic boundary conditions)
to get over related analytical complications, and we made clear that further progress in this area is needed, which should include
also to evaluate the performance of numerical methods in more realistic settings of engineering relevancy. To enhance
comparability of numerical schemes in terms of computational complexity in future contributions, we propose an
extended list of benchmark problems consisting of both, academic and applied fluid flow scenarios. And the aim
with it is to test new methods also even if their theoretical foundation is restricted by technical assumptions.

\section*{Funding}
DB has been funded by Grant BR 4302/3-1 (525608987) of the German Research Foundation (DFG) within the framework of the priority research program SPP 2410 and by Grant BR 4302/5-1 (543675748) of the German Research Foundation (DFG).

JW was supported by the Australian Government through the Australian Research Council’s Discovery Projects funding scheme (grant number DP220100937).

%

\printbibliography 

\end{document}